\documentclass[reqno,12pt]{amsart}
\usepackage{a4wide}
\usepackage{amssymb}
\usepackage{amsmath}
\usepackage{amsthm}
\usepackage[usenames]{color}

\newcommand{\R}{\mathbb{R}}

\usepackage{cite}
\allowdisplaybreaks

\newtheorem{theorem}{Theorem}[section]

\newtheorem{lemma}[theorem]{Lemma}
\newtheorem{proposition}[theorem]{Proposition}
\newtheorem{remark}[theorem]{Remark}

\numberwithin{equation}{section}

\usepackage[colorlinks]{hyperref}
\usepackage{marginnote}

\newcommand{\dis}{\displaystyle}

\usepackage{color}

\makeatletter
\@namedef{subjclassname@2020}{%
  \textup{2020} Mathematics Subject Classification}
\makeatother

\begin{document}

\title[Compressible fluid limit of the Landau equation]{Compressible fluid limit for smooth solutions to the Landau equation}

\author[R.-J. Duan]{Renjun Duan}
\address[R.-J. Duan]{Department of Mathematics, The Chinese University of Hong Kong, Shatin, Hong Kong,
        People's Republic of China}
\email{rjduan@math.cuhk.edu.hk}

\author[D.-C. Yang]{Dongcheng Yang}
\address[D.-C. Yang]{Department of Mathematics, The Chinese University of Hong Kong, Shatin, Hong Kong,
        People's Republic of China}
\email{dcyang@math.cuhk.edu.hk}

\author[H.-J. Yu]{Hongjun Yu}
\address[H.-J. Yu]{School of Mathematical Sciences,
         South China Normal University, Guangzhou 510631, People's Republic of China}
\email{yuhj2002@sina.com}

\begin{abstract}
Although the compressible fluid limit of the Boltzmann equation with cutoff has been well investigated in \cite{Caflisch} and \cite{Guo-Jang-Jiang-2010}, it still remains largely open to obtain analogous results in case of the angular non-cutoff or even in the grazing limit which gives the Landau equation, essentially due to the velocity diffusion effect of collision operator such that $L^\infty$ estimates are hard to obtain without using Sobolev embeddings. In the paper, we are concerned with the compressible Euler and acoustic limits of the Landau equation for Coulomb potentials in the whole space. Specifically, over any finite time interval where the full compressible Euler system admits a smooth solution around constant states, we construct a unique solution in a high-order weighted Sobolev space for the Landau equation with suitable initial data and also show the uniform estimates independent of the small Knudsen number $\varepsilon>0$, yielding the $O(\varepsilon)$ convergence of the Landau solution to the local Maxwellian whose fluid quantities are the given Euler solution. Moreover, the acoustic limit for smooth solutions to the Landau equation in an optimal scaling is also established. For the proof, by using the macro-micro decomposition around local Maxwellians together with techniques for viscous compressible fluid and properties of Burnett functions, we design an $\varepsilon$-dependent energy functional  to capture the dissipation in the compressible fluid limit with feature that only the highest order derivatives are most singular.  
\end{abstract}


\subjclass[2020]{35Q84, 35Q20; 35B40, 35B45}

\keywords{Landau equation, smooth solutions, compressible Euler limit, acoustic limit, macro-micro decomposition, Burnett functions, energy estimates}



\maketitle

\tableofcontents
\thispagestyle{empty}

\section{Introduction}
In plasma physics the spatially inhomogeneous Landau equation is a fundamental mathematical model at the kinetic level. It is used to describe the time evolution of the unknown density distribution function $F=F
(t,x,v)\geq0$ of particles in plasma with space position $x=(x_{1},x_{2},x_{3})\in\mathbb{R}^{3}$ and velocity $v=(v_{1},v_{2},v_{3})\in\mathbb{R}^{3}$
at time $t\geq0$, reading as 
\begin{equation}
\label{1.1}
\partial_{t}F+v\cdot\nabla_{x}F=\frac{1}{\varepsilon}Q(F,F).
\end{equation}
Here,  the dimensionless parameter $\varepsilon>0$ is small and reciprocal to logarithm of the Debye shielding length, and it plays the same role as the Knudsen number in the Boltzmann theory, cf.~\cite{Cercignani, Sone}. We have omitted the explicit dependence of the solution $F(t,x,v)$ on $\varepsilon$ for brevity whenever there is no confusion. The Landau collision operator $Q(\cdot,\cdot)$ is a bilinear integro-differential operator acting only on velocity variables of the form
\begin{equation*}
Q(F_{1},F_{2})(v)=\nabla_{v}\cdot\int_{\mathbb{R}^{3}}\Phi(v-v_{*})\left\{F_{1}(v_{*})\nabla_{v}F_{2}(v)-\nabla_{v_{*}}F_{1}(v_{*})F_{2}(v)\right\}\,dv_{*},
\end{equation*}
where for the Landau collision kernel $\Phi(\xi)=[\Phi_{ij}(\xi)]$ with $\xi=v-v_\ast$ (cf. \cite{Guo-2002,H}), we consider only the case of physically most realistic Coulomb interactions throughout the paper, namely,
\begin{equation}
\label{1.3}
\Phi_{ij}(\xi)=\frac{1}{|\xi|}\big(\delta_{ij}-\frac{\xi_i\xi_j}{|\xi|^{2}}\big),\quad 1\leq i,j\leq 3,
\end{equation}
with $\delta_{ij}$ being the Kronecker delta.

On the other hand, the hydrodynamic  description for the motion of plasmas at the fluid level is also given by the compressible Euler system
\begin{equation}
\label{1.4}
  \left\{
\begin{array}{rl}
\partial_{t}\bar{\rho}+\nabla_{x}\cdot(\bar{\rho}\bar{u})=0,
\\
\partial_{t}(\bar{\rho}\bar{u})+\nabla_{x}\cdot(\bar{\rho}\bar{u}\otimes\bar{u})
+\nabla_{x}\bar{p}=0,
\\
\partial_{t}\big[\bar{\rho}(\bar{e}+\frac{1}{2}|\bar{u}|^{2})\big]+
\nabla_{x}\cdot\big[\bar{\rho}\bar{u}(\bar{e}+\frac{1}{2}|\bar{u}|^{2})\big]
+\nabla_{x}\cdot(\bar{p}\bar{u})=0,
\end{array} \right.
\end{equation}
with the equation of state  $\bar{p}=R\bar{\rho}\bar{\theta}$.
These are the local conservation laws of
mass, momentum, and energy. Here $\bar{e}=\bar{e}(t,x)>0$ is the internal energy which is related to the temperature $\bar{\theta}=\bar{\theta}(t,x)$ by $\bar{e}=\frac{3}{2}R\bar{\theta}=\bar{\theta}$ with the gas constant $R=\frac{2}{3}$ taken for convenience,
and $\bar{\rho}=\bar{\rho}(t,x)$, $\bar{u}=\bar{u}(t,x)$ are the mass density and bulk velocity, respectively.

It is well known that the compressible Euler system \eqref{1.4} can be formally derived from the Landau equation \eqref{1.1} through the 
first order approximation to the famous Hilbert expansion  when the Knudsen number $\varepsilon>0$ is close to zero, which is similar to the case of the Boltzmann equation, see  \cite{Golse,Grad,SR}. The rigorous mathematical justification of establishing such limit in a general setting, particularly for general initial data in three dimensions, is still a challenging subject in kinetic theory, although it has been extensively studied for the Boltzmann equation with cutoff, see for instance \cite{Caflisch,Guo-Jang-Jiang,Guo-Jang-Jiang-2010,Lachowicz,Nishida,Ukai-Asano} and the references therein.

In what follows we review some specific results on the compressible Euler limit of the Boltzmann equation with cutoff so as to make a comparison with the Landau case later. Note that there are few analogous results on the issue for the non-cutoff Boltzmann equation. First of all,  based on the truncated Hilbert expansion, Caflisch \cite{Caflisch}  showed that, if the compressible Euler system has a smooth solution which exists up to a finite time, then there exist corresponding
solutions to the Boltzmann equation with a zero initial condition for the remainder term
in the same time interval such that the Boltzmann solution converges to a smooth local Maxwellian whose fluid dynamical parameters  satisfy the compressible Euler system as $\varepsilon\to 0$.
Due to such zero initial condition, the obtained Boltzmann solution may be negative. The result was extended later by Lachowicz \cite{Lachowicz} to the case allowing for the initial layer. Following the Caflisch's strategy together with a new
$L^{2}-L^{\infty}$ approach developed in \cite{Guo-2008}, Guo-Jang-Jiang \cite{Guo-Jang-Jiang,Guo-Jang-Jiang-2010} removed the restriction with zero initial condition so that the positivity of the Boltzmann solution can be  guaranteed. Recently, Guo-Huang-Wang \cite{Guo-Huang-Wang}  and Jiang-Luo-Tang \cite{Jiang-Luo,Jiang-Luo-1} made significant progress on the topic by making use of the $L^{2}-L^{\infty}$ interplay technique to generalize the results of \cite{Caflisch,Guo-Jang-Jiang,Guo-Jang-Jiang-2010} to the half-space problem with different types of boundary conditions including specular reflection, diffuse reflection and even the mixed Maxwell reflection. We also mention another recent progress Guo-Xiao \cite{GX21} for an application of the $L^{2}-L^{\infty}$ approach to the global Hilbert expansion to the relativistic Vlasov-Maxwell-Boltzmann system. 

The compressible Euler limit of the Boltzmann equation is also studied in other contexts. Based on the abstract Cauchy-Kovalevskaya  theorem and the spectral analysis of the semigroup generated by the linearized Boltzmann equation,
Nishida \cite{Nishida} constructed the local solution of the cutoff Boltzmann equation in the analytic framework and proved that the solution tends to the solution of Euler system as $\varepsilon\to 0$.
Later, Ukai-Asano \cite{Ukai-Asano} improved the Nishida's result by using the classical contraction mapping principle on a
space with a time-dependent analytic norm and also considered the case with the initial layer.
Moreover, in the special setting of one space dimension, regarding the hydrodynamic limit of the Boltzmann equation with cutoff to the compressible Euler system which admits specific solutions of basic wave patterns, such as rarefaction waves, contact discontinuities and shock waves, there have been extensive studies by \cite{Huang-2013,Huang-2010,Wang-Wang-2021,Xin-Zeng,Yu-2005} and the references therein.

As is well known, the linearization of  compressible Euler equations around constant states gives the acoustic system, see \eqref{1.12} to be specified later. For inviscid compressible fluids, the acoustic system could be the simplest one describing essentially the wave propagation, which can be formally derived from the Boltzmann equation. Under the Grad's angular cutoff assumption, Bardos-Golse-Levermore \cite{Bardos-2000} proved the convergence in the acoustic limit from the DiPerna-Lions \cite{Diperna-1989} renormalized solutions of Boltzmann equation with a restriction on the size of fluctuations. The restriction was relaxed later by Golse-Levermore \cite{Golse-2002} and Jiang-Levermore-Masmoudi
\cite{Jiang-Masmoudi}, and finally removed by 
Guo-Jang-Jiang \cite{Guo-Jang-Jiang-2010}
via the Hilbert expansion with the help of $L^{2}-L^{\infty}$ interplay approach as mentioned before. Concerning the limit to other fluid equations from the Boltzmann equation with cutoff, see \cite{Bardos,Bardos-1991,Duan-Liu,EGKM,EGM,Golse-2004,Guo-CPAM,Guo-Jang,JK,Jiang-Masmoudi-1,Liu-2014,Masmoudi-2003} and the references cited therein. We also mention a recent work \cite{Rachid} on the incompressible fluid limit from the Landau equation.

However, despite a lot of great progress mentioned above, the hydrodynamic limit to the compressible Euler or acoustic system  for the Landau equation has remained largely open and the similar situation  occurs to the Boltzmann equation in the non-cutoff case. The main reason is that those long-range collision operators expose the velocity diffusion property so that 
the strategy in \cite{Caflisch,Guo-Jang-Jiang-2010,Nishida}  cannot be directly adopted. In particular, a trouble arises from the action of the transport operator on local Maxwellians, inducing large velocity growth in $L^2$ framework, cf.~\cite{Guo-Jang-Jiang-2010}. To overcome it, a robust idea is to combine  $L^2$ estimates together with the velocity-weighted $L^\infty$ estimates which strongly reply on the Grad's splitting of the linearized Boltzmann operator with cutoff. For either the non-cutoff Boltzmann or Landau equation, it seems still very hard to obtain any $L^\infty$ estimate uniform in the fluid limit $\varepsilon\to 0$ without using the Sobolev imbedding, although one may observe several recent important results \cite{AMSY20,GHJO,Kim} on the global existence of low-regularity solutions  in $L^\infty$ space around global Maxwellians for $\varepsilon=1$. Therefore one has to develop some new ideas to deal with the problem whenever the direct $L^\infty$ estimates are not available.

In the end we further mention some early and recent results on the global existence and large-time behavior for the Landau equation \eqref{1.1} with $\varepsilon=1$, for instance, global existence of weak solutions by Lions \cite{Lions} and Villani \cite{Vi98,Vi96}, grazing collision limit of the non-cutoff Boltzmann to  Landau equation by Desvillettes \cite{Des} and Alexandre-Villani \cite{AV04}, spectrum analysis by Degond-Lemou \cite{Degond-Lemou}, global existence of classical solutions near global Maxwellians by Guo \cite{Guo-2002} for the torus and  Hsiao-Yu \cite{Hsiao-Yu} for the whole space, large time behavior of classical solutions near global Maxwellians by Strain-Guo \cite{Strain-Guo-2006,Strain-Guo-2008}
and see also \cite{CaMi,CTW,DLSS}, global classical solutions near vacuum by Luk
\cite{Luk}, regularity of solutions by Golse-Imbert-Mouhot-Vasseur \cite{GIMV} and Henderson-Snelson \cite{HeSn}. See also more recent works
\cite{Duan-Yang-Yu,Duan-Yang-Yu-1,Duan-Yu1} by the authors of this paper for the one-dimensional Landau equation around local Maxwellians with rarefaction wave and contact wave.

In this paper, we will construct a unique solution in a high-order weighted Sobolev space for the Landau equation with suitable initial data over any finite time interval where the full compressible Euler system admits a smooth solution around constant states. We will obtain the uniform estimates independent of the small Knudsen number $\varepsilon>0$, which yields the $O(\varepsilon)$ convergence of the Landau solution to the local Maxwellian whose fluid quantities are the given Euler solution. 
In the meantime, we also will establish the acoustic limit for smooth solutions to the Landau equation in optimal scaling. We believe that the same results hold for the non-cutoff Boltzmann equation even with soft potentials. For the proof, a key point is to use the macro-micro decomposition around local Maxwellians for designing the $\varepsilon$-dependent temporal energy functional and its corresponding energy dissipation functional such that the uniform estimates can be obtained under the smallness assumption. More details will be specified later on.

The rest of this paper is organized as follows. In Section \ref{sec.2}, we present the macro-micro
decomposition for the Landau equation in order to study the compressible fluid limit for smooth solutions to the Landau equation.
In Section \ref{sec.mr}, we state two main results of this paper, namely,  Theorem \ref{thm1.1} and Theorem \ref{thm1.2} for the compressible Euler limit and the acoustic limit, respectively. In the meantime, we also list key points for the strategy of the proof through the paper for conveniences of readers. To proceed the proof of main results, we first prepare in Section \ref{seca.2}  some basic estimates. Section \ref{seca.4} is the main part of the proof for establishing the a priori estimates on both the fluid part and the kinetic part.
The proofs of Theorems \ref{thm1.1} and \ref{thm1.2} are shortly given in Section \ref{sec.3}.
In the appendix Section \ref{seca.5}, we give the details of deriving an identity \eqref{4.10} for completeness.

\medskip
\noindent{\it Notations.} 
Throughout this paper, generic positive constants are denoted  by either $c$ or $C$ varying from line to line and they are independent of other small parameters such as the Knudsen number and the amplitude of Euler solutions.
$\langle \cdot , \cdot \rangle$ denotes the standard $L^{2}$ inner product in $\mathbb{R}_{v}^{3}$
with its corresponding $L^{2}$ norm $|\cdot|_2$. $( \cdot , \cdot )$ denotes the $L^{2}$ inner product in either
$\mathbb{R}^{3}_{x}$ or $\mathbb{R}^{3}_{x}\times \mathbb{R}_{v}^{3}$  with its corresponding $L^{2}$ norm $\|\cdot\|$. We
use the standard notation $H^{k}(\mathbb{R}^{3}_{x})$ to denote
the Sobolev space $W^{k,2}(\mathbb{R}^{3}_{x})$ with its corresponding norm $\|\cdot\|_{H^{k}}$,
and also use $\|\cdot\|_{L^{p}}$ to denote the $L_{x}^{p}$-norm with $p\in[1,+\infty]$.
The norm of $\nabla^{k}_x f$ means the sum of the norms of functions $\partial^{\alpha}f$ with $|\alpha|=k$.
Let  $\alpha$ and $\beta$ be multi indices $\alpha=(\alpha_{1},\alpha_{2},\alpha_{3})$ and $\beta=(\beta_{1},\beta_{2},\beta_{3})$,
respectively. As in \cite{Guo-2002}, it is convenient to denote
\begin{align*}
\partial_{\beta}^{\alpha}=\partial_{x_{1}}^{\alpha_{1}}\partial_{x_{2}}^{\alpha_{2}}
\partial_{x_{3}}^{\alpha_{3}}
\partial_{v_{1}}^{\beta_{1}}\partial_{v_{2}}^{\beta_{2}}\partial_{v_{3}}^{\beta_{3}}.
\end{align*}
If each component of $\beta$ is not greater than the corresponding one  of
$\overline{\beta}$, we use the standard notation
$\beta\leq\overline{\beta}$. And $\beta<\overline{\beta}$ means that
$\beta\leq\overline{\beta}$ and $|\beta|<|\overline{\beta}|$.
The constant $C^{\bar\beta}_{\beta}$ is the usual  binomial coefficient.

\medskip
\noindent{\bf Added:} After this work was submitted, we noticed a recent preprint \cite{OWX1} which deals with the compressible fluid limit of the relativistic Landau equation via the method of Hilbert expansion. More recently, the same authors posted another preprint \cite{OWX2} for further treating the same problem in the case of the relativistic Vlasov-Maxwell-Landau system. We remark that an essentially different difficulty occurs to the non-relativistic case under consideration of the current work because the action of the classical transport operator on local Maxwellians induces the cubic large velocity growth which is out of control via the time-velocity weighted high-order energy method,  see \cite{Duan-Yang-Yu-1} and references therein.

\section{Macro-micro decomposition}\label{sec.2}

In this section we will present the macro-micro 
decomposition for the Landau equation \eqref{1.1} in order to study the compressible fluid limit for smooth solutions to the Landau equation.

We first recall a basic property of
the Landau collision operator. It is well known that the Landau collision operator admits five collision invariants:
$$
\psi_{0}(v)=1, \quad \psi_{i}(v)=v_{i}~(i=1,2,3),\quad \psi_{4}(v)=\frac{1}{2}|v|^{2},
$$
namely, it holds that
\begin{equation}
\label{2.1}
\int_{\mathbb{R}^{3}}\psi_{i}(v)Q(F,F)\,dv=0,\quad \mbox{for $i=0,1,2,3,4$.}
\end{equation}

The macro-micro decomposition of the solution with respect to the local Maxwellian was initiated by Liu-Yu \cite{Liu-Yu} and developed by Liu-Yang-Yu \cite{Liu-Yang-Yu} in the context of Boltzmann theory, and it can be analogously carried over to the Landau equation. Indeed, associated with a solution
$F(t,x,v)$ to the Landau equation \eqref{1.1}, we introduce five macroscopic (fluid) quantities:
the mass density $\rho(t,x)>0$, momentum $\rho(t,x)u(t,x)$, and
energy density $e(t,x)+\frac 12|u(t,x)|^2$, given as
\begin{align}
\label{2.2}
\left\{
\begin{array}{rl}
\rho(t,x)&\equiv\int_{\mathbb{R}^{3}}\psi_{0}(v)F(t,x,v)\,dv,
\\
\rho(t,x) u_{i}(t,x)&\equiv\int_{\mathbb{R}^{3}}\psi_{i}(v)F(t,x,v)\,dv, \quad \mbox{for $i=1,2,3$,}
\\
\rho(t,x)\big[e(t,x)+\frac{1}{2}|u(t,x)|^{2}\big]&\equiv\int_{\mathbb{R}^{3}}\psi_{4}(v)F(t,x,v)\,dv,
\end{array} \right.
\end{align}
and the corresponding local Maxwellian:
\begin{equation}
	\label{2.3}
	M\equiv M_{[\rho,u,\theta]}(t,x,v):=\frac{\rho(t,x)}{\sqrt{(2\pi R\theta(t,x))^{3}}}\exp\big\{-\frac{|v-u(t,x)|^{2}}{2R\theta(t,x)}\big\}.
\end{equation}
Here $e(t,x)>0$ is the internal energy which is related to the temperature $\theta(t,x)$ by
$e=\frac{3}{2}R\theta=\theta$, and $u(t,x)=(u_{1},u_{2},u_{3})(t,x)$ is the bulk velocity.

Note that the $L^{2}$ inner product in $v\in\mathbb{R}^{3}$ is denoted by
$$
\langle h,g\rangle =\int_{\mathbb{R}^{3}}h(v)g(v)\,d v.
$$
Then the macroscopic space is spanned by the following five orthonormal basis functions:
\begin{align}
	\label{2.4}
	\left\{
	\begin{array}{rl}
		&\chi_{0}(v)=\frac{1}{\sqrt{\rho}}M,
		\\
		&\chi_{i}(v)=\frac{v_{i}-u_{i}}{\sqrt{R\rho\theta}}M, \quad i=1,2,3,
		\\
		&\chi_{4}(v)=\frac{1}{\sqrt{6\rho}}(\frac{|v-u|^{2}}{R\theta}-3)M,
		\\
		&\langle \chi_{i},\frac{\chi_{j}}{M}\rangle=\delta_{ij},
		\quad \mbox{for ~~$i,j=0,1,2,3,4$}.
	\end{array} \right.
\end{align}
In view of  the orthonormal basis above, we define
the macroscopic projection  $P_{0}$ and 
microscopic projection $P_{1}$ as
\begin{equation}
	\label{2.5}
	P_{0}h\equiv\sum_{i=0}^{4}\langle h,\frac{\chi_{i}}{M}\rangle\chi_{i},\quad P_{1}h\equiv h-P_{0}h,
\end{equation}
respectively, where the operators  $P_{0}$ and  $P_{1}$  are  orthogonal projections, that is
$$
P_{0}P_{0}=P_{0},\quad
P_{1}P_{1}=P_{1},\quad
P_{1}P_{0}=P_{0}P_{1}=0.
$$
A function $h(v)$ is called microscopic or non-fluid if
\begin{equation}
	\label{2.6}
	\langle h(v),\psi_{i}(v)\rangle=0, \quad \mbox{for $i=0,1,2,3,4$}.
\end{equation}

Using the notations above,
the solution $F=F(t,x,v)$ of the Landau equation \eqref{1.1} can be decomposed into the
macroscopic (fluid) part, i.e., the local Maxwellian $M=M(t,x,v)$ defined in \eqref{2.3}, and the microscopic (non-fluid) part, i.e. $G=G(t,x,v)$:
\begin{equation}
	\label{2.7}
	F=M+G, \quad P_{0}F=M, \quad P_{1}F=G.
\end{equation}
Then the Landau equation \eqref{1.1} can be rewritten as 
\begin{equation}
	\label{2.8}
	\partial_{t}(M+G)+v\cdot\nabla_{x}(M+G)=\frac{1}{\varepsilon}L_{M}G+\frac{1}{\varepsilon}Q(G,G),
\end{equation}
due to the fact that $Q(M,M)=0$.
Here $L_{M}$ is the
linearized Landau collision operator with respect to the local Maxwellian $M$, given by
\begin{equation}
	\label{2.9}
	L_{M}h:=Q(h,M)+Q(M,h),
\end{equation}
and its null space $\mathcal{N}$ is spanned by the macroscopic variables $\chi_{i}(v)~(i=0,1,2,3,4)$.

Let us now decompose the Landau equation into the  macroscopic system and microscopic system.
Multiplying the Landau equation \eqref{2.8} by the collision invariants $\psi_{i}(v)$ $(i=0,1,2,3,4)$
and integrating the resulting equations with respect to $v$ over $\mathbb{R}^{3}$,
one has the following macroscopic (fluid) system:
\begin{align}\label{2.10}
	\left\{
	\begin{array}{rl}
		&\partial_{t}\rho+\nabla_{x}\cdot(\rho u)=0,
		\\
		&\partial_{t}(\rho u)+\nabla_{x}\cdot(\rho u\otimes u)+\nabla_{x}p=-\int_{\mathbb{R}^{3}} v\otimes v\cdot\nabla_{x}G\,dv,
		\\
		&\partial_{t}[\rho(\theta+\frac{1}{2}|u|^{2})]+\nabla_{x}\cdot[\rho u(\theta+\frac{1}{2}|u|^{2})+pu]
		=-\int_{\mathbb{R}^{3}} \frac{1}{2}|v|^{2} v\cdot\nabla_{x}G\,dv,
	\end{array} \right.
\end{align}
with the pressure $p=R\rho\theta=\frac{2}{3}\rho\theta$.
Here we have used \eqref{2.1}, \eqref{2.2} and the fact that $\partial_tG$ is microscopic by \eqref{2.6}.

Since the projection of collision terms under $P_{0}$ is zero, it holds that $P_{1}L_{M}G=L_{M}G$ and $P_{1}Q(G,G)=Q(G,G)$.
It is also clear to see $P_{1}\partial_tM=0$ and  $P_{1}\partial_tG=\partial_tG$ in terms of \eqref{2.6} and \eqref{2.5}.
With these facts,  one has the following microscopic (non-fluid) system:
\begin{align}
\label{2.11}
\partial_{t}G+P_{1}(v\cdot\nabla_{x}G)+P_{1}(v\cdot\nabla_{x}M)
=\frac{1}{\varepsilon}L_{M}G+\frac{1}{\varepsilon}Q(G,G),
\end{align}
by applying the microscopic operator $P_{1}$ to the Landau equation \eqref{2.8}. Since $L_{M}$
is invertible on $\mathcal{N}^\perp$, we can rewrite \eqref{2.11} to present $G$ as
\begin{equation}
\label{2.12}
G=\varepsilon L^{-1}_{M}[P_{1}(v\cdot\nabla_{x}M)]+L^{-1}_{M}\Theta,
\end{equation}
with
\begin{equation}
\label{2.13}
\Theta:=\varepsilon \partial_{t}G+\varepsilon P_{1}(v\cdot\nabla_{x}G)-Q(G,G).
\end{equation}
Plugging \eqref{2.12} into \eqref{2.10} and using the following two identities:
\begin{align*}
-\int_{\mathbb{R}^{3}} v_{i} v\cdot\nabla_{x}
L^{-1}_{M}[P_{1}(v\cdot\nabla_{x}M)]\,dv
&\equiv\sum^{3}_{j=1}\partial_{x_{j}}[\mu(\theta)D_{ij}],\quad i=1,2,3,
\\
-\int_{\mathbb{R}^{3}} \frac{1}{2}|v|^{2} v\cdot\nabla_{x}
L^{-1}_{M}[P_{1}(v\cdot\nabla_{x}M)]\,dv
&\equiv\sum^{3}_{j=1}\partial_{x_{j}}(\kappa(\theta)\partial_{x_{j}}\theta)+
\sum^{3}_{i,j=1}\partial_{x_{j}}[\mu(\theta) u_{i}D_{ij}],
\end{align*}
with the viscous stress tensor $D=[D_{ij}]_{1\leq i,j\leq 3}$ given by
\begin{equation}\label{def.vst}
D_{ij}=\partial_{x_{j}}u_{i}+\partial_{x_{i}}u_{j}-\frac{2}{3}\delta_{ij}\nabla_{x}\cdot u,
\end{equation}
we further obtain the compressible Navier-Stokes-type equations:
\begin{align}
\label{2.14}
\left\{
\begin{array}{rl}
\partial_{t}\rho+\nabla_{x}\cdot(\rho u)&=0,
\\
\partial_{t}(\rho u_{i})+\nabla_{x}\cdot(\rho u_{i}u)+\partial_{x_{i}}p
&\dis=\varepsilon\sum^{3}_{j=1}\partial_{x_{j}}[\mu(\theta)D_{ij}]\\
&\quad-\int_{\mathbb{R}^{3}} v_{i}(v\cdot\nabla_{x}L^{-1}_{M}\Theta)\,dv, \  i=1,2,3,
\\
\partial_{t}[\rho(\theta+\frac{1}{2}|u|^{2})]+\nabla_{x}\cdot[\rho u(\theta+\frac{1}{2}|u|^{2})+pu]
&\dis=\varepsilon\sum^{3}_{j=1}\partial_{x_{j}}(\kappa(\theta)\partial_{x_{j}}\theta)
+\varepsilon
\sum^{3}_{i,j=1}\partial_{x_{j}}[\mu(\theta) u_{i}D_{ij}]\\
&\quad-\int_{\mathbb{R}^{3}} \frac{1}{2}|v|^{2} v\cdot\nabla_{x}L^{-1}_{M}\Theta\,dv.
\end{array} \right.
\end{align}
Here $\mu(\theta)>0$ and $\kappa(\theta)>0$ are the viscosity coefficient and heat conductivity coefficient
respectively, and they are smooth functions depending only on the temperature $\theta$, represented by
\begin{align*}
	\mu(\theta)=&- R\theta\int_{\mathbb{R}^{3}}\hat{B}_{ij}(\frac{v-u}{\sqrt{R\theta}})
	B_{ij}(\frac{v-u}{\sqrt{R\theta}})\,dv>0,\quad i\neq j,
	\nonumber\\
	\kappa(\theta)=&-R^{2}\theta\int_{\mathbb{R}^{3}}\hat{A}_{j}(\frac{v-u}{\sqrt{R\theta}})
	A_{j}(\frac{v-u}{\sqrt{R\theta}})\,dv>0,
\end{align*}
where $\hat{A}_{j}(\cdot)$ and $\hat{B}_{ij}(\cdot)$ are Burnett functions, cf. \cite{Bardos,Duan-Yu1,Guo-CPAM,Ukai-Yang}, defined by
\begin{equation}
\label{5.1}
\hat{A}_{j}(v)=\frac{|v|^{2}-5}{2}v_{j}\quad \mbox{and} \quad \hat{B}_{ij}(v)=v_{i}v_{j}-\frac{1}{3}\delta_{ij}|v|^{2}, \quad \mbox{for} \quad i,j=1,2,3.
\end{equation}
And $A_{j}(\cdot)$ and $B_{ij}(\cdot)$ satisfy $P_{0}A_{j}(\cdot)=0$ and $P_{0}B_{ij}(\cdot)=0$, given by
\begin{equation}
\label{5.2}
A_{j}(v)=L^{-1}_{M}[\hat{A}_{j}(v)M]\quad
\mbox{and} \quad B_{ij}(v)=L^{-1}_{M}[\hat{B}_{ij}(v)M].
\end{equation}
Some elementary properties of the Burnett functions are summarized in the following lemma, cf.~ \cite{Bardos,Duan-Yu1,Guo-CPAM,Ukai-Yang}.
\begin{lemma}\label{lem7.2}
The Burnett functions have the following properties:
\begin{itemize}
\item{$-\langle \hat{A}_{i}, A_{i}\rangle$ ~~is positive and independent of i;}
\item{$\langle \hat{A}_{i}, A_{j}\rangle=0$ ~~for ~any ~$i\neq j$;\quad $\langle
\hat{A}_{i}, B_{jk}\rangle=0$~~for ~any ~i,~j,~k;}
\item{$\langle\hat{B}_{ij},B_{kj}\rangle=\langle\hat{B}_{kl},B_{ij}\rangle=\langle\hat{B}_{ji},B_{kj}\rangle$,~~
which is independent of ~i,~j, for fixed~~k,~l;}
\item{$-\langle \hat{B}_{ij}, B_{ij}\rangle$ ~~is positive and independent of i,~j when $i\neq j$;}	\item{$\langle \hat{B}_{ii}, B_{jj}\rangle$ ~~is positive and independent of i,~j when $i\neq j$;}
\item{$-\langle \hat{B}_{ii}, B_{ii}\rangle$ ~~is positive and independent of i;}
\item{$\langle \hat{B}_{ij}, B_{kl}\rangle=0$ ~~unless~either~$(i,j)=(k,l)$~or~$(l,k)$,~or~i=j~and~k=l;}	\item{$\langle \hat{B}_{ii}, B_{ii}\rangle-\langle \hat{B}_{ii}, B_{jj}\rangle=2\langle \hat{B}_{ij},
B_{ij}\rangle$ ~~holds for any~ $i\neq j$.}
\end{itemize}
\end{lemma}

To make a conclusion, we have decomposed the Landau equation \eqref{1.1} 
as the coupling of the viscous compressible fluid-type system \eqref{2.14} and the microscopic equation \eqref{2.11}, which is similar to the case of
the Boltzmann equation, cf. \cite{Liu-Yang-Yu}. In this way, 
one advantage is that the viscosity and heat conductivity coefficients can be expressed explicitly so that the energy analysis in the context of the viscous compressible fluid can be applied to capture the dissipation of the fluid part.  The other advantage is that
the nonlinear term $Q(G,G)$ in \eqref{2.11}
depends only on the microscopic part $G$ so that
the trilinear estimate is easily obtained without treating the fluid part, which is essentially different from the one in \cite{Guo-2002} with respect to a given global  Maxwellian.

As pointed out in \cite{Liu-Yang-Yu}, when the Knudsen number $\varepsilon$ and the microscopic
part $G$ are set to be zero, the system \eqref{2.14} becomes the compressible Euler system.
If only the microscopic part $G$ is set to be zero,
the system \eqref{2.14} becomes the compressible Navier-Stokes system with the parameter $\varepsilon$. 
These fluid systems can also be derived  from the Boltzmann (Landau) equation through the Hilbert and Chapman-Enskog expansions, cf. \cite{Chapman}.
This means that the macro-micro decomposition  \eqref{2.7} in some sense can be viewed
as a unification of the classical Hilbert and Chapman-Enskog expansions up to the second order approximation.
Therefore, this decomposition gives a good framework for deriving rigorously  the compressible fluid system from collisional kinetic equations.

\section{Main results}\label{sec.mr}

In this section, we employ the macro-micro decomposition introduced in
the previous section to	establish our main results on the compressible fluid limit for smooth solutions to the Landau equation.

\subsection{Compressible Euler limit}
The first goal of the present paper is to
establish the compressible Euler limit of the Landau equation. Precisely, we will construct the solution $F^\varepsilon(t,x,v)$ of the Landau equation \eqref{1.1} which converges to a local Maxwellian 
\begin{equation}
\label{1.5}
\overline{M}\equiv M_{[\bar{\rho},\bar{u},\bar{\theta}]}(t,x,v):
=\frac{\bar{\rho}(t,x)}{\sqrt{(2\pi R\overline{\theta}(t,x))^{3}}}\exp\big\{-\frac{|v-\overline{u}(t,x)|^2}{2R\overline{\theta}(t,x)}\big\},
\end{equation}
as the Knudsen number $\varepsilon>0$ tends to zero, 
where fluid parameters $(\bar{\rho},\bar{u},\bar{\theta})(t,x)$
satisfy the compressible Euler system \eqref{1.4}.

\subsubsection{Smooth solutions for Euler system}
To solve \eqref{1.4}, we supplement it with prescribed initial data
\begin{equation}
\label{1.6a}
(\bar{\rho},\bar{u},\bar{\theta})(0,x)=(\bar{\rho}_{0},\bar{u}_{0},\bar{\theta}_{0})(x).
\end{equation}
Then we have the following existence result.

\begin{proposition}\label{prop.1.1}
Let $\tau>0$ be a fixed finite time, then there is a sufficiently small constant $\eta_\tau>0$ and a constant $C_\tau>0$ such that if the initial data 
$(\bar{\rho}_{0},\bar{u}_{0},\bar{\theta}_{0})(x)$ around the constant state $(1,0,3/2)$ satisfies
\begin{equation}
\label{1.6}
\eta_0:=\|(\bar{\rho}_{0}(x)-1,\bar{u}_{0}(x),\bar{\theta}_{0}(x)-\frac{3}{2})\|_{H^{k}}
\leq \eta_{\tau}
\end{equation}
for $k\geq N+2$ with integer $N\geq 3$ as in \eqref{3.2}, then the Cauchy problem on the compressible Euler system \eqref{1.4} and \eqref{1.6a} admits a unique smooth solution $(\bar{\rho},\bar{u},\bar{\theta})(t,x)$
over $[0,\tau]\times \R^3$ 
such that 
$$
\inf_{t\in[0,\tau],\,x\in \R^{3}}\bar{\rho}(t,x)>0,\quad
\inf_{t\in[0,\tau],\,x\in \R^{3}}\bar{\theta}(t,x)>0,
$$
and the following estimate holds true:
\begin{equation}
\label{1.7}
\sup_{t\in[0,\tau]}\|(\bar{\rho}(t,x)-1,\bar{u}(t,x),\bar{\theta}(t,x)-\frac{3}{2})\|_{H^{k}}\leq C_{\tau}\eta_{0}.
\end{equation}
\end{proposition}

The proof of Proposition \ref{prop.1.1} can be obtained by a straightforward modification of the arguments as in \cite[Lemma 3.1 and Lemma 3.2]{Guo-Jang-Jiang-2010}, so we omit the details for brevity. We remark that for any given $\tau>0$ we always let $\eta_\tau>0$ be small enough such that  $C_\tau\eta_{0}>0$ can be sufficiently small since 
the smallness of $C_\tau \eta_{0}$ is crucially used to close the a priori assumptions \eqref{3.1}.

\subsubsection{Reformulated system}
Let's now define the macroscopic perturbation around the smooth solution $(\bar{\rho},\bar{u},\bar{\theta})(t,x)$  by
\begin{equation}
\label{2.15}
\left\{
\begin{array}{rl}
\widetilde{\rho}(t,x)&={\rho}(t,x)-\bar{\rho}(t,x),
\\
\widetilde{u}(t,x)&={u}(t,x)-\bar{u}(t,x),
\\
\widetilde{\theta}(t,x)&={\theta}(t,x)-\overline{\theta}(t,x),
\end{array} \right.
\end{equation}
where $(\rho,u,\theta)(t,x)$ satisfies \eqref{2.14}, defined by \eqref{2.2}.
Subtracting \eqref{1.4} from system \eqref{2.14}, one obtains the system for the perturbation $(\widetilde{\rho},\widetilde{u},\widetilde{\theta})(t,x)$ as follows
\begin{align}
\label{2.23}
\left\{
\begin{array}{rl}
&\partial_{t}\widetilde{\rho}
+u\cdot\nabla_{x}\widetilde{\rho}+\bar{\rho}\nabla_{x}\cdot\widetilde{u}
+\widetilde{u}\cdot\nabla_{x}\bar{\rho}
+\widetilde{\rho}\nabla_{x}\cdot u=0,
\\
&\partial_{t}\widetilde{u}_{i}+u\cdot\nabla_{x}\widetilde{u}_{i}+\frac{2\bar{\theta}}{3\bar{\rho}}\partial_{x_i}\widetilde{\rho}
+\frac{2}{3}\partial_{x_i}\widetilde{\theta}
+\widetilde{u}\cdot\nabla_{x}\bar{u}_{i}+\frac{2}{3}(\frac{\theta}{\rho}-\frac{\bar{\theta}}{\bar{\rho}})\partial_{x_i}\rho
\\
&\qquad\qquad\qquad=\varepsilon\frac{1}{\rho}\sum^{3}_{j=1}\partial_{x_j}[\mu(\theta)D_{ij}]
-\frac{1}{\rho}\int_{\mathbb{R}^{3}} v_{i}v\cdot\nabla_{x}L^{-1}_{M}\Theta\,dv, \quad i=1,2,3,
\\
&\partial_{t}\widetilde{\theta}+u\cdot\nabla_{x}\widetilde{\theta}+\frac{2}{3}\bar{\theta}\nabla_{x}\cdot \widetilde{u}
+\widetilde{u}\cdot\nabla_{x}\bar{\theta}+\frac{2}{3}\widetilde{\theta}\nabla_{x}\cdot u\\
&\qquad\qquad\qquad=\varepsilon\frac{1}{\rho}\sum^{3}_{j=1}\partial_{x_j}(\kappa(\theta)\partial_{x_j}\theta)
+\varepsilon\frac{1}{\rho}\sum^{3}_{i,j=1}\mu(\theta) \partial_{x_j}u_{i}D_{ij}\\
&\qquad\qquad\qquad\qquad-\frac{1}{\rho}\int_{\mathbb{R}^{3}} \frac{1}{2}|v|^{2} v\cdot\nabla_{x}L^{-1}_{M}\Theta\,dv+\frac{1}{\rho}u\cdot\int_{\mathbb{R}^{3}} v\otimes v\cdot\nabla_{x}L^{-1}_{M}\Theta\,dv.
\end{array} \right.
\end{align}
Throughout the paper we fix a normalized global Maxwellian with the fluid constant state $(1,0,3/2)$
\begin{equation}
\label{1.8}
\mu\equiv M_{[1,0,\frac{3}{2}]}(v):=(2\pi)^{-\frac{3}{2}}\exp\big(-\frac{|v|^{2}}{2}\big)
\end{equation}
as the reference equilibrium state. Then we define the microscopic perturbation $f(t,x,v)$  by
\begin{equation}
\label{3.8A}
\sqrt{\mu}f(t,x,v)=G(t,x,v)-\overline{G}(t,x,v),
\end{equation}
where the function $\overline{G}(t,x,v)$ is defined as
\begin{align}
\label{2.16}
\overline{G}(t,x,v)\equiv\varepsilon L_{M}^{-1}P_{1}\big\{v\cdot(\frac{|v-u|^{2}
\nabla_{x}\bar{\theta}}{2R\theta^{2}}+\frac{(v-u)\cdot\nabla_{x}\bar{u}}{R\theta})M\big\},
\end{align}
which corresponds to the first-order correction term in the Chapman-Enskog expansion, see \eqref{2.12} and \eqref{2.14}.
Note that the similar correction function $\overline{G}(t,x,v)$ was first introduced by Liu-Yang-Yu-Zhao \cite{Liu-Yang-Yu-Zhao} for the stability of the rarefaction wave to the one-dimensional Boltzmann equation with cutoff. Some detailed comments on introducing the subtraction $G-\overline{G}$ in our case
will be given later on.

To derive the equation of the microscopic perturbation $f(t,x,v)$ in \eqref{3.8A}, by the properties of $P_{1}$ in \eqref{2.5} and the definition of
$\overline{G}$, we have from a direct computation that
\begin{align}
\label{3.18AA}
P_{1}(v\cdot\nabla_{x}M)
&=P_{1}\big\{v\cdot(\frac{|v-u|^{2}
\nabla_{x}\theta}{2R\theta^{2}}
+\frac{(v-u)\cdot\nabla_{x}u}{R\theta})M\big\}
\nonumber\\
&=P_{1}\big\{v\cdot(\frac{|v-u|^{2}\nabla_{x}\widetilde{\theta}}{2R\theta^{2}}
+\frac{(v-u)\cdot\nabla_{x}\widetilde{u}}{R\theta})M\big\}
+\frac{1}{\varepsilon}L_{M}\overline{G}.
\end{align}
With this and the fact that $L_{M}G=L_{M}(\overline{G}+\sqrt{\mu}f)$, we can rewrite \eqref{2.11} as
\begin{equation}
\label{2.18}
\partial_{t}G+P_{1}(v\cdot\nabla_{x}G)
+P_{1}\big\{v\cdot(\frac{|v-u|^{2}\nabla_{x}\widetilde{\theta}}{2R\theta^{2}}
+\frac{(v-u)\cdot\nabla_{x}\widetilde{u}}{R\theta})M\big\}
=\frac{1}{\varepsilon}L_{M}(\sqrt{\mu}f)+\frac{1}{\varepsilon}Q(G,G).
\end{equation}
Inspired by \cite{Guo-2002}, we denote
\begin{equation}
\label{2.19}
\Gamma(h,g):=\frac{1}{\sqrt{\mu}}Q(\sqrt{\mu}h,\sqrt{\mu}g),
\quad \mathcal{L}h:=\Gamma(h,\sqrt{\mu})+\Gamma(\sqrt{\mu},h),
\end{equation}
which together with \eqref{2.9} immediately gives rise to
\begin{align}
\label{2.20}
\frac{1}{\sqrt{\mu}}L_{M}(\sqrt{\mu}f)&=
\frac{1}{\sqrt{\mu}}\{Q(M,\sqrt{\mu}f)+Q(\sqrt{\mu}f,M)\}
\nonumber\\
&=\mathcal{L}f+\Gamma(\frac{M-\mu}{\sqrt{\mu}},f)
+\Gamma(f,\frac{M-\mu}{\sqrt{\mu}}).
\end{align}
With these and the fact  that $G=\overline{G}+\sqrt{\mu}f$, \eqref{2.18} implies that $f$ satisfies
\begin{align}
\label{2.21}
\partial_{t}f+v\cdot\nabla_{x}f=&\frac{1}{\varepsilon}\mathcal{L}f+\frac{1}{\varepsilon}\Gamma(\frac{M-\mu}{\sqrt{\mu}},f)
+\frac{1}{\varepsilon}\Gamma(f,\frac{M-\mu}{\sqrt{\mu}})
+\frac{1}{\varepsilon}\Gamma(\frac{G}{\sqrt{\mu}},\frac{G}{\sqrt{\mu}})
\nonumber\\
&+\frac{P_{0}[v\cdot\nabla_{x}(\sqrt{\mu}f)]}{\sqrt{\mu}}
-\frac{P_{1}(v\cdot\nabla_{x}\overline{G})}{\sqrt{\mu}}-\frac{\partial_{t}\overline{G}}{\sqrt{\mu}}
\nonumber\\
&-\frac{1}{\sqrt{\mu}}P_{1}\big\{v\cdot(\frac{|v-u|^{2}
\nabla_{x}\widetilde{\theta}}{2R\theta^{2}}+\frac{(v-u)\cdot\nabla_{x}\widetilde{u}}{R\theta})M\big\}.
\end{align}

\begin{remark}\label{rm3.2}
It is well known that the linearized operator $\mathcal{L}$ in \eqref{2.21} is self-adjoint and non-positive definite, and its null space $\ker\mathcal{L}$ is spanned by the basis $\{\sqrt{\mu},v\sqrt{\mu},|v|^{2}\sqrt{\mu}\}$, cf. \cite{Guo-2002}.
\end{remark}

\begin{remark}\label{rm3.3}
It should be pointed out that the unknown $f(t,x,v)$ in \eqref{2.21}  is purely
microscopic since $G$ and $\overline{G}$ are purely microscopic, namely $f(t,x,v)\in (\ker\mathcal{L})^{\perp}$. This is essentially different from the one in \cite{Guo-2002}
since $f(t,x,v)$  used in \cite{Guo-2002} involves the macroscopic part. More comments will be given in Section \ref{sec.3.3A}.
\end{remark}

Motivated by the fundamental work Guo \cite{Guo-2002}, we introduce a velocity weight function as
\begin{equation}
\label{2.24}
w\equiv w(v):=\langle v\rangle^{-1},\quad
\langle v\rangle=\sqrt{1+|v|^{2}}.
\end{equation}
Such a weight function is designed to deal with the velocity-derivatives of the free streaming term $v\cdot\nabla_{x}f$, see \eqref{4.78b}.
With \eqref{2.24}, for $\ell\in \mathbb{R}$ we denote the weighted $L^{2}$ norms as
\begin{equation*}
|f|_{2,\ell}^{2}\equiv\int_{\mathbb{R}^{3}}w^{2\ell}|f|^{2}\,dv, \quad  \|f\|_{2,\ell}^{2}\equiv\int_{\mathbb{R}^{3}}|f|_{2,\ell}^{2}\,dx.
\end{equation*}
Corresponding to the reference global Maxwellian $\mu$ in \eqref{1.8},
the Landau collision frequency is given by
\begin{equation}
\label{2.25}
\sigma^{ij}(v):=\Phi_{ij}\ast \mu=\int_{{\mathbb R}^3}\Phi_{ij}(v-v_{\ast})\mu(v_{\ast})\,dv_{\ast},
\end{equation}
where $[\sigma^{ij}(v)]_{1\leq i,j\leq 3}$ is a positive-definite self-adjoint matrix.
In terms of linearization of the nonlinear Landau operator around $\mu$ (cf.~\cite{Guo-2002}),
with \eqref{2.25}, we define the weighted dissipation norms as
\begin{equation*}
|f|_{\sigma,\ell}^{2}\equiv\sum^{3}_{i,j=1}\int_{\mathbb{R}^{3}}w^{2\ell}\{\sigma^{ij}\partial_{v_i}f\partial_{v_j}f+
\sigma^{ij}\frac{v_{i}}{2}\frac{v_{j}}{2}|f|^{2}\}\,dv, \quad  \|f\|_{\sigma,\ell}^{2}\equiv\int_{\mathbb{R}^{3}}|f|_{\sigma,\ell}^{2}\,dx,
\end{equation*}
where $|f|_{\sigma}=|f|_{\sigma,0}$ and $\|f\|_{\sigma}=\|f\|_{\sigma,0}$.
Moreover, we have from \cite[Corollary 1, p.399]{Guo-2002} that
\begin{equation}
\label{2.26}
|f|_{\sigma}\approx |\langle v\rangle^{-\frac{1}{2}}f|_2+\Big|\langle v\rangle^{-\frac{3}{2}}\nabla_vf\cdot\frac{v}{|v|}\Big|_2+\Big|\langle v\rangle^{-\frac{1}{2}}\nabla_vf\times\frac{v}{|v|}\Big|_2.
\end{equation}
The notation  $A\approx B$ means that there exists $C>1$
such that $C^{-1}B\leq A\leq CB$.

In order to prove the uniform-in-$\varepsilon$ existence of smooth solutions for the Landau equation, a key point is to
establish uniform energy estimates in the high order Sobolev space. For this, we define the instant  energy functional $\mathcal{E}_N(t)$ as
\begin{align}
\label{3.2}
\mathcal{E}_N(t):=&\sum_{|\alpha|\leq N-1}\{\|\partial^{\alpha}(\widetilde{\rho},\widetilde{u},\widetilde{\theta})(t)\|^{2}
+\|\partial^{\alpha}f(t)\|^{2}\}+\sum_{|\alpha|+|\beta|\leq N,|\beta|\geq1}\|\partial^{\alpha}_{\beta}f(t)\|_{2,|\beta|}^{2}
\nonumber\\
&+\varepsilon^{2}\sum_{|\alpha|=N}\{\|\partial^{\alpha}(\widetilde{\rho},\widetilde{u},\widetilde{\theta})(t)\|^{2}
+\|\partial^{\alpha}f(t)\|^{2}\}.
\end{align}
Correspondingly, the dissipation energy functional $\mathcal{D}_N(t)$ is given by
\begin{align}
\label{3.8}
\mathcal{D}_N(t)&:=\varepsilon\sum_{1\leq|\alpha|\leq N}\|\partial^{\alpha}(\widetilde{\rho},\widetilde{u},\widetilde{\theta})(t)\|^{2}
+\varepsilon\sum_{|\alpha|=N}\|\partial^{\alpha}f(t)\|_{\sigma}^{2}
\nonumber\\
&\quad+\frac{1}{\varepsilon}\sum_{|\alpha|\leq N-1}\|\partial^{\alpha}f(t)\|_{\sigma}^{2}
+\frac{1}{\varepsilon}\sum_{|\alpha|+|\beta|\leq N,|\beta|\geq1}\|\partial^{\alpha}_{\beta}f(t)\|_{\sigma,|\beta|}^{2}.
\end{align}
Throughout this article, we assume $N\geq 3$. A crucial
feature of the above instant energy $\mathcal{E}_N(t)$ is that the highest $N$th-order space derivatives
are much more singular with respect to $\varepsilon$ than those derivatives of order up to $N-1$, and
it occurs similarly to the dissipation rate $\mathcal{D}_N(t)$ for the non-fluid component.

\subsubsection{Main result}
With the above preparations, our first result
on the compressible Euler limit can be stated as follows.

\begin{theorem}\label{thm1.1}
Let $\tau>0$ be given and $(\bar{\rho},\bar{u},\bar{\theta})(t,x)$  
be the smooth solution of the compressible Euler system \eqref{1.4} and \eqref{1.6a} given in Proposition \ref{prop.1.1}.
Construct the  local Maxwellian $M_{[\bar{\rho},\bar{u},\bar{\theta}]}(t,x,v)$
from $(\bar{\rho},\bar{u},\bar{\theta})(t,x)$  
as in \eqref{1.5}. Then there exists a small constant $\varepsilon_{0}>0$
such that for each $\varepsilon\in(0,\varepsilon_{0})$,  the Cauchy problem on the Landau equation \eqref{1.1} with non-negative initial data
\begin{align}
\label{1.10}
F^{\varepsilon}(0,x,v)\equiv M_{[\bar{\rho},\bar{u},\bar{\theta}]}(0,x,v)
\end{align}
admits a unique smooth solution $F^{\varepsilon}(t,x,v)\geq 0$ over $[0,\tau]\times \R^3\times \R^3$, satisfying the estimate
\begin{align}
\label{3.21A}
\mathcal{E}_N(t)+\frac{1}{2}\int^{t}_{0}\mathcal{D}_N(s)\,ds\leq \frac{1}{2}\varepsilon^{2},
\end{align}
for any $0\leq t\leq \tau$. In particular, there exists a constant $C_\tau>0$ independent of $\varepsilon$ such that
\begin{align}
\label{1.11}
\sup_{t\in[0,\tau]}\|\frac{F^{\varepsilon}(t,x,v)-M_{[\bar{\rho},\bar{u},\bar{\theta}]}(t,x,v)}{\sqrt{\mu}}
\|_{L_{x}^{2}L_{v}^{2}}+
\sup_{t\in[0,\tau]}\|\frac{F^{\varepsilon}(t,x,v)-M_{[\bar{\rho},\bar{u},\bar{\theta}]}(t,x,v)}{\sqrt{\mu}}
\|_{L_{x}^{\infty}L_{v}^{2}}\leq C_{\tau}\varepsilon.
\end{align}
\end{theorem}

\begin{remark}	
We point out that the energy inequality \eqref{3.21A} holds true at $t=0$ under the choice of the initial data
\eqref{1.10}. For this, one can claim that the initial data \eqref{1.10} automatically satisfies
\begin{align}
\label{3.5}
\mathcal{E}_N(t)\mid_{t=0}\leq C\eta^{2}_{0}\varepsilon^{2}.
\end{align}
In fact, we have $M_{[\rho,u,\theta]}(0,x,v)=M_{[\bar{\rho},\bar{u},\bar{\theta}]}(0,x,v)$ and $G(0,x,v)=0$
in terms of the decomposition $F(t,x,v)=M_{[\rho,u,\theta]}(t,x,v)+G(t,x,v)$ and \eqref{1.10}. 
This implies that $(\rho,u,\theta)(0,x)=(\bar{\rho},\bar{u},\bar{\theta})(0,x)$ and $\overline{G}(0,x,v)+\sqrt{\mu}f(0,x,v)=0$.
Then it holds that
\begin{align*}
(\widetilde{\rho},\widetilde{u},\widetilde{\theta})(0,x)
=({\rho}-\overline{\rho},u-\overline{u},\theta-\overline{\theta})(0,x)=0 \quad \mbox{and} \quad
f(0,x,v)=-\frac{\overline{G}(0,x,v)}{\sqrt{\mu}}.
\end{align*}
With these and Lemma \ref{lem5.4}, it is direct to see that \eqref{3.5} holds. Therefore,  \eqref{3.21A} for $t=0$ follows by letting $\eta_0>0$ be small enough.
\end{remark}

\begin{remark}
As pointed out in the aforementioned relevant literature, all those known results \cite{Caflisch,Lachowicz,Nishida,Ukai-Asano,Guo-Jang-Jiang,Guo-Jang-Jiang-2010,Guo-Huang-Wang,Jiang-Luo,Jiang-Luo-1} on the hydrodynamic limit from the Boltzmann equation to the compressible Euler system treat only the angular cutoff case and it still remains largely open to obtain the similar results for  both the non-cutoff Boltzmann equation and the Landau equation, essentially due to  the effect of grazing singularity of collision operator. To the best of our knowledge, Theorem \ref{thm1.1} seems the first result concerning the hydrodynamic limit for smooth solutions of the Landau equation to the ones of the compressible Euler system in the whole space over any finite time interval. Moreover, one may expect that
similar result should also hold for the non-cutoff Boltzmann equation.
\end{remark}

\begin{remark}
It should be pointed out that all the results \cite{Guo-Huang-Wang,Guo-Jang-Jiang,Guo-Jang-Jiang-2010,Jiang-Luo,Jiang-Luo-1} are based on
the Hilbert expansion and an $L^{2}-L^{\infty}$  interplay method developed  by Guo \cite{Guo-2008} for the Boltzmann equation with cutoff in a general bounded domain. Recently Kim-Guo-Hwang \cite{Kim} also developed an  $L^{2}-L^{\infty}$ approach to the Landau equation around global Maxwellians in the torus domain, where initial data are required to be small in  $L_{x,v}^{\infty}$ but additionally belong to
$H^1_{x,v}$, and this result was improved later by
Guo-Hwan-Jang-Ouyang \cite{GHJO} in a general bounded domain. However, these methods seem  difficult to be directly carried over to our problem
on the compressible Euler limit of the Landau equation in the whole space.
\end{remark}

\begin{remark}
Our analysis tool is mainly the combination of techniques for viscous compressible fluids, properties of Burnett functions and  the elaborate energy approach based on the macro-micro decomposition of the solution
for the Landau equation with respect to the local Maxwellian that was initiated by
Liu-Yu \cite{Liu-Yu} and developed by Liu-Yang-Yu \cite{Liu-Yang-Yu} in the Boltzmann theory. Thus, the idea we will adopt for the proof is different from the approach used by Caflisch in \cite{Caflisch} via the  truncated Hilbert expansion
and by Guo-Jang-Jiang in \cite{Guo-Jang-Jiang-2010}
via the Hilbert expansion and the $L^{2}-L^{\infty}$ interplay estimates. 
\end{remark}


\subsection{Acoustic limit}
The second goal of the present paper is to
establish  the acoustic limit of the Landau equation. The acoustic system is the linearization around the uniform equilibrium
$\rho=R\theta=1$ and $u=0$ for the compressible Euler system.  After a suitable choice of units to be consistent with notations in \cite{Guo-Jang-Jiang-2010}, the fluid fluctuations
 $(\varrho,\varphi,\vartheta)=(\varrho,\varphi,\vartheta)(t,x)$ satisfy
\begin{align}
\label{1.12}
\left\{
\begin{array}{rl}
\partial_{t}\varrho+\nabla_{x}\cdot\varphi=0,
\\
\partial_{t}\varphi+\nabla_{x}(\varrho+\vartheta)=0,
\\
\partial_{t}\vartheta+\frac{2}{3}\nabla_{x}\cdot\varphi=0.
\end{array} \right.
\end{align}
To solve \eqref{1.12}, we supplement it with prescribed initial data
\begin{align}
\label{1.13}
(\varrho,\varphi,\vartheta)(0,x)=(\varrho_{0},\varphi_{0},\vartheta_{0})(x)\in H^{s}(\R^{3}),
\quad \mbox{for}~~s\geq0.
\end{align}
It is well known that the Cauchy problem
on the acoustic system \eqref{1.12}-\eqref{1.13}
admit a unique global-in-time classical solution 
$(\varrho,\varphi,\vartheta)(t,x)\in C([0,+\infty);H^{s}(\mathbb{R}^3))$. Moreover, the solution satisfies
\begin{align}\notag
\|(\varrho,\varphi,\sqrt{\frac{3}{2}}\vartheta)(t)\|^2_{H^{s}}=
\|(\varrho_0,\varphi_0,\sqrt{\frac{3}{2}}\vartheta_0)\|^2_{H^{s}}, \quad \mbox{for~all}~t\geq 0.
\end{align}

On the other hand, the acoustic system \eqref{1.12} can also be formally derived from
the Landau equation \eqref{1.1} by letting
\begin{align}
\label{1.15}
F^{\varepsilon}(t,x,v)=\mu+\delta\sqrt{\mu}\ \mathbf{f}^{\varepsilon}(t,x,v),
\end{align}
where $\mu$ is the global Maxwellian given by \eqref{1.8}.
The fluctuation amplitude $\delta$ is a function of $\varepsilon$ satisfying
\begin{align}
\label{1.16}
\delta\to 0,\quad \mbox{as}\quad \varepsilon\to 0.
\end{align}
For instance, one may take
$$
\delta=\varepsilon^{\sigma}, \quad \mbox{for~any}~~\sigma>0.
$$
With the above scalings, $\mathbf{f}^{\varepsilon}=\mathbf{f}^{\varepsilon}(t,x,v)$ formally converges to
\begin{align}
\label{1.17}
\mathbf{f}=\big\{\varrho+v\cdot\varphi+(\frac{|v|^2-3}{2})\vartheta   \big\}\sqrt{\mu},\quad \mbox{as}\quad\varepsilon\to 0,
\end{align}
where $(\varrho,\varphi,\vartheta)$ is the solution of the acoustic system \eqref{1.12}. For a detailed formal derivation, see \cite{Bardos,Golse-2002}.

As in \cite{Guo-Jang-Jiang-2010}, to get
the optimal scaling, we use $\delta$  to denote the fluctuation amplitude and assume that
in addition to \eqref{1.16},
\begin{align}
\label{1.18}
\frac{\varepsilon}{\delta}\to 0,\quad \mbox{as}\quad\varepsilon\to 0.
\end{align}

We now state the second result on the  acoustic limit.
\begin{theorem}\label{thm1.2}
Let $\tau>0$ be given and
$(\varrho_{0},\varphi_{0},\vartheta_{0})(x)\in H^{N+2}(\R^{3})$ with $N\geq3$ be  initial data for the acoustic system \eqref{1.12}. Construct  the local Maxwellian
$$
\mu^{\delta}(0,x,v):=\frac{1+\delta\varrho_0(x)}{\sqrt{[2\pi(1+\delta\vartheta_0(x))]^{3}}}\exp\big\{-\frac{|v-\delta\varphi_0(x)|^2}{2(1+\delta\vartheta_0(x))}\big\}
$$	
in terms of the initial datum $1+\delta\varrho_{0}$, $\delta\varphi_{0}$ and $1+\delta\vartheta_{0}$. Then there exist  small constants $\varepsilon_{0}>0$ and $\delta_{0}>0$ such that for any $\varepsilon\in(0,\varepsilon_{0})$ and any $\delta\in(0,\delta_{0})$ satisfying the restrictions \eqref{1.16} and \eqref{1.18},
the Cauchy problem on the Landau equation \eqref{1.1} with non-negative initial data
\begin{align}\notag
F^{\varepsilon}(0,x,v)\equiv\mu^{\delta}(0,x,v)
\end{align}
admits a unique smooth solution $F^{\varepsilon}(t,x,v)\geq 0$ over $[0,\tau]\times \R^3\times\R^3$.
Moreover, let $F^{\varepsilon}(t,x,v)=\mu+\delta\sqrt{\mu}\ \mathbf{f}^{\varepsilon}(t,x,v)$ in terms of \eqref{1.15}, then the following convergence estimate holds:
\begin{align}
\label{1.21}
\sup_{t\in[0,\tau]}\|\mathbf{f}^{\varepsilon}(t,x,v)-\mathbf{f}(t,x,v)\|_{L_{x}^{2}L_{v}^{2}}+
\sup_{t\in[0,\tau]}\|\mathbf{f}^{\varepsilon}(t,x,v)-\mathbf{f}(t,x,v)\|_{L_{x}^{\infty}L_{v}^{2}}\leq C_{\tau}(\frac{\varepsilon}{\delta}+\delta),
\end{align}
where 
$\mathbf{f}(t,x,v)$ is defined in \eqref{1.17}  and the constant $C_\tau>0$ is independent of $\varepsilon$ and $\delta$.
\end{theorem}

\begin{remark}
The convergence rate obtained in \eqref{1.21}
should be optimal similar to the one in \cite{Guo-Jang-Jiang-2010} for the acoustic limit
of the Boltzmann equation with cutoff.
\end{remark}

\subsection{Strategy of the proof}\label{sec.3.3A}

In what follows we give some key points in the proof of main results.
As mentioned above, the $L^{2}-L^{\infty}$ framework  in \cite{Guo-Jang-Jiang-2010} cannot
be directly carried over to our problem on the compressible fluid limit of the Landau equation in the whole space
because their analysis depends crucially on the Grad's splitting of the linearized Boltzmann operator with cutoff,
which fails for the angular non-cutoff case or even in the grazing limit giving the Landau equation. Therefore, we have to
develop new ideas  instead of the $L^{2}-L^{\infty}$ framework.

Our strategy is based on the pure high order energy estimates framework developed in \cite{Liu-Yu,Liu-Yang-Yu} and \cite{Guo-2002} to construct global solutions in the perturbative framework for kinetic equations. Note that  the energy method  in \cite{Liu-Yu,Liu-Yang-Yu} is based on the macro-micro decomposition with respect to the local Maxwellian $M$ determined by the solution of kinetic equations. We can make use of such macro-micro decomposition to rewrite the Landau equation as a Navier-Stokes-type system with the non-fluid component appearing in the conservative source terms,
cf.~\eqref{2.14}, coupled with an equation for  the non-fluid component, cf.~\eqref{2.11}.

In order to carry out the energy estimate on the non-fluid component, we have
to overcome some major difficulties. First of all, we need to subtract $\overline{G}$ in \eqref{2.16} from $G$ so as to remove the trouble term
\begin{align*}
\frac{1}{\varepsilon}L_{M}\overline{G}=P_{1}\big\{v\cdot(\frac{|v-u|^{2}
\nabla_{x}\bar{\theta}}{2R\theta^{2}}+\frac{(v-u)\cdot\nabla_{x}\bar{u}}{R\theta})M\big\}
\end{align*}
when estimating the linear term $P_{1}(v\cdot\nabla_{x}M)$ in \eqref{2.11}, see the identity \eqref{3.18AA}.
Otherwise, a trouble term $\varepsilon(\|\nabla_{x}\bar{\theta}\|^2+\|\nabla_{x}\bar{u}\|^2)$
appears in the $L^2$ estimate. This term is bounded by $C\eta^2_{0}\varepsilon$ in terms of \eqref{1.7},
and it is out of control by $O(\varepsilon^2)$ corresponding to \eqref{3.1} so that the
energy estimate can not be closed. To further decompose $G-\overline{G}$, if one sets  $\sqrt{M}f=G-\overline{G}$, the equation of $f$ includes
the large-velocity growth term $\sqrt{M}^{-1}(\partial_{t}+v\cdot\nabla_{x})\sqrt{M} f$, which involves the cubic power of $v$ and this creates 
a key analytical difficulty in the
$L^2$ estimate similar to the one in \cite{Guo-Jang-Jiang-2010}. To avoid this difficulty, we set $\sqrt{\mu}f=G-\overline{G}$ 
to deduce the  microscopic equation for $f$ as \eqref{2.21} and then perform the energy estimate for $f$.

One of the most important point of the proof is that  $f\in (\ker\mathcal{L})^{\perp}$ in \eqref{2.21}  is purely microscopic such that the estimate
\begin{equation*}
-\frac{1}{\varepsilon}(\mathcal{L}f,f)\geq c_1\frac{1}{\varepsilon}\|f\|^{2}_{\sigma}
\end{equation*}
holds ture, which is again not true for the one in \cite{Guo-2002}. 
Indeed, \cite{Guo-2002}  used the decomposition $F=\mu+\sqrt{\mu}f$ with the perturbation
$f$ involving the macroscopic part, namely $f\notin(\ker\mathcal{L})^{\perp}$,  and as such one has to decompose $f=\mathbf{P}f+\{\mathbf{I-P}\}f$ with $\mathbf{P}$ the projection on the kernel space of $\mathcal{L}$,
so that one can only obtain
\begin{equation}\notag
-\frac{1}{\varepsilon}(\mathcal{L}f,f)\geq c_1\frac{1}{\varepsilon}\|\{\mathbf{I-P}\}f\|^{2}_{\sigma}.
\end{equation}
This results in the appearance of a difficult term $\frac{1}{\varepsilon}\|\mathbf{P}f\|_\sigma^2$
that would not be able to be controlled  since it involves a strong singularity about $\frac{1}{\varepsilon}$.
The main reason is that the $\varepsilon$-dependent coefficient of the hydrodynamic part $\mathbf{P}f$ in the energy dissipation functional is just $\varepsilon$ instead of $1/\varepsilon$.  Therefore, the fact that $f\in (\ker\mathcal{L})^{\perp}$ as in \eqref{2.21} is 
the most important point in the whole proof.
In this case, the trilinear estimate $\frac{1}{\varepsilon}(\Gamma(f,f),f)$ can be obtained easily because the difficult term
$\frac{1}{\varepsilon}(\Gamma(\mathbf{P}f,\mathbf{P}f),\{\mathbf{I-P}\}f)$ would no longer appear. Moreover, many known estimates on $\mathcal{L}$ and $\Gamma$ in \cite{Guo-2002} can be directly employed, for instance, see Lemma \ref{lem5.5} and Lemma \ref{lem5.6}. 

Note that the linearized Landau operator $\mathcal{L}$ in \eqref{2.21}  lacks a spectral gap, which results in the very weak velocity dissipation by 
\eqref{2.22a} and \eqref{2.26}.  The velocity-derivatives of the free streaming term $v\cdot\nabla_{x}f$ in the
$L^2$ estimate can not be directly bounded by the dissipation of $\mathcal{L}$. So
we adapt techniques in \cite{Guo-2002} basing on the velocity weight function as \eqref{2.24}
to overcome this difficult. To treat the $N$-order space derivative estimate, we need to deal with a complicated term 
$\frac{1}{\varepsilon}(\mathcal{L}\partial^{\alpha}f,\frac{\partial^{\alpha}M}{\sqrt{\mu}})=\frac{1}{\varepsilon}(\mathcal{L}\partial^{\alpha}f,\frac{I_{4}}{\sqrt{\mu}})
+\frac{1}{\varepsilon}(\mathcal{L}\partial^{\alpha}f,\frac{I_{5}}{\sqrt{\mu}})$ in \eqref{4.52a} and \eqref{4.48}.
The linear term $\frac{1}{\varepsilon}(\mathcal{L}\partial^{\alpha}f,\frac{I_{4}}{\sqrt{\mu}})$ cannot be directly estimated 
since a significant difficulty occurs to it. The key technique for handling this term is to use the properties  of the linearized operator  $\mathcal{L}$ and 
the relation between $M$ and $\mu$ as in \eqref{3.6} as well as the smallness of $\varepsilon$ and $\eta_0$.
In particular, to estimate the term $\frac{1}{\varepsilon}(\mathcal{L}\partial^{\alpha}f,\frac{I^{22}_{4}}{\sqrt{\mu}})$
in $\frac{1}{\varepsilon}(\mathcal{L}\partial^{\alpha}f,\frac{I_{4}}{\sqrt{\mu}})$,
we must move one derivative from the  $N$-order derivative $\mathcal{L}\partial^{\alpha}f$ to the other component of the inner product by integration by part, see \eqref{4.49}.
In the end, we can obtain the estimate 
\begin{align}
\label{1.26}
\frac{1}{\varepsilon}|(\mathcal{L}\partial^{\alpha}f,\frac{\partial^{\alpha}M}{\sqrt{\mu}})|\leq C(\eta_{0}+\varepsilon^{\frac{1}{2}})\frac{1}{\varepsilon}(\|\partial^{\alpha}f\|^{2}_{\sigma}+\|\partial^{\alpha}(\widetilde{\rho},\widetilde{u},\widetilde{\theta})\|^{2}+\frac{1}{\varepsilon^{2}}\|\partial^{\alpha'} f\|^{2}_{\sigma}+\varepsilon^{2}),
\end{align}
for $|\alpha'|=N-1$, see \eqref{4.48}-\eqref{4.60b} for detailed calculations. Similar difficulties also arise elsewhere, such as $\frac{1}{\varepsilon}(\Gamma(\frac{M-\mu}{\sqrt{\mu}},\partial^{\alpha}f),\frac{\partial^{\alpha}M}{\sqrt{\mu}})$ in \eqref{4.54}.
Because of  $\varepsilon^{-1}$ factor in front of the fluid part $\|\partial^{\alpha}(\widetilde{\rho},\widetilde{u},\widetilde{\theta})\|^{2}$ in \eqref{1.26}, one has to multiply the estimate \eqref{1.26} by $\varepsilon^2$
so that the fluid term can be controlled by $\mathcal{D}_N(t)$ in \eqref{3.8}. Hence, this motives how we include the Knudsen number $\varepsilon$ to the highest-order derivative $|\alpha|=N$ for the energy functional $\mathcal{E}_N(t)$ in \eqref{3.2}.

On the other hand, to make the energy estimate on the  fluid component, we need to treat those integral terms in \eqref{4.12}
involving the inverse of the linearized operator $L_M$ around the local Maxwellian $M$, such as the terms
\begin{equation}
\label{1.22}
\int_{\R^3} v_iv_j L_M^{-1} \Theta\,dv,\quad \int_{\R^3} v_i|v|^2L_M^{-1} \Theta\,dv,
\end{equation}
where $\Theta$ is defined in \eqref{2.13}. Those terms in \eqref{1.22} are difficult to estimate in a direct way since they
involve the polynomial velocity growth. To bound them, we follow the  strategy developed in \cite{Duan-Yu1,Duan-Yang-Yu} basing on the Burnett functions $\hat{A}_i$ and $\hat{B}_{ij}$ as in \eqref{5.1}. Indeed, in terms of the basic properties of the Burnett functions, the integral terms in \eqref{1.22} can be represented as the inner products of $A_i$ and $B_{ij}$ with $\Theta$, where $A_i$ and $B_{ij}$ defined in \eqref{5.2} are the inverse of $\hat{A}_i$ and $\hat{B}_{ij}$ under the linear operator $L_M$, respectively, see the identities \eqref{4.13}, \eqref{4.14} and \eqref{4.15} for details. Hence, any polynomial velocity growth in $\Theta$ can be absorbed since  $A_i$ and $B_{ij}$ enjoy the fast velocity decay, see \eqref{4.16}. In addition, we notice that any smooth solution of the compressible Euler system \eqref{1.4} and \eqref{1.6a} given in Proposition \ref{prop.1.1} does not enjoy an explicit time decay rate. One then has to use the smallness of $C_\tau\eta_{0}$ in \eqref{1.7} to control
those hard terms as in \eqref{4.14b}. This kind of technique will be used for the energy estimates on both the fluid-type system and the non-fluid system. 

Combining the energy estimates on the non-fluid component and the fluid component, we are able to
obtain the uniform a priori estimate \eqref{3.21A} and then  derive
the convergence rate \eqref{1.11} as stated in Theorem \ref{thm1.1}. Following the same strategy as above, we can obtain the similar
arguments as  \eqref{3.21A} and \eqref{1.11} under the assumptions in Theorem \ref{thm1.2}, and then  derive
the convergence rate \eqref{1.21}.

\section{Basic estimates}\label{seca.2}

In this section, we first make the a priori assumption in order to perform energy analysis conveniently.
Then we estimate the correction term and the complicated collision terms. Finally we derive an estimate on fluid
quantities involving with the temporal derivatives.
We should emphasize that in all estimates below, all constants $C>0$ at different places may depend on $\tau $ but  do not depend on both small parameters $\varepsilon$ and $\eta_0$.

\subsection{A priori assumption}Since the local existence of the solutions to the Landau equation near a global Maxwellian is well known in the  torus or the whole space, cf.~\cite{Guo-2002,Hsiao-Yu}, by a straightforward modification of the arguments there one can construct a unique short-time solution to the Landau equation \eqref{1.1}
under the assumptions in Theorem \ref{thm1.1}. The details are omitted for simplicity of presentation.
In order to extend the short-time solution to any finite time where Proposition \ref{prop.1.1} is satisfied, we only need to close the following a priori assumption
\begin{align}
\label{3.1}
\sup_{0\leq t\leq T}\mathcal{E}_N(t)\leq  \varepsilon^{2},
\end{align}
for an arbitrary time $T\in(0,\tau]$ with $\tau$ as in Proposition \ref{prop.1.1}, where $\mathcal{E}_N(t)$ is given by \eqref{3.2}.

Under the a priori assumption \eqref{3.1}, we have from the embedding inequality and \eqref{3.2} that
$$
\sup_{t\in[0,\tau]}\|\widetilde{\rho}(t,\cdot)\|_{L_x^{\infty}}
\leq C\sup_{t\in[0,\tau]}\|\widetilde{\rho}(t,\cdot)\|_{H^2_x}
\leq C\varepsilon,
$$
which together with \eqref{1.7} and \eqref{2.15} yields
$$
|\rho(t,x)-1|\leq |\rho(t,x)-\bar{\rho}(t,x)|+|\bar{\rho}(t,x)-1|
\leq C_\tau(\varepsilon+\eta_{0}).
$$
Similar estimates also hold for $u(t,x)$ and $\theta(t,x)$. Therefore,
for sufficiently small $\varepsilon$ and $\eta_{0}$, it holds that
\begin{equation}
\label{3.6}
|\rho(t,x)-1|+|u(t,x)|+|\theta(t,x)-\frac{3}{2}|<C_\tau(\varepsilon+\eta_{0}),\quad
1<\theta(t,x)<2,
\end{equation}
uniformly in all $(t,x)\in[0,\tau]\times\mathbb{R}^{3}$.
Note that under the condition \eqref{3.6}, there exist constants $C>0$ and $1<\sigma_1<\frac{3}{2}$ such that $M^{\sigma_1}\leq C\mu$ for all $(t,x,v)$.
It is crucial in our energy estimates.
\subsection{Sobolev inequalities}
We list several basic inequalities frequently used throughout this paper.
\begin{lemma}\label{lem5.1}
For any function $h=h(x)\in H^{1}(\mathbb{R}^{3})$, we have
\begin{align*}
\|h\|_{L^{6}(\mathbb{R}^{3})}&\leq C\|\nabla_{x} h\|,
\\
\|h\|_{L^{3}(\mathbb{R}^{3})}&\leq C\| h\|^{\frac{1}{2}}\|\nabla_{x} h\|^{\frac{1}{2}},
\\
\|h\|_{L^{p}(\mathbb{R}^{3})}&\leq C\| h\|_{H^1},\quad 2\leq p\leq 6,
\end{align*}
and for any function $h=h(x)\in H^{2}(\mathbb{R}^{3})$, it holds that
\begin{align*}
\|h\|_{L^{\infty}(\mathbb{R}^{3})}
\leq C\|\nabla_{x} h\|^{\frac{1}{2}}\|\nabla_{x}^{2} h\|^{\frac{1}{2}}.
\end{align*}
Here $C$ is a positive constant independent of $h(x)$.
\end{lemma}

\subsection{Estimates on correction term $\overline{G}$}
To perform the energy estimates for the equations \eqref{2.21} 
and \eqref{2.23}, one has to treat those integral terms involving $L^{-1}_{M}$ and $\overline{G}$.
For this, we need to give the estimate concerning the inverse of the linearized
operator whose proof can be found in \cite[Lemma 6.1]{Duan-Yu1}.

\begin{lemma}\label{lem5.2}
Suppose that $U(v)$ is any polynomial of $\frac{v-\hat{u}}{\sqrt{R}\hat{\theta}}$ such that
$U(v)\widehat{M}\in(\ker{L_{\widehat{M}}})^{\perp}$ for any Maxwellian $\widehat{M}=M_{[\widehat{\rho},\widehat{u},\widehat{\theta}]}(v)$ where $L_{\widehat{M}}$ is as \eqref{2.9}.
For any $\epsilon\in(0,1)$ and any multi-index $\beta$, there exists constant $C_{\beta}>0$ such that
$$
|\partial_{\beta}L^{-1}_{\widehat{M}}(U(v)\widehat{M})|\leq C_{\beta}(\widehat{\rho},\widehat{u},\widehat{\theta})\widehat{M}^{1-\epsilon}.
$$
In particular, under the condition of \eqref{3.6}, there exists a constant $C_{\beta}>0$ such that
\begin{equation}
\label{5.3}
|\partial_{\beta}A_{j}(\frac{v-u}{\sqrt{R\theta}})|+|\partial_{\beta}B_{ij}(\frac{v-u}{\sqrt{R\theta}})|
\leq C_{\beta}M^{1-\epsilon}.
\end{equation}
\end{lemma}
Based on the properties of the Burnett functions and Lemma \ref{lem5.2},
we can prove the following lemma which will be used frequently in the energy analysis later on.

\begin{lemma}
\label{lem5.4}
Assume \eqref{3.1} and \eqref{3.6} hold. Let $\overline{G}$ be defined
in \eqref{2.16} and $\langle v\rangle=\sqrt{1+|v|^{2}}$, then for any $l\geq 0$, $|\beta|\geq 0$ and $|\alpha|\leq N$, one has
\begin{align}
\label{5.5}
\|\langle v\rangle^{l}\partial^{\alpha}_{\beta}(\frac{\overline{G}}{\sqrt{\mu}})\|_{2,|\beta|}+
\|\langle v\rangle^{l}\partial^{\alpha}_{\beta}(\frac{\overline{G}}{\sqrt{\mu}})\|_{\sigma,|\beta|}
\leq C\eta_{0}\varepsilon.
\end{align}
\end{lemma}

\begin{proof}
In view of \eqref{5.1} and \eqref{5.2}, 
the term $\overline{G}$ in \eqref{2.16} can be represented precisely as
\begin{equation}
\label{5.6}
\overline{G}=\varepsilon\frac{\sqrt{R}}{\sqrt{\theta}}\sum^{3}_{j=1}
\frac{\partial\bar{\theta}}{\partial x_{j}}A_{j}(\frac{v-u}{\sqrt{R\theta}})
+\varepsilon\sum^{3}_{j=1}\sum^{3}_{i=1}
\frac{\partial\bar{u}_{j}}{\partial x_{i}}B_{ij}(\frac{v-u}{\sqrt{R\theta}}).
\end{equation}
Then, for $k=1,2,3$, we have
\begin{equation}
\label{5.7}
\frac{\partial\overline{G}}{\partial v_{k}}=\varepsilon\frac{\sqrt{R}}{\sqrt{\theta}}\sum^{3}_{j=1}
\frac{\partial\bar{\theta}}{\partial x_{j}}\partial_{v_{k}}A_{j}(\frac{v-u}{\sqrt{R\theta}})\frac{1}{\sqrt{R\theta}} +\varepsilon\sum^{3}_{i,j=1}\frac{\partial\bar{u}_{j}}{\partial x_{i}}
\partial_{v_{k}}B_{ij}(\frac{v-u}{\sqrt{R\theta}})\frac{1}{\sqrt{R\theta}},
\end{equation}
and
\begin{align}
\label{5.8}
\frac{\partial\overline{G}}{\partial x_{k}}&=\varepsilon\Big\{
\frac{\sqrt{R}}{\sqrt{\theta}}\sum^{3}_{j=1}\frac{\partial^{2}\bar{\theta}}{\partial x_{j}\partial x_{k}}A_{j}(\frac{v-u}{\sqrt{R\theta}})
-\frac{\sqrt{R}}{2\sqrt{\theta^{3}}}\sum^{3}_{j=1}\frac{\partial\bar{\theta}}{\partial x_{j}}
\frac{\partial\theta}{\partial x_{k}}A_{j}(\frac{v-u}{\sqrt{R\theta}})
\nonumber\\
&\quad-\frac{\sqrt{R}}{\sqrt{\theta}}\sum^{3}_{j=1}
\frac{\partial\bar{\theta}}{\partial x_{j}}\frac{\partial u}{\partial x_{k}}\cdot
\nabla_{v}A_{j}(\frac{v-u}{\sqrt{R\theta}})\frac{1}{\sqrt{R\theta}}
-\frac{\sqrt{R}}{\sqrt{\theta}}\sum^{3}_{j=1}
\frac{\partial\bar{\theta}}{\partial x_{j}}\frac{\partial\theta}{\partial x_{k}}\nabla_{v}A_{j}(\frac{v-u}{\sqrt{R\theta}})
\cdot\frac{v-u}{2\sqrt{R\theta^{3}}}
\nonumber\\
&\quad+\sum^{3}_{i,j=1}\frac{\partial^{2}\bar{u}_{j}}{\partial x_{i}\partial x_{k}}B_{ij}(\frac{v-u}{\sqrt{R\theta}})
-\sum^{3}_{i,j=1}\frac{\partial\bar{u}_{j}}{\partial x_{i}}\frac{\partial u}{\partial x_{k}}\cdot\nabla_{v}B_{ij}(\frac{v-u}{\sqrt{R\theta}})\frac{1}{\sqrt{R\theta}}
\nonumber\\
&\quad-\sum^{3}_{i,j=1}\frac{\partial\bar{u}_{j}}{\partial x_{i}}\frac{\partial\theta}{\partial x_{k}}
\nabla_{v}B_{ij}(\frac{v-u}{\sqrt{R\theta}})
\cdot\frac{v-u}{2\sqrt{R\theta^{3}}}
\Big\}.
\end{align}
Likewise, $\partial_{t}\overline{G}$ has the similar expression as \eqref{5.8}.
For any $|\beta|\geq0$, any $l\geq0$ and sufficiently
small $\epsilon>0$, we can deduce
from \eqref{3.6}, \eqref{2.24} and \eqref{2.26} that
\begin{equation}
\label{5.9}
|\langle v\rangle^{l}w^{|\beta|}\mu^{-\frac{1}{2}}M^{1-\epsilon}|_{2}
+|\langle v\rangle^{l}w^{|\beta|}\mu^{-\frac{1}{2}}M^{1-\epsilon}|_{\sigma}\leq C.
\end{equation}
For any $|\beta|\geq0$, we get by \eqref{5.3}, \eqref{5.6}, the similar arguments as \eqref{5.7}, \eqref{5.9} and \eqref{1.7} that
\begin{equation}
\label{5.4}
\|\langle v\rangle^{l}\partial_{\beta}(\frac{\overline{G}}{\sqrt{\mu}})\|^2_{2,|\beta|}
+\|\langle v\rangle^{l}\partial_{\beta}(\frac{\overline{G}}{\sqrt{\mu}})\|^2_{\sigma,|\beta|}
\leq C\varepsilon^2(\|\nabla_{x}\bar{u}\|^2+\|\nabla_{x}\bar{\theta}\|^2)\leq C\eta^2_{0}\varepsilon^2.
\end{equation}
For $|\alpha|=1$, we use \eqref{5.8}, \eqref{5.9}, \eqref{5.3}, \eqref{3.1} and \eqref{1.7} to get
\begin{align}
\label{4.10AA}
&\|\langle v\rangle^{l}\partial^{\alpha}_{\beta}(\frac{\overline{G}}{\sqrt{\mu}})\|^2_{2,|\beta|}+
\|\langle v\rangle^{l}\partial^{\alpha}_{\beta}(\frac{\overline{G}}{\sqrt{\mu}})\|^2_{\sigma,|\beta|}
\nonumber\\
&\leq 
C\varepsilon^2\int_{{\mathbb R}^3} (|\partial^{\alpha}\nabla_{x}\bar{u}|^2
+|\partial^{\alpha}\nabla_{x}\bar{\theta}|^2)+(|\nabla_{x}\bar{u}|^2+|\nabla_{x}\bar{\theta}|^2)
(|\partial^{\alpha}u|^2+|\partial^{\alpha}\theta|^2)\,dx
\nonumber\\
&\leq 
C\varepsilon^2(\|(\partial^{\alpha}\nabla_{x}\bar{u},\partial^{\alpha}\nabla_{x}\bar{\theta})\|^2+\|(\nabla_{x}\bar{u},\nabla_{x}\bar{\theta})\|_{L^{\infty}}^2\|(\partial^{\alpha}u,\partial^{\alpha}\theta)\|^2)
\nonumber\\
&\leq  C\eta^2_{0}\varepsilon^2(1+\eta^2_0+\varepsilon^2)\leq  C\eta^2_{0}\varepsilon^2.
\end{align}
Similar arguments also hold for the cases $2\leq|\alpha|\leq N$. Therefore, we can prove that the desired estimate \eqref{5.5} holds true. 
This ends the proof of Lemma \ref{lem5.4}.
\end{proof}

\subsection{Estimates on collision terms $\mathcal{L}$ and $\Gamma$}
Next, we summarize some refined estimates for the
linearized Landau  operator $\mathcal{L}$ and the nonlinear collision terms $\Gamma(g_1,g_2)$ defined in
\eqref{2.19}. We start from collecting some known basic estimates. The following two lemmas can be found in \cite[Lemma 6]{Guo-2002}
and \cite[Proposition 1]{Strain-Zhu}, respectively.
\begin{lemma}\label{lem5.5}
Let $|\alpha|\geq0$ and $|\beta|>0$, then for	
$w$ defined in \eqref{2.24} and any small $\eta>0$,
there exist $c_0>0$ and $C_{\eta}>0$ such that
\begin{equation}
\label{5.10}
-\langle\partial^{\alpha}_{\beta}\mathcal{L}g,w^{2|\beta|}\partial^{\alpha}_{\beta}g\rangle\geq
c_0|\partial^{\alpha}_{\beta}g|^{2}_{\sigma,|\beta|}-\eta\sum_{|\beta_{1}|\leq|\beta|}
|\partial^{\alpha}_{\beta_{1}}g|_{\sigma,|\beta_{1}|}^{2}-C_{\eta}|\partial^{\alpha}g|_{\sigma}^{2}.
\end{equation}
If $|\beta|=0$, then for any  $g\in (\ker\mathcal{L})^{\perp}$,
there exists a generic constant $c_{1}>0$ such that
\begin{equation}
\label{2.22a}
-\langle\mathcal{L}\partial^{\alpha}g, \partial^{\alpha}g \rangle\geq c_{1}|\partial^{\alpha}g|^{2}_{\sigma}.
\end{equation}
\end{lemma}
\begin{lemma}
\label{lem5.6}
Let $|\alpha|\geq0$ and $|\beta|\geq0$, then for arbitrarily large constant $b>0$, one has
\begin{align}
\label{5.11}
|\langle \partial^{\alpha}\Gamma(g_{1},g_{2}), \partial^{\alpha}g_{3}\rangle|
\leq C\sum_{\alpha_{1}\leq\alpha}|\langle v\rangle^{-b}\partial^{\alpha_{1}}g_{1}|_{2}
|\partial^{\alpha-\alpha_{1}}g_{2}|_{\sigma}
|\partial^{\alpha}g_{3}|_{\sigma}.
\end{align}
Moreover, for $w$ defined in \eqref{2.24} and $l\geq0$, one has
\begin{align}
\label{5.12}
|\langle \partial^{\alpha}_{\beta}\Gamma(g_{1},g_{2}), w^{2l}\partial^{\alpha}_{\beta}g_{3}\rangle|
\leq C\sum_{\alpha_{1}\leq\alpha}
\sum_{\beta'\leq\beta_{1}\leq\beta}|\langle v\rangle^{-b}\partial^{\alpha_{1}}_{\beta^{'}}g_{1}|_{2}
|\partial^{\alpha-\alpha_{1}}_{\beta-\beta_{1}}g_{2}|_{\sigma,l}
|\partial^{\alpha}_{\beta}g_{3}|_{\sigma,l}.
\end{align}
\end{lemma}
With Lemma \ref{lem5.6}, we now prove some nonlinear energy estimates. We first consider estimates on linear collision terms $\Gamma(\frac{M-\mu}{\sqrt{\mu}},f)$ and $\Gamma(f,\frac{M-\mu}{\sqrt{\mu}})$ in \eqref{2.21}, which will be used in Section \ref{seca.4}.
\begin{lemma}
\label{lem5.7}
Let $|\alpha|+|\beta|\leq N$ with $|\beta|\geq1$ and $w=(1+|v|^{2})^{-1/2}$ as in \eqref{2.24}.
Let $F=M+\overline{G}+\sqrt{\mu}f$ be the solution to the Landau equation \eqref{1.1} and \eqref{1.10}, and
assume \eqref{3.1} and \eqref{3.6} hold. If we choose $\eta_0>0$ in \eqref{1.7} and $\varepsilon>0$ in \eqref{3.1} small enough, then for any  $\eta>0$, one has 
\begin{align}
\label{5.13}
\frac{1}{\varepsilon}&|(\partial^{\alpha}_{\beta}\Gamma(\frac{M-\mu}{\sqrt{\mu}},f),	w^{2|\beta|}\partial^{\alpha}_{\beta}f)|+\frac{1}{\varepsilon}|(\partial^{\alpha}_{\beta}\Gamma(f,\frac{M-\mu}{\sqrt{\mu}}),w^{2|\beta|}\partial^{\alpha}_{\beta}f)|
\nonumber\\
&\leq 
C\eta\frac{1}{\varepsilon}\|\partial^{\alpha}_{\beta}f\|^{2}_{\sigma,|\beta|}
+C_{\eta}(\eta_{0}+\varepsilon^{\frac{1}{2}})\mathcal{D}_N(t).
\end{align}
Moreover, for $|\beta|=0$ and  $|\alpha|\leq N-1$, one has
\begin{align}
\label{5.14}
\frac{1}{\varepsilon}&|(\partial^{\alpha}\Gamma(\frac{M-\mu}{\sqrt{\mu}},f),\partial^{\alpha}f)|
+\frac{1}{\varepsilon}|(\partial^{\alpha}\Gamma(f,\frac{M-\mu}{\sqrt{\mu}}),\partial^{\alpha}f)|
\nonumber\\
&\leq C\eta\frac{1}{\varepsilon}\|\partial^{\alpha}f\|^{2}_{\sigma}+C_{\eta}(\eta_{0}+\varepsilon^{\frac{1}{2}})\mathcal{D}_N(t).
\end{align}
Here $\mathcal{D}_N(t)$ is defined by \eqref{3.8}.
\end{lemma}

\begin{proof}
For the first term on the left hand side of \eqref{5.13}, since $w^{2|\beta|}\leq w^{2|\beta-\beta_{1}|}$ for $|\beta-\beta_{1}|\leq|\beta|$, we have from this and \eqref{5.12} that
\begin{align}
\label{5.15}	
&\frac{1}{\varepsilon}|(\partial^{\alpha}_{\beta}\Gamma(\frac{M-\mu}{\sqrt{\mu}},f),	w^{2|\beta|}\partial^{\alpha}_{\beta}f)|
\nonumber\\
&\leq C\frac{1}{\varepsilon}\sum_{\alpha_{1}\leq\alpha}\sum_{\beta'\leq\beta_{1}\leq\beta}\int_{\mathbb{R}^{3}}|\langle v\rangle^{-b}\partial^{\alpha_{1}}_{\beta^{'}}(\frac{M-\mu}{\sqrt{\mu}})|_{2}|\partial^{\alpha-\alpha_{1}}_{\beta-\beta_{1}}f|_{\sigma,|\beta-\beta_{1}|}|\partial^{\alpha}_{\beta}f|_{\sigma,|\beta|}\,dx.
\end{align}
In order to further compute \eqref{5.15}, for any $|\bar{\beta}|\geq0$ and $l\geq0$ we claim that 
\begin{align}
\label{5.16}
|\langle v\rangle^{l}\partial_{\bar{\beta}}(\frac{M-\mu}{\sqrt{\mu}})|_{\sigma}+|\langle v\rangle^{l}\partial_{\bar{\beta}}(\frac{M-\mu}{\sqrt{\mu}})|_{2}\leq C(\eta_{0}+\varepsilon).
\end{align}
In fact, for any $|\bar{\beta}|\geq0$ and $l\geq 0$,
using \eqref{2.26}, we know that there exists a small constant $\epsilon_{1}>0$ such that
\begin{align*}
|\langle v\rangle^{l}\partial_{\bar{\beta}}(\frac{M-\mu}{\sqrt{\mu}})|^{2}_{\sigma}+|\langle v\rangle^{l}\partial_{\bar{\beta}}(\frac{M-\mu}{\sqrt{\mu}})|^{2}_{2}\leq C_{l}\sum_{|\bar{\beta}|\leq|\beta'|\leq|\bar{\beta}|+1}\int_{\mathbb{R}^{3}}\mu^{-\epsilon_{1}}|\partial_{\beta'}(\frac{M-\mu}{\sqrt{\mu}})|^{2}\,dv.
\end{align*}
Thanks to \eqref{3.6}, we can choose a suitably large constant  $R>0$ such that 
\begin{align*}
\int_{|v|\geq R}\mu^{-\epsilon_{1}}|\partial_{\beta'}(\frac{M-\mu}{\sqrt{\mu}})|^{2}\,dv\leq C(\eta_{0}+\varepsilon)^{2},
\end{align*}
and
\begin{align*}
\int_{|v|\leq R}\mu^{-\epsilon_{1}}|\partial_{\beta'}(\frac{M-\mu}{\sqrt{\mu}})|^{2}\,dv\leq C(|\rho-1|+|u-0|+|\theta-\frac{3}{2}|)^{2}\leq C(\eta_{0}+\varepsilon)^{2}.
\end{align*}
By these estimates, we can get
the desired estimate \eqref{5.16} and thus ends the proof of \eqref{5.16}.

We  first have from a direct calculations that
\begin{align}
\label{4.19B}
\partial_{x_i}M=M\big(\frac{\partial_{x_i}\rho}{\rho}+\frac{(v-u)\cdot\partial_{x_i}u}{R\theta}
+(\frac{|v-u|^{2}}{2R\theta}-\frac{3}{2})\frac{\partial_{x_i}\theta}{\theta} \big).
\end{align}
Then for $|\alpha|\geq 2$ and $\partial^{\alpha}=\partial^{\alpha'}\partial_{x_i}$, it holds that
\begin{align}
\label{4.20B}
\partial^{\alpha}M=&M\big(\frac{\partial^{\alpha}\rho}{\rho}+\frac{(v-u)\cdot\partial^{\alpha}u}{R\theta}
+(\frac{|v-u|^{2}}{2R\theta}-\frac{3}{2})\frac{\partial^{\alpha}\theta}{\theta} \big)
\nonumber\\
&+\sum_{1\leq\alpha_{1}\leq \alpha'}C^{\alpha_1}_{\alpha'}\big(\partial^{\alpha_{1}}(M\frac{1}{\rho})\partial^{\alpha'-\alpha_{1}}\partial_{x_i}\rho+\partial^{\alpha_{1}}(M\frac{v-u}{R\theta})\cdot\partial^{\alpha'-\alpha_{1}}\partial_{x_i}u
\nonumber\\
&\hspace{2cm}+\partial^{\alpha_{1}}(M\frac{|v-u|^{2}}{2R\theta^{2}}-M\frac{3}{2\theta})\partial^{\alpha'-\alpha_{1}}\partial_{x_i}\theta\big).
\end{align}	
We now turn to compute \eqref{5.15}.
Note that $|\alpha_1|\leq|\alpha|\leq N-1$ since we only consider the case $|\alpha|+|\beta|\leq N$ and $|\beta|\geq1$. If $|\alpha_{1}|=0$,  we get from \eqref{5.16}, the Cauchy-Schwarz inequality and \eqref{3.8} that
\begin{align*}
&\frac{1}{\varepsilon}\int_{\mathbb{R}^{3}}|\langle
v\rangle^{-b}\partial^{\alpha_{1}}_{\beta^{'}}(\frac{M-\mu}{\sqrt{\mu}})|_{2}|\partial^{\alpha-\alpha_{1}}_{\beta-\beta_{1}}f|_{\sigma,|\beta-\beta_{1}|}|\partial^{\alpha}_{\beta}f|_{\sigma,|\beta|}\,dx
\\
&\leq\eta\frac{1}{\varepsilon}\|\partial^{\alpha}_{\beta}f\|^{2}_{\sigma,|\beta|}+C_{\eta}(\eta_{0}+\varepsilon)^{2}\frac{1}{\varepsilon}\|\partial^{\alpha-\alpha_{1}}_{\beta-\beta_{1}}f\|^{2}_{\sigma,|\beta-\beta_{1}|}
\\
&\leq \eta\frac{1}{\varepsilon}\|\partial^{\alpha}_{\beta}f\|^{2}_{\sigma,|\beta|}+C_{\eta}(\eta_{0}+\varepsilon)^{2}\mathcal{D}_N(t).
\end{align*}
If  $1\leq|\alpha_1|\leq|\alpha|\leq N-1$, then $|\alpha-\alpha_{1}|\leq N-2$ and $|\alpha-\alpha_{1}|+|\beta-\beta_{1}|\leq|\alpha|+|\beta|-1$. Taking the $L^{6}-L^{3}-L^{2}$ H\"{o}lder inequality and using \eqref{3.6} and the Cauchy-Schwarz and Sobolev inequalities, we get
\begin{align*}
&\frac{1}{\varepsilon}\int_{\mathbb{R}^{3}}|\langle
v\rangle^{-b}\partial^{\alpha_{1}}_{\beta^{'}}(\frac{M-\mu}{\sqrt{\mu}})|_{2}|\partial^{\alpha-\alpha_{1}}_{\beta-\beta_{1}}f|_{\sigma,|\beta-\beta_{1}|}|\partial^{\alpha}_{\beta}f|_{\sigma,|\beta|}\,dx
\\
&\leq C\frac{1}{\varepsilon}\big\||\langle
v\rangle^{-b}\partial^{\alpha_{1}}_{\beta^{'}}(\frac{M-\mu}{\sqrt{\mu}})|_{2}\big\|_{L^{3}}
\big\||\partial^{\alpha-\alpha_{1}}_{\beta-\beta_{1}}f|_{\sigma,|\beta-\beta_{1}|}\big\|_{L^{6}}
\big\||\partial^{\alpha}_{\beta}f|_{\sigma,|\beta|}\big\|_{L^{2}}
\\
&\leq C\frac{1}{\varepsilon}(\eta_{0}+\varepsilon^{\frac{1}{2}})
\|\partial^{\alpha-\alpha_{1}}_{\beta-\beta_{1}}\nabla_{x}f\|_{\sigma,|\beta-\beta_{1}|}
\|\partial^{\alpha}_{\beta}f\|_{\sigma,|\beta|}
\\
&\leq \eta\frac{1}{\varepsilon}\|\partial^{\alpha}_{\beta}f\|^{2}_{\sigma,|\beta|}+C_{\eta}(\eta_{0}+\varepsilon^{\frac{1}{2}})^2\mathcal{D}_N(t),
\end{align*}
where we have used \eqref{1.7}, \eqref{3.1}, \eqref{4.19B}, \eqref{4.20B} and the fact 
$$
\||\langle v\rangle^{-b}\partial^{\alpha_{1}}_{\beta^{'}}(\frac{M-\mu}{\sqrt{\mu}})|_{2}\big\|_{L^{3}}
\leq C(\eta_{0}+\varepsilon^{\frac{1}{2}}),
$$
since one has to deal with  $|\alpha_1|=N-1$,
\begin{align*}
\|\partial^{\alpha_{1}}(\rho,u,\theta)\|_{L^{3}}&\leq C\|\partial^{\alpha_{1}}(\bar{\rho},\bar{u},\bar{\theta})\|_{L^{3}}
+C\|\partial^{\alpha_{1}}(\widetilde{\rho},\widetilde{u},\widetilde{\theta})\|_{L^{3}}
\\
&\leq C\eta_{0}
+C\|\partial^{\alpha_{1}}(\widetilde{\rho},\widetilde{u},\widetilde{\theta})\|^{\frac{1}{2}}
\|\nabla_x\partial^{\alpha_{1}}(\widetilde{\rho},\widetilde{u},\widetilde{\theta})\|^{\frac{1}{2}}
\\
&\leq  C(\eta_{0}+\varepsilon^{\frac{1}{2}}).
\end{align*}
Therefore, substituting the above estimates into \eqref{5.15} and using the smallness of  $\eta_{0}$ and $\varepsilon$, we obtain
\begin{align}
\label{5.17}
\frac{1}{\varepsilon}|(\partial^{\alpha}_{\beta}\Gamma(\frac{M-\mu}{\sqrt{\mu}},f),	w^{2|\beta|}\partial^{\alpha}_{\beta}f)|
\leq  C\eta\frac{1}{\varepsilon}\|\partial^{\alpha}_{\beta}f\|^{2}_{\sigma,|\beta|}+C_{\eta}(\eta_{0}+\varepsilon^{\frac{1}{2}})\mathcal{D}_N(t).
\end{align}
	
For the second term on the left hand side of \eqref{5.13}, it is straightforward to see by \eqref{5.12} that
\begin{align}
\label{5.18}	&\frac{1}{\varepsilon}|(\partial^{\alpha}_{\beta}\Gamma(f,\frac{M-\mu}{\sqrt{\mu}}),	w^{2|\beta|}\partial^{\alpha}_{\beta}f)|
\nonumber\\
&\leq C\frac{1}{\varepsilon}\sum_{\alpha_{1}\leq\alpha}\sum_{\beta'\leq\beta_{1}\leq\beta}\int_{\mathbb{R}^{3}}|\langle v\rangle^{-b}\partial^{\alpha_{1}}_{\beta^{'}}f|_{2}|\partial^{\alpha-\alpha_{1}}_{\beta-\beta_{1}}(\frac{M-\mu}{\sqrt{\mu}})|_{\sigma,|\beta|}|\partial^{\alpha}_{\beta}f|_{\sigma,|\beta|}\,dx.
\end{align}
The estimate \eqref{5.18} can be treated in the similar way as \eqref{5.15}. First note that $|\alpha_1|\leq|\alpha|\leq N-1$ due to $|\alpha|+|\beta|\leq N$ and $|\beta|\geq1$. If $|\alpha-\alpha_{1}|=0$,  we  use \eqref{2.26}, \eqref{5.16}, the Cauchy-Schwarz inequality and \eqref{3.8} again, to obtain
\begin{align*}
&\frac{1}{\varepsilon}
\int_{\mathbb{R}^{3}}|\langle v\rangle^{-b}\partial^{\alpha_{1}}_{\beta^{'}}f|_{2}|\partial^{\alpha-\alpha_{1}}_{\beta-\beta_{1}}(\frac{M-\mu}{\sqrt{\mu}})|_{\sigma,|\beta|}|\partial^{\alpha}_{\beta}f|_{\sigma,|\beta|}\,dx
\\
&\leq C(\eta_{0}+\varepsilon)\frac{1}{\varepsilon}
\|\langle v\rangle^{-b}\partial^{\alpha_{1}}_{\beta^{'}}f\|
\|\partial^{\alpha}_{\beta}f\|_{\sigma,|\beta|}
\\
&\leq C(\eta_{0}+\varepsilon)\frac{1}{\varepsilon}
\|\partial^{\alpha_{1}}_{\beta^{'}}f\|_{\sigma,|\beta^{'}|}
\|\partial^{\alpha}_{\beta}f\|_{\sigma,|\beta|}
\\
&\leq \eta\frac{1}{\varepsilon}\|\partial^{\alpha}_{\beta}f\|^{2}_{\sigma,|\beta|}+C_{\eta}(\eta_{0}+\varepsilon)^{2}\mathcal{D}_N(t).
\end{align*}
Here we have used $\langle v\rangle^{-b}\leq\langle v\rangle^{-\frac{1}{2}}\langle v\rangle^{-|\beta^{'}|}$
by choosing $b\geq N+1/2\geq|\beta^{'}|+ 1/2$.
If $|\alpha-\alpha_{1}|\neq0$, that is $|\alpha_{1}|<|\alpha|\leq N-1$, then $|\alpha_{1}|+|\beta^{'}|\leq|\alpha_{1}|+|\beta|
\leq|\alpha|-1+|\beta|\leq N-1$, and it holds that
\begin{align*}
&\frac{1}{\varepsilon}
\int_{\mathbb{R}^{3}}|\langle v\rangle^{-b}\partial^{\alpha_{1}}_{\beta^{'}}f|_{2}|\partial^{\alpha-\alpha_{1}}_{\beta-\beta_{1}}(\frac{M-\mu}{\sqrt{\mu}})|_{\sigma,|\beta|}|\partial^{\alpha}_{\beta}f|_{\sigma,|\beta|}\,dx
\\
&\leq C\frac{1}{\varepsilon}\big\||\langle v\rangle^{-b}\partial^{\alpha_{1}}_{\beta^{'}}f|_{2}\big\|_{L^{6}}
\big\||\partial^{\alpha-\alpha_{1}}_{\beta-\beta_{1}}(\frac{M-\mu}{\sqrt{\mu}})|_{\sigma,|\beta|}\big\|_{L^{3}}\|\partial^{\alpha}_{\beta}f\|_{\sigma,|\beta|}
\\
&\leq
\eta\frac{1}{\varepsilon}\|\partial^{\alpha}_{\beta}f\|^{2}_{\sigma,|\beta|}+C_{\eta}(\eta_{0}+\varepsilon^{\frac{1}{2}})^2\mathcal{D}_N(t).
\end{align*}
All in all, plugging the above estimates into \eqref{5.18}, yields
\begin{align}
\label{5.19}
\frac{1}{\varepsilon}|(\partial^{\alpha}_{\beta}\Gamma(f,\frac{M-\mu}{\sqrt{\mu}}),
w^{2|\beta|}\partial^{\alpha}_{\beta}f)|
\leq C\eta\frac{1}{\varepsilon}\|\partial^{\alpha}_{\beta}f\|^{2}_{\sigma,|\beta|}
+C_{\eta}(\eta_{0}+\varepsilon^{\frac{1}{2}})\mathcal{D}_N(t).
\end{align}
In summary, the desired estimate \eqref{5.13}
follows from  \eqref{5.17} and \eqref{5.19}.
This concludes the proof of \eqref{5.13}. By \eqref{5.11} and the similar arguments as \eqref{5.17} and \eqref{5.19}, we can prove that \eqref{5.14} holds and details are omitted for brevity. This ends the proof of Lemma \ref{lem5.7}.
\end{proof}

Next we consider the estimates of 
the nonlinear collision term $\Gamma(\frac{G}{\sqrt{\mu}},\frac{G}{\sqrt{\mu}})$, which will be used in Section \ref{seca.4}.

\begin{lemma}
	\label{lem5.8}
	Let $|\alpha|+|\beta|\leq N$ with $|\beta|\geq1$
	and the conditions of Lemma \ref{lem5.7} be satisfied,
	then for any $\eta>0$, one has
	\begin{align}
		\label{5.20}
		&\frac{1}{\varepsilon}|(\partial^{\alpha}_{\beta}\Gamma(\frac{G}{\sqrt{\mu}},\frac{G}{\sqrt{\mu}}),
		w^{2|\beta|}\partial^{\alpha}_{\beta}f)|
		\nonumber\\
		&\leq  C\eta\frac{1}{\varepsilon}\|\partial^{\alpha}_{\beta}f\|^{2}_{\sigma,|\beta|}
		+C_{\eta}(\eta_{0}+\varepsilon^{\frac{1}{2}})\mathcal{D}_N(t)+C_{\eta}(\eta_{0}+\varepsilon^{\frac{1}{2}})\varepsilon^2.
	\end{align}
	Moreover, for $|\beta|=0$ and  $|\alpha|\leq N-1$,  it holds that
	\begin{align}
		\label{5.21}
		\frac{1}{\varepsilon}|(\partial^{\alpha}\Gamma(\frac{G}{\sqrt{\mu}},\frac{G}{\sqrt{\mu}}),\partial^{\alpha}f)|
		\leq  C\eta\frac{1}{\varepsilon}\|\partial^{\alpha}f\|^{2}_{\sigma}+C_{\eta}(\eta_{0}+\varepsilon^{\frac{1}{2}})\mathcal{D}_N(t)+C_{\eta}(\eta_{0}+\varepsilon^{\frac{1}{2}})\varepsilon^2.
	\end{align}
\end{lemma}
\begin{proof}
	We only prove the estimate \eqref{5.20} since the estimate \eqref{5.21} can be handled in the same way. 
	Let $|\alpha|+|\beta|\leq N$ with $|\beta|\geq1$, we use $G=\overline{G}+\sqrt{\mu}f$ to show
	\begin{equation*}
		\partial^{\alpha}_{\beta}\Gamma(\frac{G}{\sqrt{\mu}},\frac{G}{\sqrt{\mu}})=\partial^{\alpha}_{\beta}\Gamma(\frac{\overline{G}}{\sqrt{\mu}},\frac{\overline{G}}{\sqrt{\mu}})
		+\partial^{\alpha}_{\beta}\Gamma(\frac{\overline{G}}{\sqrt{\mu}},f)+\partial^{\alpha}_{\beta}\Gamma(f,\frac{\overline{G}}{\sqrt{\mu}})+\partial^{\alpha}_{\beta}\Gamma(f,f).
	\end{equation*}
	We take the inner product of the above equality with $\frac{1}{\varepsilon}w^{2|\beta|}\partial^{\alpha}_{\beta}f$ and then
	compute each term. In view of \eqref{5.12}, Lemma \ref{lem5.4}, the Cauchy-Schwarz and Sobolev inequalities, \eqref{1.7} and \eqref{3.1}, we arrive at
	\begin{align}
		\label{5.23}
		&\frac{1}{\varepsilon}|(\partial^{\alpha}_{\beta}
		\Gamma(\frac{\overline{G}}{\sqrt{\mu}},\frac{\overline{G}}{\sqrt{\mu}}),w^{2|\beta|}\partial^{\alpha}_{\beta}f)|
		\nonumber\\
		&\leq C\frac{1}{\varepsilon}\sum_{\alpha_{1}\leq\alpha}\sum_{\beta'\leq\beta_{1}\leq\beta}\int_{\mathbb{R}^{3}}|\langle v\rangle^{-b}\partial^{\alpha_{1}}_{\beta^{'}}(\frac{\overline{G}}{\sqrt{\mu}})|_{2}|\partial^{\alpha-\alpha_{1}}_{\beta-\beta_{1}}(\frac{\overline{G}}{\sqrt{\mu}})|_{\sigma,|\beta-\beta_{1}|}|\partial^{\alpha}_{\beta}f|_{\sigma,|\beta|}\,dx
		\nonumber\\
		&\leq 
		\eta\frac{1}{\varepsilon}\|\partial^{\alpha}_{\beta}f\|^{2}_{\sigma,|\beta|}+C_{\eta}(\eta_{0}+\varepsilon^{\frac{1}{2}})\varepsilon^2.
	\end{align}
Using \eqref{5.12} again, it gives
\begin{align}
\label{5.23a}
&\frac{1}{\varepsilon}|(\partial^{\alpha}_{\beta}\Gamma(\frac{\overline{G}}{\sqrt{\mu}},f),w^{2|\beta|}\partial^{\alpha}_{\beta}f)|
\nonumber\\
&\leq  C\sum_{\alpha_{1}\leq\alpha}\sum_{\beta'\leq\beta_{1}\leq\beta}\underbrace{\frac{1}{\varepsilon}\int_{\mathbb{R}^{3}}|\langle v\rangle^{-b}\partial^{\alpha_{1}}_{\beta^{'}}(\frac{\overline{G}}{\sqrt{\mu}})|_{2}|\partial^{\alpha-\alpha_{1}}_{\beta-\beta_{1}}f|_{\sigma,|\beta-\beta_{1}|}|\partial^{\alpha}_{\beta}f|_{\sigma,|\beta|}\,dx}_{J_1}.
\end{align}
Note that for $|\alpha|+|\beta|\leq N$ with $|\beta|\geq1$, one has $|\alpha|\leq N-1$.
To estimate the term $J_1$, we consider the following two cases.
If $|\alpha-\alpha_{1}|+|\beta-\beta_{1}|\leq\frac{|\alpha|+|\beta|}{2}$, we derive from the Cauchy-Schwarz and Sobolev inequalities,
\eqref{5.5}, \eqref{1.7}, \eqref{3.1} and \eqref{3.8} that
\begin{align*}
J_1&\leq C\frac{1}{\varepsilon}\big\||\langle v\rangle^{-b}\partial^{\alpha_{1}}_{\beta^{'}}(\frac{\overline{G}}{\sqrt{\mu}})|_{2}\big\|_{L^2}\big\||\partial^{\alpha-\alpha_{1}}_{\beta-\beta_{1}}f|_{\sigma,|\beta-\beta_{1}|}\big\|_{L^\infty}\big\||\partial^{\alpha}_{\beta}f|_{\sigma,|\beta|}\big\|_{L^2}
\\
&\leq \eta\frac{1}{\varepsilon}\|\partial^{\alpha}_{\beta}f\|^{2}_{\sigma,|\beta|}+
C_\eta\eta_{0}\varepsilon\|\nabla_{x}\partial^{\alpha-\alpha_{1}}_{\beta-\beta_{1}}f\|_{\sigma,|\beta-\beta_{1}|}
\|\nabla^2_{x}\partial^{\alpha-\alpha_{1}}_{\beta-\beta_{1}}f\|_{\sigma,|\beta-\beta_{1}|}
\\	
&\leq \eta\frac{1}{\varepsilon}\|\partial^{\alpha}_{\beta}f\|^{2}_{\sigma,|\beta|}+C_{\eta}(\eta_{0}+\varepsilon^{\frac{1}{2}})\mathcal{D}_N(t).
\end{align*}
If $\frac{|\alpha|+|\beta|}{2}<|\alpha-\alpha_{1}|+|\beta-\beta_{1}|\leq |\alpha|+|\beta|$, then it holds that
\begin{align*}
J_1&\leq C\frac{1}{\varepsilon}\big\||\langle v\rangle^{-b}\partial^{\alpha_{1}}_{\beta^{'}}(\frac{\overline{G}}{\sqrt{\mu}})|_{2}\big\|_{L^\infty}\big\||\partial^{\alpha-\alpha_{1}}_{\beta-\beta_{1}}f|_{\sigma,|\beta-\beta_{1}|}\big\|_{L^2}\big\||\partial^{\alpha}_{\beta}f|_{\sigma,|\beta|}\big\|_{L^2}
\\
&\leq \eta\frac{1}{\varepsilon}\|\partial^{\alpha}_{\beta}f\|^{2}_{\sigma,|\beta|}+C_{\eta}(\eta_{0}+\varepsilon^{\frac{1}{2}})\mathcal{D}_N(t).
\end{align*}
Consequently, plugging the above estimates into \eqref{5.23a}, implies 
\begin{align}
\label{5.24}
\frac{1}{\varepsilon}|(\partial^{\alpha}_{\beta}\Gamma(\frac{\overline{G}}{\sqrt{\mu}},f),w^{2|\beta|}\partial^{\alpha}_{\beta}f)|
\leq 
C\eta\frac{1}{\varepsilon}\|\partial^{\alpha}_{\beta}f\|^{2}_{\sigma,|\beta|}+C_{\eta}(\eta_{0}+\varepsilon^{\frac{1}{2}})\mathcal{D}_N(t).
\end{align}
Carrying out the similar calculations as \eqref{5.24}, one has the following same bound
\begin{align}
\label{5.25a}
\frac{1}{\varepsilon}|(\partial^{\alpha}_{\beta}\Gamma(f,\frac{\overline{G}}{\sqrt{\mu}}),w^{2|\beta|}\partial^{\alpha}_{\beta}f)|
\leq 
C\eta\frac{1}{\varepsilon}\|\partial^{\alpha}_{\beta}f\|^{2}_{\sigma,|\beta|}+C_{\eta}(\eta_{0}+\varepsilon^{\frac{1}{2}})\mathcal{D}_N(t).
\end{align}
Similarly, one has from \eqref{5.12} that
\begin{align*}
&\frac{1}{\varepsilon}
|(\partial^{\alpha}_{\beta}\Gamma(f,f), w^{2|\beta|}\partial^{\alpha}_{\beta}f)|
\\
&\leq C\sum_{\alpha_{1}\leq\alpha}\sum_{\beta'\leq\beta_{1}\leq\beta}\underbrace{\frac{1}{\varepsilon}\int_{\mathbb{R}^{3}}|\langle v\rangle^{-b}\partial^{\alpha_{1}}_{\beta^{'}}f|_{2}|\partial^{\alpha-\alpha_{1}}_{\beta-\beta_{1}}f|_{\sigma,|\beta-\beta_{1}|}
|\partial^{\alpha}_{\beta}f|_{\sigma,|\beta|}\,dx}_{J_2}.
\end{align*}
The term $J_2$ can be treated in the similar way as the term $J_1$. 
If $|\alpha-\alpha_{1}|+|\beta-\beta_{1}|
\leq\frac{|\alpha|+|\beta|}{2}$, using the Cauchy-Schwarz and Sobolev inequalities,
\eqref{3.1} and \eqref{3.8}, we get
\begin{align*}
J_2&\leq C\frac{1}{\varepsilon}\big\||\langle v\rangle^{-b}\partial^{\alpha_{1}}_{\beta^{'}}f|_{2}\big\|_{L^2}
\big\||\partial^{\alpha-\alpha_{1}}_{\beta-\beta_{1}}f|_{\sigma,|\beta-\beta_{1}|}\big\|_{L^\infty}
\big\||\partial^{\alpha}_{\beta}f|_{\sigma,|\beta|}\big\|_{L^2}
\\
&\leq 
\eta\frac{1}{\varepsilon}\|\partial^{\alpha}_{\beta}f\|^{2}_{\sigma,|\beta|}+C_{\eta}
\frac{1}{\varepsilon}\|\partial^{\alpha_{1}}_{\beta^{'}}f\|_{2,|\beta^{'}|}^2
\|\nabla_{x}\partial^{\alpha-\alpha_{1}}_{\beta-\beta_{1}}f\|_{\sigma,|\beta-\beta_{1}|}\|\nabla^2_{x}\partial^{\alpha-\alpha_{1}}_{\beta-\beta_{1}}f\|_{\sigma,|\beta-\beta_{1}|}	
\\
&\leq \eta\frac{1}{\varepsilon}\|\partial^{\alpha}_{\beta}f\|^{2}_{\sigma,|\beta|}+C_{\eta}\varepsilon^{\frac{1}{2}}\mathcal{D}_N(t).
\end{align*}
Here we have used $\langle v\rangle^{-b}\leq\langle v\rangle^{-|\beta^{'}|}$ by choosing $b\geq N\geq |\beta^{'}|$.
If $\frac{|\alpha|+|\beta|}{2}<|\alpha-\alpha_{1}|+|\beta-\beta_{1}|\leq |\alpha|+|\beta|$, we also have
\begin{align*}
J_2&\leq C\frac{1}{\varepsilon}\big\||\langle v\rangle^{-b}\partial^{\alpha_{1}}_{\beta^{'}}f|_{2}\big\|_{L^\infty}
\big\||\partial^{\alpha-\alpha_{1}}_{\beta-\beta_{1}}f|_{\sigma,|\beta-\beta_{1}|}\big\|_{L^2}
\big\||\partial^{\alpha}_{\beta}f|_{\sigma,|\beta|}\big\|_{L^2}
\\
&\leq \eta\frac{1}{\varepsilon}\|\partial^{\alpha}_{\beta}f\|^{2}_{\sigma,|\beta|}+C_{\eta}\varepsilon^{\frac{1}{2}}\mathcal{D}_N(t).
\end{align*}
Therefore, with these estimates in hand, it follows that
\begin{align}
\label{5.25}
\frac{1}{\varepsilon}
|(\partial^{\alpha}_{\beta}\Gamma(f,f), w^{2|\beta|}\partial^{\alpha}_{\beta}f)|
\leq C \eta\frac{1}{\varepsilon}\|\partial^{\alpha}_{\beta}f\|^{2}_{\sigma,|\beta|}+C_{\eta}\varepsilon^{\frac{1}{2}}\mathcal{D}_N(t).
\end{align}
	
In summary, the desired estimate \eqref{5.20} follows from \eqref{5.23}, \eqref{5.24}, \eqref{5.25a} and \eqref{5.25}.
By the similar arguments, we can prove that \eqref{5.21} holds and we omit the details for brevity. This completes the proof of Lemma \ref{lem5.8}. 
\end{proof}

\subsection{Estimates on fluid quantities}
In what follows  we give the estimates of the fluid quantities involving the temporal derivatives, which will be used in Section \ref{seca.4}.
\begin{lemma}\label{lem5.10}
For $|\alpha|\leq N-1$, one has
\begin{align}
\label{5.26}
\|\partial^{\alpha}\partial_t(\widetilde{\rho},\widetilde{u},\widetilde{\theta})\|^{2}\leq &
C\|\partial^{\alpha}\nabla_{x}(\widetilde{\rho},\widetilde{u},\widetilde{\theta})\|^{2}+C\|\langle v\rangle^{-\frac{1}{2}}\partial^{\alpha}\nabla_{x} f\|^{2}+C(\eta_{0}+\varepsilon^{\frac{1}{2}})\varepsilon^{2}.
\end{align}
For $|\alpha|\leq N-2$, it also holds that
\begin{align}
\label{5.28a}
\|\partial^{\alpha}\partial_t(\widetilde{\rho},\widetilde{u},\widetilde{\theta})\|^{2}\leq &C\varepsilon^{2}.
\end{align}
\end{lemma}
\begin{proof}
Subtracting \eqref{1.4} from system \eqref{2.10} yields  that
\begin{align}
\label{2.22}
\left\{
\begin{array}{rl}
&\partial_{t}\widetilde{\rho}+u\cdot\nabla_{x}\widetilde{\rho}+\bar{\rho}\nabla_{x}\cdot\widetilde{u}
+\widetilde{u}\cdot\nabla_{x}\bar{\rho}
+\widetilde{\rho}\nabla_{x}\cdot u=0,
\\
&\partial_{t}\widetilde{u}+u\cdot\nabla_{x}\widetilde{u}
+\frac{2\bar{\theta}}{3\bar{\rho}}\nabla_{x}\widetilde{\rho}
+\frac{2}{3}\nabla_{x}\widetilde{\theta}
+\widetilde{u}\cdot\nabla_{x}\bar{u}+\frac{2}{3}(\frac{\theta}{\rho}-\frac{\bar{\theta}}{\bar{\rho}})
\nabla_{x}\rho=-\frac{1}{\rho}\int_{\mathbb{R}^{3}} v\otimes v\cdot\nabla_{x} G\,dv,
\\
&\partial_{t}\widetilde{\theta}+u\cdot\nabla_{x}\widetilde{\theta}+\frac{2}{3}\bar{\theta}\nabla_{x}\cdot \widetilde{u}
+\widetilde{u}\cdot\nabla_{x}\bar{\theta}+\frac{2}{3}\widetilde{\theta}\nabla_{x}\cdot u
=-\frac{1}{\rho}\int_{\mathbb{R}^{3}} \frac{1}{2}|v|^{2} v\cdot\nabla_{x} G\,dv\\
&\qquad\qquad\qquad\qquad\qquad\qquad\qquad\qquad\qquad\qquad\qquad+\frac{1}{\rho}u\cdot\int_{\mathbb{R}^{3}} v\otimes v\cdot\nabla_{x}G\,dv.
\end{array} \right.
\end{align}	
Applying $\partial^{\alpha}$ with $|\alpha|\leq N-1$ to the second equation of \eqref{2.22} and taking the inner product of the resulting equation with  $\partial^{\alpha}\partial_t\widetilde{u}$, we can arrive at
\begin{multline*}
(\partial^{\alpha}\partial_t\widetilde{u},\partial^{\alpha}\partial_t\widetilde{u})=
-(\partial^{\alpha}[u\cdot\nabla_{x}\widetilde{u}
+\frac{2\bar{\theta}}{3\bar{\rho}}\nabla_{x}\widetilde{\rho}
+\frac{2}{3}\nabla_{x}\widetilde{\theta}
+\widetilde{u}\cdot\nabla_{x}\bar{u}+\frac{2}{3}(\frac{\theta}{\rho}-\frac{\bar{\theta}}{\bar{\rho}})
\nabla_{x}\rho],\partial^{\alpha}\partial_t\widetilde{u})
\nonumber\\
-(\partial^{\alpha}(\frac{1}{\rho}\int_{\mathbb{R}^{3}} v\otimes v\cdot\nabla_{x} G\,dv),\partial^{\alpha}\partial_t\widetilde{u})
\nonumber\\
\leq C\eta\|\partial^{\alpha}\partial_t\widetilde{u}\|^{2}+
C_\eta\|\partial^{\alpha}(\nabla_{x}\widetilde{\rho},\nabla_{x}\widetilde{u},\nabla_{x}\widetilde{\theta})\|^{2}+C_\eta\|\langle v\rangle^{-\frac{1}{2}}\partial^{\alpha}\nabla_{x} f\|^{2}+C_\eta(\eta_{0}+\varepsilon^{\frac{1}{2}})\varepsilon^{2}.
\end{multline*}
Similar estimates also hold for $\partial^{\alpha}\partial_t\widetilde{\rho}$ and $\partial^{\alpha}\partial_t\widetilde{\theta}$. Therefore, choosing sufficiently small $\eta>0$,
we can obtain the desired estimate \eqref{5.26}.
If $|\alpha|\leq N-2$, we can deduce from \eqref{3.1} that
\begin{align*}
\|\partial^{\alpha}(\nabla_{x}\widetilde{\rho},\nabla_{x}\widetilde{u},\nabla_{x}\widetilde{\theta})\|^{2}+\|\langle v\rangle^{-\frac{1}{2}}\partial^{\alpha}\nabla_{x} f\|^{2}\leq C\varepsilon^{2}.
\end{align*}
This and \eqref{5.26} together gives \eqref{5.28a}
by using the smallness of $\eta_{0}$ and $\varepsilon$. Thus the proof of Lemma \ref{lem5.10} is completed.
\end{proof}

\section{A priori estimates}\label{seca.4}
This section is a core part to make preparations for the proof of the main results. We shall obtain the desired a priori estimate \eqref{3.21A} on the solution step by step in a series of lemmas in order to close the a priori assumption \eqref{3.1}. In all lemmas below, $F=M+\overline{G}+\sqrt{\mu}f$ is assumed to be the smooth solution to the Landau equation \eqref{1.1} and \eqref{1.10} for $t\in[0,T]$ with $T\in(0,\tau]$ and all derived estimates are satisfied for any $0\leq t\leq T$. In the meantime, we assume that \eqref{3.1} and \eqref{3.6} are valid.

\subsection{Zero-order estimate on fluid part}\label{sec5.1}
We start from the zeroth order energy estimates of the fluid part
	$(\widetilde{\rho},\widetilde{u},\widetilde{\theta})$  by the entropy and entropy flux motivated in \cite{Liu-Yang-Yu}. We show that the energy and energy dissipation for the fluid part are bounded by the dissipation up to first order for the non-fluid part. The proof also makes use of the dissipation mechanism for the Navier-Stokes-type equations \eqref{2.14} as well as the Euler-type equations \eqref{2.10}. 
	
\begin{lemma}\label{lem.zefp}
It holds that
\begin{align}
\label{4.22}
&\|(\widetilde{\rho},\widetilde{u},\widetilde{\theta})(t)\|^{2}+c\varepsilon\int^{t}_0
\|\nabla_x(\widetilde{\rho},\widetilde{u},\widetilde{\theta})(s)\|^{2}\,ds
\nonumber\\
&\leq C\varepsilon\int^{t}_0(\|f(s)\|_{\sigma}^{2}
+\|\nabla_{x}f(s)\|_{\sigma}^{2})\,ds+C(1+t)(\eta_{0}+\varepsilon)\varepsilon^{2}.
\end{align}
\end{lemma}
\begin{proof}
As in \cite{Liu-Yang-Yu}, we define the macroscopic entropy $S$ by
\begin{equation}
\label{4.3b}
-\frac{3}{2}\rho S:=\int_{\mathbb{R}^{3}}M\ln M\,dv.
\end{equation}
By plugging \eqref{2.3} and further taking integration, it follows that
\begin{equation}
\label{4.3}
S=-\frac{2}{3}\ln\rho+\ln(2\pi R\theta)+1,\quad
p=R\rho\theta=\frac{1}{2\pi e}\rho^{\frac{5}{3}}\exp (S),\quad R=\frac{2}{3}.
\end{equation}
Multiplying \eqref{2.8} by $\ln M$  and integrating over $v$, direct computations give
\begin{align*}
-\frac{3}{2}\partial_t(\rho S)-\frac{3}{2}\nabla_{x}\cdot(\rho uS)
+\nabla_{x}\cdot\int_{\mathbb{R}^{3}} vG\ln M\,dv=\int_{\mathbb{R}^{3}} \frac{G v\cdot\nabla_{x} M}{M}\,dv.
\end{align*}
We denote
\begin{align*}
m:=&(m_{0},m_{1},m_{2},m_{3},m_{4})^{t}=(\rho,\rho u_{1},\rho u_{2},\rho u_{3},\rho(\theta+\frac{|u|^{2}}{2}))^{t},
\\
n:=&(n_{0},n_{1},n_{2},n_{3},n_{4})^{t}
=(\rho u,\rho uu_{1}+p\mathbb{I}_{1},\rho uu_{2}+p\mathbb{I}_{2},\rho uu_{3}+p\mathbb{I}_{3},\rho u(\theta+\frac{|u|^{2}}{2})+pu)^{t},
\end{align*}
where $\mathbb{I}_{1}=(1,0,0)^{t}$, $\mathbb{I}_{2}=(0,1,0)^{t}$, $\mathbb{I}_{3}=(0,0,1)^{t}$
and $(\cdot,\cdot,\cdot)^{t}$ is the transpose of a row vector.
Then the conservation law \eqref{2.14} can be rewritten as
\begin{equation*}
\partial_tm+\nabla_{x}\cdot n=\begin{pmatrix}
0
\\
\varepsilon\sum^{3}_{j=1}\partial_{x_{j}}[\mu(\theta)D_{1j}]
-\int_{\mathbb{R}^{3}} v_{1}(v\cdot\nabla_{x}L^{-1}_{M}\Theta)\,dv
\\
\varepsilon\sum^{3}_{j=1}\partial_{x_{j}}[\mu(\theta)D_{2j}]
-\int_{\mathbb{R}^{3}} v_{2}(v\cdot\nabla_{x}L^{-1}_{M}\Theta)\,dv
\\
\varepsilon\sum^{3}_{j=1}\partial_{x_{j}}[\mu(\theta)D_{3j}]
-\int_{\mathbb{R}^{3}} v_{3}(v\cdot\nabla_{x}L^{-1}_{M}\Theta)\,dv
\\
\varepsilon\nabla_{x}\cdot(\kappa(\theta)\nabla_{x}\theta)
+\varepsilon\nabla_{x}\cdot[\mu(\theta) u\cdot D]-\int_{\mathbb{R}^{3}} \frac{1}{2}|v|^{2} v\cdot\nabla_{x}L^{-1}_{M}\Theta\,dv
\end{pmatrix},
\end{equation*}
where $D=[D_{ij}]_{1\leq i,j\leq 3}$ is given in \eqref{def.vst}.
Define a relative entropy-entropy flux pair $(\eta,q)(t,x)$ around the local Maxwellian
$\overline{M}=M_{[\bar{\rho},\bar{u},\bar{S}]}$ 
with $\bar{S}=-\frac{2}{3}\ln\bar{\rho}+\ln(2\pi R\bar{\theta})+1$ as
\begin{equation}
\label{4.5}
\left\{
\begin{array}{rl}
&\eta(t,x)=\bar{\theta}\{-\frac{3}{2}\rho S+\frac{3}{2}\bar{\rho}\bar{S}+\frac{3}{2}\nabla_{m}(\rho S)|_{m=\bar{m}}(m-\bar{m})\},
\\
&q_{j}(t,x)=\bar{\theta}\{-\frac{3}{2}\rho u_{j}S+\frac{3}{2}\bar{\rho}\bar{u}_{j}\bar{S}+\frac{3}{2}\nabla_{m}(\rho S)|_{m=\bar{m}}\cdot(n_{j}-\bar{n}_{j})\},\quad j=1,2,3.
\end{array} \right.
\end{equation}
Here $\bar{m}=(\bar{\rho},\bar{\rho}\bar{ u}_{1},\bar{\rho}\bar{ u}_{2},\bar{\rho}\bar{ u}_{3},\bar{\rho}(\bar{\theta}
+\frac{1}{2}|\bar{u}|^{2}))^{t}$. Since
\begin{equation}
\label{4.6}
\partial_{m_{0}}(\rho S)=S+\frac{|u|^{2}}{2\theta}-\frac{5}{3}, \quad
\partial_{m_{i}}(\rho S)=-\frac{u_{i}}{\theta}, ~~~i=1,2,3, \quad \partial_{m_{4}}(\rho S)=\frac{1}{\theta},
\end{equation}
an elementary calculation leads to
\begin{equation}\notag
\left\{
\begin{array}{rl}
\eta(t,x)&=\frac{3}{2}\{\rho\theta-\bar{\theta}\rho S+\rho[(\bar{S}-\frac{5}{3})\bar{\theta}
+\frac{|u-\bar{u}|^{2}}{2}]+\frac{2}{3}\bar{\rho}\bar{\theta}\}
\\
&=\rho\bar{\theta}\Psi(\frac{\bar{\rho}}{\rho})+\frac{3}{2}\rho\bar{\theta}\Psi(\frac{\theta}{\bar{\theta}})
+\frac{3}{4}\rho|u-\bar{u}|^{2},
\\
q_{j}(t,x)&=u_{j}\eta(t,x)+(u_{j}-\bar{u}_{j})(\rho\theta-\bar{\rho}\bar{\theta}), \quad j=1,2,3,
\end{array} \right.
\end{equation}
where $\Psi(s)=s-\ln s-1$ is a strictly convex function around $s=1$. By these facts and \eqref{2.15}, for $\bf{X}$ in any closed
bounded region in $\sum=\{{\bf{X}}:\rho>0,~~\theta>0\}$,  there exists a constant $C>1$ such that
\begin{equation}
\label{4.8}
C^{-1}|(\widetilde{\rho},\widetilde{u},\widetilde{\theta})|^{2}\leq \eta(t,x)\leq C|(\widetilde{\rho},\widetilde{u},\widetilde{\theta})|^{2}.
\end{equation}

Using \eqref{4.5} and making a direct calculation, it holds that
\begin{multline}
\label{4.9}
\partial_{t}\eta(t,x)+\nabla_{x}\cdot q(t,x)
=\nabla_{[\bar{\rho},\bar{u},\bar{S}]}\eta(t,x)\cdot \partial_t(\bar{\rho},\bar{u},\bar{S})
+\sum^{3}_{j=1}\nabla_{[\bar{\rho},\bar{u},\bar{S}]}q_{j}(t,x)\cdot \partial_{x_j}(\bar{\rho},\bar{u},\bar{S})\\
+\bar{\theta}\{-\frac{3}{2}\partial_t(\rho S)-\frac{3}{2}\nabla_{x}\cdot(\rho uS)\}
+\frac{3}{2}\bar{\theta}\nabla_{m}(\rho S)|_{m=\bar{m}}(\partial_tm+\nabla_{x}\cdot n).
\end{multline}
In order to further estimate \eqref{4.9}, we first claim that the following identity holds
\begin{align}
\label{4.10}
&\nabla_{[\bar{\rho},\bar{u},\bar{S}]}\eta(t,x)\cdot \partial_t(\bar{\rho},\bar{u},\bar{S})
+\sum^{3}_{j=1}\nabla_{[\bar{\rho},\bar{u},\bar{S}]}q_{j}(t,x)\cdot \partial_{x_j}(\bar{\rho},\bar{u},\bar{S})
\nonumber\\
=&-\frac{3}{2}\rho\widetilde{u}\cdot(\widetilde{u}\cdot\nabla_{x}\bar{u})
-\frac{2}{3}\rho\bar{\theta}(\nabla_{x}\cdot\bar{u})\Psi(\frac{\bar{\rho}}{\rho})
-\rho\bar{\theta}(\nabla_{x}\cdot\bar{u})\Psi(\frac{\theta}{\bar{\theta}})
\nonumber\\
&-\frac{3}{2}\rho\nabla_{x}\bar{\theta}\cdot\widetilde{u}(\frac{2}{3}
\ln\frac{\bar{\rho}}{\rho}+\ln\frac{\theta}{\bar{\theta}}),
\end{align}
whose proof is given in the Appendix in Section \ref{seca.5}. 
Note that
$$
\partial_tS=-\frac{2}{3}\frac{1}{\rho}\partial_t\rho+\frac{1}{\theta}\partial_t\theta,\quad
\nabla_{x} S=-\frac{2}{3}\frac{1}{\rho}\nabla_{x}\rho+\frac{1}{\theta}\nabla_{x}\theta,
$$
due to \eqref{4.3},
then we deduce from  this and \eqref{2.14}  that
\begin{align*}
&\bar{\theta}\{-\frac{3}{2}\partial_t(\rho S)-\frac{3}{2}\nabla_{x}\cdot(\rho uS)\}
=\bar{\theta}\{-\frac{3}{2}\frac{\rho}{\theta}(\partial_t\theta+u\cdot\nabla_{x}\theta)-\rho\nabla_{x}\cdot u\}
\\
&=-\frac{3}{2}\frac{\bar{\theta}}{\theta}\varepsilon\Big\{\sum^{3}_{j=1}\partial_{x_j}(\kappa(\theta)\partial_{x_{j}}\theta)+
\sum^{3}_{i,j=1}(\mu(\theta) \partial_{x_j}u_{i}D_{ij})\Big\}
\\
&\hspace{0.5cm}+\frac{3}{2}\frac{\bar{\theta}}{\theta}\Big\{\int_{\mathbb{R}^{3}}\frac{1}{2}|v|^{2}v\cdot\nabla_{x}
L^{-1}_{M}\Theta\, dv-\sum^{3}_{i=1}u_{i}\int_{\mathbb{R}^{3}} v_{i}v\cdot\nabla_{x}L^{-1}_{M}\Theta\, dv\Big\}.
\end{align*}
In view of the conservation law and the similar arguments as \eqref{4.6}, one has
\begin{align*}
&\frac{3}{2}\bar{\theta}\nabla_{m}(\rho S)|_{m=\bar{m}}(\partial_tm+\nabla_{x}\cdot n)
\\
=&\frac{3}{2}\sum^{3}_{i=1}\bar{u}_{i}\int_{\mathbb{R}^{3}} v_{i}v\cdot\nabla_{x}L^{-1}_{M}\Theta\,dv
-\frac{3}{2}\int_{\mathbb{R}^{3}} \frac{1}{2}|v|^{2} v\cdot\nabla_{x}L^{-1}_{M}\Theta\,dv
+\frac{3}{2}\varepsilon\nabla_{x}\cdot(\kappa(\theta)\nabla_{x}\theta)
\\
&-\frac{3}{2}\varepsilon\sum^{3}_{i,j=1}\bar{u}_{i}\partial_{x_j}[\mu(\theta)D_{ij}]
+\frac{3}{2}\varepsilon\nabla_{x}\cdot[\mu(\theta) u\cdot D].
\end{align*}
Therefore, plugging these estimates into \eqref{4.9}  and making a direct computation, we get
\begin{align}
\label{4.11}
&\partial_{t}\eta(t,x)+\nabla_{x}\cdot q(t,x)
+\varepsilon\frac{3\bar{\theta}}{2\theta}\mu(\theta)\sum^{3}_{i,j=1} \partial_{x_j}\widetilde{u}_{i}(\partial_{x_j}\widetilde{u}_{i}+\partial_{x_i}\widetilde{u}_{j}-\frac{2}{3}\delta_{ij}\nabla_{x}\cdot \widetilde{u})
+\varepsilon\frac{3\bar{\theta}}{2\theta^{2}}\kappa(\theta)|\nabla_{x}\widetilde{\theta}|^{2}
\nonumber\\
=&\frac{3}{2}\varepsilon\nabla_{x}\cdot[\mu(\theta)\widetilde{u}\cdot D] +\frac{3}{2}\varepsilon\nabla_{x}\cdot(\frac{\widetilde{\theta}}{\theta}\kappa(\theta)\nabla_{x}\theta)
-\frac{3}{2}\rho\widetilde{u}\cdot(\widetilde{u}\cdot\nabla_{x}\bar{u})
\nonumber\\
&-\frac{2}{3}\rho\bar{\theta}(\nabla_{x}\cdot\bar{u})\Psi(\frac{\bar{\rho}}{\rho})
-\rho\bar{\theta}(\nabla_{x}\cdot\bar{u})\Psi(\frac{\theta}{\bar{\theta}})
-\frac{3}{2}\rho\nabla_{x}\bar{\theta}\cdot\widetilde{u}(\frac{2}{3} \ln\frac{\bar{\rho}}{\rho}+\ln\frac{\theta}{\bar{\theta}})
\nonumber\\
&-\varepsilon\frac{3\bar{\theta}}{2\theta^{2}}\kappa(\theta)
\nabla_{x}\widetilde{\theta}\cdot\nabla_{x}\bar{\theta} +\varepsilon\frac{3\widetilde{\theta}}{2\theta^{2}}\kappa(\theta)\nabla_{x}\bar{\theta}\cdot\nabla_{x}\theta          +\varepsilon\frac{3\widetilde{\theta}}{2\theta}\sum^{3}_{i,j=1}[\partial_{x_j}\bar{u}_{i}\mu(\theta)D_{ij}]
\nonumber\\
&-\varepsilon\frac{3\bar{\theta}}{2\theta}\sum^{3}_{i,j=1}\mu(\theta) \partial_{x_j}\widetilde{u}_{i}(\partial_{x_j}\bar{u}_{i}+\partial_{x_i}\bar{u}_{j}-\frac{2}{3}\delta_{ij}\nabla_{x}\cdot \bar{u})
+I_{1}.
\end{align}
Here $I_{1}$ is given as
\begin{align}
\label{4.12}
I_{1}
=&-\frac{3}{2}\nabla_{x}\cdot(\frac{\widetilde{\theta}}{\theta}\int_{\mathbb{R}^{3}}(\frac{1}{2}|v|^{2}-u\cdot v)vL^{-1}_{M}\Theta\, dv)
-\frac{3}{2}\nabla_{x}\cdot(\sum^{3}_{i=1}\widetilde{u}_{i}\int_{\mathbb{R}^{3}} v_{i}vL^{-1}_{M}\Theta\,dv)
\nonumber\\
&+\frac{3}{2}\nabla_{x}(\frac{\widetilde{\theta}}{\theta})\cdot\int_{\mathbb{R}^{3}}(\frac{1}{2}|v|^{2}-v\cdot u)vL^{-1}_{M}\Theta\, dv
+\frac{3}{2}\frac{\bar{\theta}}{\theta}\sum^{3}_{i=1}\nabla_{x}\widetilde{u}_{i}\cdot\int_{\mathbb{R}^{3}} v_{i}vL^{-1}_{M}\Theta\,dv
\nonumber\\
&-\frac{3}{2}\frac{\widetilde{\theta}}{\theta}\sum^{3}_{i=1}\nabla_{x}\bar{u}_{i}\cdot\int_{\mathbb{R}^{3}} v_{i}vL^{-1}_{M}\Theta\,dv.
\end{align}

To complete the estimate \eqref{4.11}, we integrate it with respect to $x$ over $\mathbb{R}^{3}$ and then compute the resulting equation for each term.
Thanks to \eqref{2.15} and \eqref{3.1}, one has
$$
\Psi(\frac{\bar{\rho}}{\rho})\approx|\widetilde{\rho}|^{2},\quad
\Psi(\frac{\theta}{\bar{\theta}})\approx|\widetilde{\theta}|^{2},\quad
|\ln\frac{\bar{\rho}}{\rho}|\approx|\widetilde{\rho}|,\quad
|\ln\frac{\theta}{\bar{\theta}}|\approx|\widetilde{\theta}|.
$$
By taking the $L^2-L^2-L^\infty$ estimate, we get from
this, \eqref{3.6}, \eqref{1.7}, \eqref{3.1} and the imbedding inequality that
\begin{align}
\label{4.14b}
&\int_{\mathbb{R}^{3}}|\frac{3}{2}\rho\widetilde{u}\cdot(\widetilde{u}\cdot\nabla_{x}\bar{u})
+\frac{2}{3}\rho\bar{\theta}(\nabla_{x}\cdot\bar{u})\Psi(\frac{\bar{\rho}}{\rho})
+\rho\bar{\theta}(\nabla_{x}\cdot\bar{u})\Psi(\frac{\theta}{\bar{\theta}})
+\frac{3}{2}\rho\nabla_{x}\bar{\theta}\cdot\widetilde{u}(\frac{2}{3}
\ln\frac{\bar{\rho}}{\rho}+\ln\frac{\theta}{\bar{\theta}})|\,dx
\nonumber\\
&\leq C(\|\nabla_{x}\bar{u}\|_{L^{\infty}}+\|\nabla_{x}\bar{\theta}\|_{L^{\infty}})
\|(\widetilde{\rho},\widetilde{u},\widetilde{\theta})\|^{2}
\leq C\eta_{0}\varepsilon^{2}.
\end{align}
Since both $\mu(\theta)$ and $\kappa(\theta)$ are smooth functions of $\theta$, there exists a constant $C>1$ 
such that $\mu(\theta),\kappa(\theta)\in[C^{-1},C]$.
It follows from this, \eqref{3.6}, \eqref{1.7} and \eqref{3.1} that
\begin{align*}
&\int_{\mathbb{R}^{3}}(|\varepsilon\frac{3\bar{\theta}}{2\theta^{2}}\kappa(\theta)
\nabla_{x}\widetilde{\theta}\cdot\nabla_{x}\bar{\theta}|+
|\varepsilon\frac{3\widetilde{\theta}}{2\theta^{2}}\kappa(\theta)\nabla_{x}\bar{\theta}\cdot\nabla_{x}\theta|)\,dx
\\
&\leq C\varepsilon\|\nabla_{x}\widetilde{\theta}\|\|\nabla_{x}\bar{\theta}\|
+C\varepsilon\|\nabla_{x}\bar{\theta}\|_{L^{\infty}}\|\widetilde{\theta}\|\|\nabla_{x}\theta\|
\leq C(\eta_{0}+\varepsilon)\varepsilon^{2}.
\end{align*}
Recall $D_{ij}$ in \eqref{def.vst}, then the similar calculations as the above lead us to
\begin{align*}
&\int_{\mathbb{R}^{3}}|\varepsilon\frac{3\widetilde{\theta}}{2\theta}\sum^{3}_{i,j=1}[\partial_{x_j}\bar{u}_{i}\mu(\theta)D_{ij}]|\,dx
+\int_{\mathbb{R}^{3}}|\varepsilon\frac{3\bar{\theta}}{2\theta}\sum^{3}_{i,j=1}\mu(\theta)
\partial_{x_j}\widetilde{u}_{i}(\partial_{x_j}\bar{u}_{i}+\partial_{x_i}\bar{u}_{j}-\frac{2}{3}\delta_{ij}\nabla_{x}\cdot \bar{u})|\,dx
\nonumber\\
&\leq C(\eta_{0}+\varepsilon)\varepsilon^{2}.
\end{align*}
Applying the above estimates together with \eqref{lem.I1dx.bdd} whose proof will be postponed to Lemma \ref{lem.I1dx} later, we have from \eqref{4.11} that
\begin{align}
\label{4.20}
\frac{d}{dt}\int_{\mathbb{R}^3}\eta(t,x)\,dx
-\frac{d}{dt}E(t)
+c\varepsilon(\|\nabla_x\widetilde{u}\|^2+\|\nabla_x\widetilde{\theta}\|^{2})
\leq C\varepsilon\|f\|_{\sigma}^{2}+C(\eta_{0}+\varepsilon)\varepsilon^{2}.
\end{align}
Here the following crucial estimate has been  used
\begin{align*}
&\varepsilon\int_{\mathbb{R}^{3}}\frac{3\bar{\theta}}{2\theta}\mu(\theta)\sum^{3}_{i,j=1} \partial_{x_j}\widetilde{u}_{i}(\partial_{x_j}\widetilde{u}_{i}+\partial_{x_i}\widetilde{u}_{j}-\frac{2}{3}\delta_{ij}\nabla_{x}\cdot \widetilde{u})\,dx
\\
=&\varepsilon\sum^{3}_{i,j=1}\int_{\mathbb{R}^{3}}\frac{3\bar{\theta}}{2\theta}\mu(\theta) (\partial_{x_j}\widetilde{u}_{i})^2\,dx
+
\varepsilon\int_{\mathbb{R}^{3}}\frac{3\bar{\theta}}{2\theta}\mu(\theta)\frac{1}{3}(\nabla_{x}\cdot \widetilde{u})^2\,dx
\\
&-\varepsilon\sum^{3}_{i,j=1}\int_{\mathbb{R}^{3}}
\partial_{x_i}(\frac{3\bar{\theta}}{2\theta}\mu(\theta))\partial_{x_j}\widetilde{u}_{i}\widetilde{u}_{j}\,dx
+\varepsilon\sum^{3}_{i,j=1}\int_{\mathbb{R}^{3}}
\partial_{x_j}(\frac{3\bar{\theta}}{2\theta}\mu(\theta)) \partial_{x_i}\widetilde{u}_{i}\widetilde{u}_{j}\,dx
\\
\geq&c\varepsilon\|\nabla_x\widetilde{u}\|^2
-C(\eta_{0}+\varepsilon)\varepsilon^{2}.
\end{align*}
Recall $E(t)$ in \eqref{5.17A}. Then we employ \eqref{4.16}, \eqref{1.7} and \eqref{3.1} to obtain
\begin{align*}
|E(t)|&\leq C\varepsilon\{\|\nabla_x(\frac{\widetilde{\theta}}{\theta})\|\|f\|+
\|\nabla_x\widetilde{u}\|\|f\|+\|\nabla_x\bar{u}\|_{L^{\infty}}\|\widetilde{\theta}\|\|f\|\}
\\
&\leq C(\eta_{0}+\varepsilon)\varepsilon^{2}.
\end{align*}
Integrating \eqref{4.20} with respect to $t$ and using \eqref{3.5}, \eqref{4.8} and
the above estimate on $E(t)$, we have
\begin{align}
\label{5.13A}
&\|(\widetilde{\rho},\widetilde{u},\widetilde{\theta})(t)\|^{2}+c\varepsilon\int^{t}_0
\|\nabla_x(\widetilde{u},\widetilde{\theta})(s)\|^{2}\,ds
\nonumber\\
&\leq C\varepsilon\int^{t}_0\|f(s)\|_{\sigma}^{2}\,ds+C(1+t)(\eta_{0}+\varepsilon)\varepsilon^{2}.
\end{align}

\begin{remark}
Instead of the above proof based on the entropy-entropy flux method, one can also simply use the usual $L^2$ energy method similar to the one in the
space derivative estimate on the fluid part in Lemma \ref{lem.sdefpN1} in order
to get the same estimate as \eqref{5.13A}, since the estimate is restricted to only the finite time interval $[0,\tau]$.
\end{remark}

It should be noted that there are no dissipation terms for density function  in \eqref{5.13A}. For this, we turn to the Euler-type equations \eqref{2.10}. Recall that the difference system  \eqref{2.22} is derived from subtraction of \eqref{2.10} and \eqref{1.4}. We then take the inner product of the second equation of \eqref{2.22} with $\nabla_{x}\widetilde{\rho}$ to get
\begin{align*}
(\frac{2\bar{\theta}}{3\bar{\rho}}\nabla_{x}\widetilde{\rho},\nabla_{x}\widetilde{\rho})
=&-(\partial_{t}\widetilde{u},\nabla_{x}\widetilde{\rho})-(u\cdot\nabla_{x}\widetilde{u}
+\frac{2}{3}\nabla_{x}\widetilde{\theta}
,\nabla_{x}\widetilde{\rho})
\\
&-(\widetilde{u}\cdot\nabla_{x}\bar{u}+\frac{2}{3}(\frac{\theta}{\rho}-\frac{\bar{\theta}}{\bar{\rho}})\nabla_{x}\rho,\nabla_{x}\widetilde{\rho})
-(\frac{1}{\rho}\int_{\mathbb{R}^{3}} v\otimes v\cdot\nabla_{x} G\,dv,\nabla_{x}\widetilde{\rho}).
\end{align*}
Using the integration by parts and the first equation of \eqref{2.22}, one gets
\begin{align*}
-(\partial_{t}\widetilde{u},\nabla_{x}\widetilde{\rho})=&-\frac{d}{dt}(\widetilde{u},\nabla_{x}\widetilde{\rho})-(\nabla_{x}\widetilde{u},
\partial_t\widetilde{\rho})
\\
=&-\frac{d}{dt}(\widetilde{u},\nabla_{x}\widetilde{\rho})+(\nabla_{x}\widetilde{u},
u\cdot\nabla_{x}\widetilde{\rho}+\bar{\rho}\nabla_{x}\cdot\widetilde{u}
+\widetilde{u}\cdot\nabla_{x}\bar{\rho}
+\widetilde{\rho}\nabla_{x}\cdot u)
\\
\leq&-\frac{d}{dt}(\widetilde{u},\nabla_{x}\widetilde{\rho})+C\eta\|\nabla_{x}\widetilde{\rho}\|^2
+C_\eta\|\nabla_{x}\widetilde{u}\|^2+
C(\eta_0+\varepsilon)\varepsilon^2,
\end{align*}
where in the last inequality we have used the Cauchy-Schwarz inequality, \eqref{3.6}, \eqref{1.7} and \eqref{3.1}.
Likewise, we have
\begin{align*}
&|(u\cdot\nabla_{x}\widetilde{u}
+\frac{2}{3}\nabla_{x}\widetilde{\theta}
,\nabla_{x}\widetilde{\rho})|+|(\widetilde{u}\cdot\nabla_{x}\bar{u}+\frac{2}{3}(\frac{\theta}{\rho}-\frac{\bar{\theta}}{\bar{\rho}})\nabla_{x}\rho,\nabla_{x}\widetilde{\rho})|
\\
&\leq C\eta\|\nabla_{x}\widetilde{\rho}\|^2
+C_\eta(\|\nabla_{x}\widetilde{u}\|^2
+\|\nabla_{x}\widetilde{\theta}\|^2)
+C_\eta(\eta_0+\varepsilon)\varepsilon^2.
\end{align*}
On the other hand, it holds by the fact $G=\overline{G}+\sqrt{\mu}f$ and \eqref{5.5} that
\begin{align*}
|(\frac{1}{\rho}\int_{\mathbb{R}^{3}} v\otimes v\cdot\nabla_{x} G\,dv,\nabla_{x}\widetilde{\rho})|
\leq C\eta\|\nabla_{x}\widetilde{\rho}\|^2
+C_\eta\|\nabla_{x}f\|_{\sigma}^{2}
+C_\eta(\eta_0+\varepsilon)\varepsilon^2.
\end{align*}
Collecting the above estimates and taking a small constant $\eta>0$, one has
\begin{align}
\label{4.21}
\frac{d}{dt}\varepsilon(\widetilde{u},\nabla_{x}\widetilde{\rho})+c\varepsilon\|\nabla_{x}\widetilde{\rho}\|^2
\leq C\varepsilon(\|\nabla_{x}\widetilde{u}\|^2
+\|\nabla_{x}\widetilde{\theta}\|^2+\|\nabla_{x}f\|_{\sigma}^{2})+C(\eta_0+\varepsilon)\varepsilon^3.
\end{align}
Integrating \eqref{4.21} with respect to $t$ and using the estimate
$$
\varepsilon|(\widetilde{u},\nabla_{x}\widetilde{\rho})|\leq C\varepsilon\|\widetilde{u}\|\|\nabla_{x}\widetilde{\rho}\|
\leq C\varepsilon^3,
$$
as well as \eqref{3.5}, we have
\begin{align}
\label{5.15A}
\varepsilon\int^{t}_0\|\nabla_{x}\widetilde{\rho}(s)\|^2\,ds
\leq C\varepsilon\int^{t}_0\{
\|\nabla_x(\widetilde{u},\widetilde{\theta})(s)\|^{2}+\|\nabla_{x}f\|_{\sigma}^{2}\}\,ds
+C(1+t)(\eta_{0}+\varepsilon)\varepsilon^{2}.
\end{align}

In summary, a suitable linear combination of 
\eqref{5.13A} and \eqref{5.15A} gives the desired estimate \eqref{4.22}.
This completes the proof of Lemma \ref{lem.zefp} for the zero-order energy estimates on the macroscopic component $(\widetilde{\rho},\widetilde{u},\widetilde{\theta})$.
\end{proof}

For deducing \eqref{4.20} in the proof of Lemma \ref{lem.zefp}  above, we have used the following estimate on the basis of the Burnett functions.

\begin{lemma}\label{lem.I1dx}
Recall \eqref{4.12} for $I_1$. It holds that
\begin{align}\label{lem.I1dx.bdd}
\int_{\mathbb{R}^{3}} I_{1}\,dx
\leq&
\frac{d}{dt}E(t)+C\varepsilon\|f\|_{\sigma}^{2}+C(\eta_{0}+\varepsilon)\varepsilon^{2}.
\end{align}
Here $E(t)$ is given by
\begin{align}\label{5.17A}
E(t)=&\frac{3}{2}\sum^{3}_{i=1}\int_{\mathbb{R}^{3}}\int_{\mathbb{R}^{3}}\partial_{x_{i}}(\frac{\widetilde{\theta}}{\theta})
(R\theta)^{\frac{3}{2}}A_{i}(\frac{v-u}{\sqrt{R\theta}})\frac{\varepsilon \sqrt{\mu}}{M}f\,dv\,dx
\notag\\
&+\frac{3}{2}\sum^{3}_{i,j=1}\int_{\mathbb{R}^{3}}\int_{\mathbb{R}^{3}}\partial_{x_j}\widetilde{u}_{i}R\bar{\theta} B_{ij}(\frac{v-u}{\sqrt{R\theta}})
\frac{\varepsilon\sqrt{\mu}}{M}f\,dv\,dx
\notag\\
&-\frac{3}{2}\sum^{3}_{i,j=1}\int_{\mathbb{R}^{3}}\int_{\mathbb{R}^{3}}\partial_{x_j}\bar{u}_{i}R\widetilde{\theta} B_{ij}(\frac{v-u}{\sqrt{R\theta}})
\frac{\varepsilon\sqrt{\mu}}{M}f\,dv\,dx.
\end{align}
\end{lemma}

\begin{proof}
In order to estimate $\int_{\mathbb{R}^{3}} I_{1}\,dx$,
we only need to estimate the last three terms in \eqref{4.12} since other terms vanish after integration.
First note that the following identities hold: 
\begin{align}
\label{4.13}
\int_{\mathbb{R}^{3}}& (\frac{1}{2}v_{i}|v|^{2}-v_{i}u\cdot v)L^{-1}_{M}\Theta \,dv=
\int_{\mathbb{R}^{3}} L^{-1}_{M}\{P_{1}(\frac{1}{2}v_{i}|v|^{2}-v_{i}u\cdot v)M\}\frac{\Theta}{M}\,dv
\nonumber\\
=&\int_{\mathbb{R}^{3}} L^{-1}_{M}\{(R\theta)^{\frac{3}{2}}\hat{A}_{i}(\frac{v-u}{\sqrt{R\theta}})M\}\frac{\Theta}{M}\,dv
=(R\theta)^{\frac{3}{2}}\int_{\mathbb{R}^{3}}A_{i}(\frac{v-u}{\sqrt{R\theta}})\frac{\Theta}{M}\,dv,
\end{align}
and
\begin{align}
\label{4.14}
\int_{\mathbb{R}^{3}}v_{i}v_{j}L^{-1}_{M}\Theta \,dv=&
\int_{\mathbb{R}^{3}} L^{-1}_{M}\{P_{1}(v_{i}v_{j}M)\}\frac{\Theta}{M}\,dv
\nonumber\\
=&\int_{\mathbb{R}^{3}} L^{-1}_{M}\{R\theta\hat{B}_{ij}(\frac{v-u}{\sqrt{R\theta}})M\}\frac{\Theta}{M}\,dv
=R\theta\int_{\mathbb{R}^{3}}B_{ij}(\frac{v-u}{\sqrt{R\theta}})\frac{\Theta}{M}\,dv,
\end{align}
for $i,j=1,2,3$, in term of the self-adjoint property of $L^{-1}_{M}$, \eqref{2.5}, \eqref{5.1} and \eqref{5.2}.
Hence, the third term in \eqref{4.12} can be rewritten as
\begin{align}
\label{4.15}
&\int_{\mathbb{R}^{3}}\frac{3}{2}\nabla_{x}(\frac{\widetilde{\theta}}{\theta})
\cdot\int_{\mathbb{R}^{3}}(\frac{1}{2}|v|^{2}-v\cdot u)vL^{-1}_{M}\Theta\,dv\,dx
\nonumber\\ &
=\sum_{i=1}^{3}\int_{\mathbb{R}^{3}}\Big\{\frac{3}{2}\partial_{x_{i}}(\frac{\widetilde{\theta}}{\theta}) (R\theta)^{\frac{3}{2}}\int_{\mathbb{R}^{3}}A_{i}(\frac{v-u}{\sqrt{R\theta}})\frac{\Theta}{M}\,dv\Big\}\,dx.
\end{align}
Before computing \eqref{4.15}, we give the following desired estimate that, for any multi-index $\beta$ and $k\geq 0$,
\begin{equation}
\label{4.16}
\int_{\mathbb{R}^{3}}\frac{|\langle v\rangle^{k}\sqrt{\mu}
\partial_{\beta}A_{i}(\frac{v-u}{\sqrt{R\theta}})|^{2}}{M^{2}}\,dv
+\int_{\mathbb{R}^{3}}\frac{|\langle v\rangle^{k}
\sqrt{\mu}\partial_{\beta}B_{ij}(\frac{v-u}{\sqrt{R\theta}})|^{2}}{M^{2}}\,dv\leq C,
\end{equation}
in terms of \eqref{5.3} and \eqref{3.6}.
Recall $\Theta$ in \eqref{2.13} given by
\begin{equation}
\label{4.17}
\Theta:=\varepsilon \partial_tG+\varepsilon P_{1}(v\cdot\nabla_{x}G)-Q(G,G).
\end{equation}
We now estimate \eqref{4.15} associated  with \eqref{4.17}.
For the first term of \eqref{4.17}, noting that $G=\overline{G}+\sqrt{\mu}f$,
we use \eqref{4.16} and the similar arguments as \eqref{5.5} to obtain
\begin{align*}
&\int_{\mathbb{R}^{3}}\Big\{\frac{3}{2}\partial_{x_{i}}(\frac{\widetilde{\theta}}{\theta})
(R\theta)^{\frac{3}{2}}\int_{\mathbb{R}^{3}}A_{i}(\frac{v-u}{\sqrt{R\theta}})
\frac{\varepsilon \partial_t\overline{G}}{M}\,dv\Big\}\,dx
\nonumber\\
&\leq C\varepsilon^{2}\int_{\mathbb{R}^{3}}|\partial_{x_{i}}(\frac{\widetilde{\theta}}{\theta})|
\{|(\nabla_{x}\partial_t\bar{u},\nabla_{x}\partial_t\bar{\theta})|
+|(\nabla_{x}\bar{u},\nabla_{x}\bar{\theta})|\cdot|(\partial_tu,\partial_t\theta)|\}\,dx
\nonumber\\
&\leq C\varepsilon^{2}(\|\partial_{x_{i}}\widetilde{\theta}\|
+\|\widetilde{\theta}\partial_{x_{i}}\theta\|)
\{\|(\nabla_{x}\partial_t\bar{u},\nabla_{x}\partial_t\bar{\theta})\|
+\|(\nabla_{x}\bar{u},\nabla_{x}\bar{\theta})\|_{L^{\infty}}\|(\partial_tu,\partial_t\theta)\|\}
\nonumber\\
&\leq C\eta_{0}\varepsilon^{3},
\end{align*}
where in the last inequality we have used \eqref{3.1} and the fact
\begin{equation}
\label{4.18}
\|\partial_t(u,\theta)\|\leq C,\quad \|\partial_t(\bar{\rho},\bar{u},\bar{\theta})\|_{H^{k}}\leq C\eta_{0}, \quad \mbox{for}\quad k\geq 3,
\end{equation}
due to \eqref{5.28a},  \eqref{1.7} and \eqref{1.4}.
We use an integration by parts about $t$ to get
\begin{align*}
&\int_{\mathbb{R}^{3}}\Big\{\frac{3}{2}\partial_{x_{i}}(\frac{\widetilde{\theta}}{\theta})
(R\theta)^{\frac{3}{2}}\int_{\mathbb{R}^{3}}A_{i}(\frac{v-u}{\sqrt{R\theta}})
\frac{\varepsilon \sqrt{\mu}\partial_tf}{M}\,dv\Big\}\,dx
\nonumber\\
=&\frac{d}{dt}\int_{\mathbb{R}^{3}}\int_{\mathbb{R}^{3}}\frac{3}{2}\partial_{x_{i}}(\frac{\widetilde{\theta}}{\theta})
(R\theta)^{\frac{3}{2}}A_{i}(\frac{v-u}{\sqrt{R\theta}})
\frac{\varepsilon \sqrt{\mu}}{M}f\,dv\,dx
\nonumber\\
&-\int_{\mathbb{R}^{3}}\int_{\mathbb{R}^{3}}\partial_t\Big\{\frac{3}{2}\partial_{x_{i}}(\frac{\widetilde{\theta}}{\theta})
(R\theta)^{\frac{3}{2}}A_{i}(\frac{v-u}{\sqrt{R\theta}})
\frac{\varepsilon \sqrt{\mu}}{M}\Big\}f\,dv\,dx
\nonumber\\
\leq&\frac{d}{dt}\int_{\mathbb{R}^{3}}\int_{\mathbb{R}^{3}}
\Big\{\frac{3}{2}\partial_{x_{i}}(\frac{\widetilde{\theta}}{\theta})
(R\theta)^{\frac{3}{2}}A_{i}(\frac{v-u}{\sqrt{R\theta}})
\frac{\varepsilon \sqrt{\mu}}{M}f\Big\}\,dv\,dx
+C(\eta_{0}+\varepsilon)\varepsilon^{2},
\end{align*}
with the help of the following estimates
\begin{align*}
&-\int_{\mathbb{R}^{3}}\int_{\mathbb{R}^{3}}\partial_t\Big\{\frac{3}{2}\partial_{x_{i}}(\frac{\widetilde{\theta}}{\theta})
(R\theta)^{\frac{3}{2}}A_{i}(\frac{v-u}{\sqrt{R\theta}})
\frac{\varepsilon \sqrt{\mu}}{M}\Big\}f\,dv\,dx
\\
&\leq C\varepsilon\|\partial_{x_{i}}\partial_t(\frac{\widetilde{\theta}}{\theta})\|\|f\|
+C\varepsilon\|\partial_{x_{i}}(\frac{\widetilde{\theta}}{\theta})\|
\||f|_2\|_{L_x^{\infty}}\|\partial_t(\rho,u,\theta)\|
\\
&\leq C(\eta_{0}+\varepsilon)\varepsilon^{2}.
\end{align*}
Collecting the above estimates, we obtain
\begin{align*}
&\sum^{3}_{i=1}\int_{\mathbb{R}^{3}}\Big\{\frac{3}{2}\partial_{x_{i}}(\frac{\widetilde{\theta}}{\theta})
(R\theta)^{\frac{3}{2}}\int_{\mathbb{R}^{3}}A_{i}(\frac{v-u}{\sqrt{R\theta}})
\frac{\varepsilon \partial_tG}{M}\,dv\Big\}\,dx
\\
&\leq\frac{d}{dt}\sum^{3}_{i=1}\int_{\mathbb{R}^{3}}\int_{\mathbb{R}^{3}}
\Big\{\frac{3}{2}\partial_{x_{i}}(\frac{\widetilde{\theta}}{\theta})
(R\theta)^{\frac{3}{2}}A_{i}(\frac{v-u}{\sqrt{R\theta}})
\frac{\varepsilon \sqrt{\mu}}{M}f\Big\}\,dv\,dx
+C(\eta_{0}+\varepsilon)\varepsilon^{2}.
\end{align*}
Similarly, it is straightforward to check that
\begin{align*}
&\sum^{3}_{i=1}\int_{\mathbb{R}^{3}}\Big\{\frac{3}{2}\partial_{x_{i}}(\frac{\widetilde{\theta}}{\theta})
(R\theta)^{\frac{3}{2}}\int_{\mathbb{R}^{3}}A_{i}(\frac{v-u}{\sqrt{R\theta}})
\frac{ \varepsilon P_{1}(v\cdot\nabla_{x}G)}{M}\,dv\Big\}\,dx
\\
&=\sum^{3}_{i=1}\int_{\mathbb{R}^{3}}\Big\{\frac{3}{2}\partial_{x_{i}}(\frac{\widetilde{\theta}}{\theta})
(R\theta)^{\frac{3}{2}}\int_{\mathbb{R}^{3}}A_{i}(\frac{v-u}{\sqrt{R\theta}})
\frac{ \varepsilon [v\cdot\nabla_{x}G-P_{0}(v\cdot\nabla_{x}G)]}{M}\,dv\Big\}\,dx
\\
&\leq C(\eta_{0}+\varepsilon)\varepsilon^{2}.
\end{align*}
For the third term of \eqref{4.17}, applying
\eqref{2.19}, \eqref{5.11} and \eqref{4.16} yields that
\begin{align*}
&\sum^{3}_{i=1}|\int_{\mathbb{R}^{3}}\Big\{\frac{3}{2}\partial_{x_{i}}(\frac{\widetilde{\theta}}{\theta})
(R\theta)^{\frac{3}{2}}\int_{\mathbb{R}^{3}}A_{i}(\frac{v-u}{\sqrt{R\theta}})
\frac{Q(G,G)}{M}\,dv\Big\}\,dx|
\\
&=\sum^{3}_{i=1}|\int_{\mathbb{R}^{3}}\Big\{\frac{3}{2}\partial_{x_{i}}(\frac{\widetilde{\theta}}{\theta})
(R\theta)^{\frac{3}{2}}\int_{\mathbb{R}^{3}}A_{i}(\frac{v-u}{\sqrt{R\theta}})
\frac{\sqrt{\mu}}{M}\Gamma(\frac{G}{\sqrt{\mu}},\frac{G}{\sqrt{\mu}})\,dv\Big\}\,dx|
\\
&\leq C\int_{\mathbb{R}^{3}}|\frac{G}{\sqrt{\mu}}|_{2}
|\frac{G}{\sqrt{\mu}}|_{\sigma}|\nabla_{x}(\frac{\widetilde{\theta}}{\theta})|\,dx
\\
&\leq C\sup_{x\in\mathbb{R}^3}(|\frac{\overline{G}}{\sqrt{\mu}}|_{2}
+|f|_{2})(\|\frac{\overline{G}}{\sqrt{\mu}}\|_{\sigma}+\|f\|_{\sigma})
\|\nabla_{x}(\frac{\widetilde{\theta}}{\theta})\|
\\
&\leq C\varepsilon\|f\|_{\sigma}^{2}+C(\eta_{0}+\varepsilon)\varepsilon^{2},
\end{align*}
where in the last two inequalities we have used $G=\overline{G}+\sqrt{\mu}f$, Lemma \ref{lem5.4}, 
the Cauchy-Schwarz inequality, \eqref{1.7} and \eqref{3.1}.
Recalling \eqref{4.17} and plugging the estimates above into \eqref{4.15} leads to
\begin{align}
\label{4.19}
&\frac{3}{2}\int_{\mathbb{R}^{3}}\nabla_{x}(\frac{\widetilde{\theta}}{\theta})
\cdot\int_{\mathbb{R}^{3}}(\frac{1}{2}|v|^{2}-v\cdot u)vL^{-1}_{M}\Theta\,dv\,dx
\nonumber\\
\leq& \frac{d}{dt}\frac{3}{2}\sum^{3}_{i=1}\int_{\mathbb{R}^{3}}\int_{\mathbb{R}^{3}}\partial_{x_{i}}(\frac{\widetilde{\theta}}{\theta})
(R\theta)^{\frac{3}{2}}A_{i}(\frac{v-u}{\sqrt{R\theta}})
\frac{\varepsilon \sqrt{\mu}}{M}f\,dv\,dx
+C\varepsilon\|f\|_{\sigma}^{2}+C(\eta_{0}+\varepsilon)\varepsilon^{2}.
\end{align}
The estimation of the last two terms in \eqref{4.12} can be done similarly as \eqref{4.19} because they have the same structure.
It follows that
\begin{align*}
&\frac{3}{2}\sum^{3}_{i,j=1}\int_{\mathbb{R}^{3}}\frac{\bar{\theta}}{\theta}\partial_{x_j}\widetilde{u}_{i}\int_{\mathbb{R}^{3}} v_{i}v_jL^{-1}_{M}\Theta\,dv\,dx
=\frac{3}{2}\sum^{3}_{i,j=1}\int_{\mathbb{R}^{3}}\int_{\mathbb{R}^{3}}\frac{\bar{\theta}}{\theta}\partial_{x_j}\widetilde{u}_{i}R\theta B_{ij}(\frac{v-u}{\sqrt{R\theta}})\frac{\Theta}{M}\,dv\,dx
\\
&\leq\frac{d}{dt}\frac{3}{2}\sum^{3}_{i,j=1}\int_{\mathbb{R}^{3}}\int_{\mathbb{R}^{3}}\partial_{x_j}\widetilde{u}_{i}R\bar{\theta} B_{ij}(\frac{v-u}{\sqrt{R\theta}})
\frac{\varepsilon\sqrt{\mu}}{M}f\,dv\,dx
+C\varepsilon\|f\|_{\sigma}^{2}+C(\eta_{0}+\varepsilon)\varepsilon^{2},
\end{align*}
and
\begin{align*}
&-\frac{3}{2}\sum^{3}_{i,j=1}\int_{\mathbb{R}^{3}}\frac{\widetilde{\theta}}{\theta}\partial_{x_j}\bar{u}_{i}\int_{\mathbb{R}^{3}} v_{i}v_jL^{-1}_{M}\Theta\,dv\,dx
=-\frac{3}{2}\sum^{3}_{i,j=1}\int_{\mathbb{R}^{3}}\int_{\mathbb{R}^{3}}\frac{\widetilde{\theta}}{\theta}\partial_{x_j}\bar{u}_{i}R\theta B_{ij}(\frac{v-u}{\sqrt{R\theta}})\frac{\Theta}{M}\,dv\,dx
\\
&\leq-\frac{d}{dt}\frac{3}{2}\sum^{3}_{i,j=1}\int_{\mathbb{R}^{3}}\int_{\mathbb{R}^{3}}\partial_{x_j}\bar{u}_{i}R\widetilde{\theta} B_{ij}(\frac{v-u}{\sqrt{R\theta}})
\frac{\varepsilon\sqrt{\mu}}{M}f\,dv\,dx
+C\varepsilon\|f\|_{\sigma}^{2}+C(\eta_{0}+\varepsilon)\varepsilon^{2}.
\end{align*}
With the above two estimates and \eqref{4.19} in hand, the desired estimate \eqref{lem.I1dx.bdd} follows. This then completes the proof of Lemma \ref{lem.I1dx}.
\end{proof}

\subsection{Zero-order estimate on non-fluid part}
Next we make use of
the microscopic equation \eqref{2.21} to derive the zero-order energy estimates of the non-fluid part $f$ by using the properties of the linearized
operator. We should emphasize that the fact that $f\in (\ker\mathcal{L})^{\perp}$ is purely microscopic is crucial in the estimates.

\begin{lemma}\label{lem.zenfp}
It holds that
\begin{align}
\label{4.25}
&\|f(t)\|^2+c_1\frac{1}{\varepsilon}\int_0^t\|f(s)\|_\sigma^2\,ds \notag\\
&\leq C\varepsilon\int_0^t(\|(\nabla_{x}\widetilde{u},\nabla_{x}\widetilde{\theta})(s)\|^{2}+\|\nabla_{x}f(s)\|_{\sigma}^{2})\,ds+C(1+t)(\eta_{0}+\varepsilon)\varepsilon^{2}.
\end{align}
\end{lemma}

\begin{proof} 
Recall the microscopic equation \eqref{2.21} together with Remarks \ref{rm3.2} and \ref{rm3.3}. Taking the inner product of \eqref{2.21} with $f$ over $\mathbb{R}^3\times\mathbb{R}^{3}$ and using \eqref{2.22a}, one has
\begin{align}
\label{4.23}
\frac{1}{2}\frac{d}{dt}\|f\|^2+c_1\frac{1}{\varepsilon}\|f\|_\sigma^2
\leq&\frac{1}{\varepsilon}(\Gamma(\frac{M-\mu}{\sqrt{\mu}},f)+\Gamma(f,\frac{M-\mu}{\sqrt{\mu}}),f)
+\frac{1}{\varepsilon}(\Gamma(\frac{G}{\sqrt{\mu}},\frac{G}{\sqrt{\mu}}),f)
\nonumber\\
&-(\frac{P_{1}(v\cdot\nabla_{x}\overline{G})}{\sqrt{\mu}},f)-(\frac{\partial_{t}\overline{G}}{\sqrt{\mu}},f)
+(\frac{P_{0}[v\cdot\nabla_{x}(\sqrt{\mu}f)]}{\sqrt{\mu}},f)
\nonumber\\
&-(\frac{1}{\sqrt{\mu}}P_{1}\big\{v\cdot(\frac{|v-u|^{2}\nabla_{x}\widetilde{\theta}}{2R\theta^{2}}+\frac{(v-u)\cdot\nabla_{x}\widetilde{u}}{R\theta})M\big\},f).	
\end{align}
We will deal with each term in \eqref{4.23}. For the first term on the right hand side of \eqref{4.23},
in view of \eqref{5.11}, \eqref{5.16} and \eqref{2.26}, we get
\begin{align*}
&\frac{1}{\varepsilon}|(\Gamma(\frac{M-\mu}{\sqrt{\mu}},f),f)|+\frac{1}{\varepsilon}|(\Gamma(f,\frac{M-\mu}{\sqrt{\mu}}),f)|
\\
&\leq C\frac{1}{\varepsilon}\int_{\mathbb{R}^{3}}|\langle v\rangle^{-1}(\frac{M-\mu}{\sqrt{\mu}})|_{2}
|f|_{\sigma}^2\,dx+C\frac{1}{\varepsilon}\int_{\mathbb{R}^{3}}|\langle v\rangle^{-1}f|_{2}|(\frac{M-\mu}{\sqrt{\mu}})|_{\sigma}|f|_{\sigma}\,dx
\\
&\leq C(\varepsilon+\eta_0)\frac{1}{\varepsilon}\|f\|_\sigma^2.
\end{align*}
For the second term on the right hand side of \eqref{4.23}, we first note that
\begin{equation*}
\Gamma(\frac{G}{\sqrt{\mu}},\frac{G}{\sqrt{\mu}})
=\Gamma(\frac{\overline{G}}{\sqrt{\mu}},\frac{\overline{G}}{\sqrt{\mu}})
+\Gamma(\frac{\overline{G}}{\sqrt{\mu}},f)+\Gamma(f,\frac{\overline{G}}{\sqrt{\mu}})+\Gamma(f,f).
\end{equation*}
Then, using \eqref{5.11}, Lemma \ref{lem5.4}, the Cauchy-Schwarz and Sobolev inequalities, \eqref{1.7}, \eqref{3.1} and \eqref{3.6}, we get
\begin{align*}
\frac{1}{\varepsilon}|(\Gamma(\frac{\overline{G}}{\sqrt{\mu}},\frac{\overline{G}}{\sqrt{\mu}}),f)|	
\leq C\varepsilon\int_{\mathbb{R}^{3}}
|(\nabla_x\bar{u},\nabla_x\bar{\theta})|^2|f|_{\sigma}\,dx
\leq C\eta_{0}\frac{1}{\varepsilon}\|f\|^{2}_{\sigma}+C\eta_{0}\varepsilon^{3},
\end{align*}
\begin{align*}
\frac{1}{\varepsilon}|(\Gamma(\frac{\overline{G}}{\sqrt{\mu}},f),f)|+\frac{1}{\varepsilon}|(\Gamma(f,\frac{\overline{G}}{\sqrt{\mu}}),f)|
\leq C\int_{\mathbb{R}^{3}}|(\nabla_x\bar{u},\nabla_x\bar{\theta})||f|^{2}_{\sigma}\,dx
\leq  C\eta_{0}\|f\|^{2}_{\sigma},
\end{align*}
and
\begin{align*}
\frac{1}{\varepsilon}|(\Gamma(f,f),f)|
\leq C\frac{1}{\varepsilon}\||f|_2\|_{L_{x}^{\infty}} \|f\|^{2}_{\sigma}\leq C\|f\|^{2}_{\sigma}.
\end{align*}
With these, it follows that
\begin{equation*}
\frac{1}{\varepsilon}|(\Gamma(\frac{G}{\sqrt{\mu}},\frac{G}{\sqrt{\mu}}),f)|\leq C(\eta_{0}+\varepsilon)\frac{1}{\varepsilon}\|f\|^{2}_{\sigma}
+C\eta_{0}\varepsilon^{3}.
\end{equation*}
For the fourth term on the right hand side of \eqref{4.23}, applying the similar arguments as \eqref{5.5}, \eqref{2.26}, \eqref{4.18} and the Cauchy-Schwarz inequality, we get
\begin{align*}
|(\frac{\partial_{t}\overline{G}}{\sqrt{\mu}},f)|
\leq& \eta\frac{1}{\varepsilon}\|\langle v\rangle^{-\frac{1}{2}}f\|^{2}+C_{\eta}\varepsilon\|\langle v\rangle^{\frac{1}{2}}\frac{\partial_{t}\overline{G}}{\sqrt{\mu}}\|^{2}
\\
\leq& C\eta\frac{1}{\varepsilon}\|f\|_{\sigma}^{2}
+C_{\eta}\varepsilon^{3}\int_{\mathbb{R}^{3}}\{|\partial_{t}(\nabla_x\bar{u},\nabla_x\bar{\theta})|
+|(\nabla_x\bar{u},\nabla_x\bar{\theta})||\partial_{t}(u,\theta)|\}^{2}\,dx
\\
\leq& C\eta\frac{1}{\varepsilon}\|f\|_{\sigma}^{2}+C_{\eta}(\eta_{0}+\varepsilon)\varepsilon^{2}.
\end{align*}
The third term on the right hand side of \eqref{4.23} shares the similar bound as
\begin{align*}
|(\frac{P_{1}(v\cdot\nabla_{x}\overline{G})}{\sqrt{\mu}},f)|&=|(\frac{v\cdot\nabla_{x}\overline{G}}{\sqrt{\mu}},f)-(\frac{P_{0}(v\cdot\nabla_{x}\overline{G})}{\sqrt{\mu}},f)|\leq C\eta\frac{1}{\varepsilon}\|f\|_{\sigma}^{2}
+C_{\eta}(\eta_{0}+\varepsilon)\varepsilon^{2}.
\end{align*}
For the fifth term on the right hand side of \eqref{4.23}, we deduce from \eqref{2.5}, \eqref{2.4}, \eqref{3.6}, \eqref{2.26}
and the Cauchy-Schwarz inequality that
\begin{align*}
|(\frac{P_{0}[v\cdot\nabla_{x}(\sqrt{\mu}f)]}{\sqrt{\mu}},f)|&=|(\frac{1}{\sqrt{\mu}}\sum_{i=0}^{4}\langle v\sqrt{\mu}\cdot\nabla_{x}f,\frac{\chi_{i}}{M}\rangle\chi_{i},f)|
\\	
&\leq \eta\frac{1}{\varepsilon}\|\langle v\rangle^{-\frac{1}{2}}f\|^{2}+C_{\eta}\varepsilon\|\langle v\rangle^{\frac{1}{2}}\frac{1}{\sqrt{\mu}}\sum_{i=0}^{4}\langle v\sqrt{\mu}\cdot\nabla_{x}f,\frac{\chi_{i}}{M}\rangle\chi_{i}\|^{2}
\\
&\leq C\eta\frac{1}{\varepsilon}\|f\|_{\sigma}^{2}+ C_{\eta}\varepsilon\|\nabla_{x}f\|_{\sigma}^{2},
\end{align*}
where we have used the fact that $|\langle v\rangle^{l}\mu^{-\frac{1}{2}}M|_{2}\leq C$ for any $l\geq0$
by \eqref{3.6}.
The last term of \eqref{4.23} can be handled in the same manner and it is bounded by
\begin{align*}
&|(\frac{1}{\sqrt{\mu}}P_{1}\big\{v\cdot(\frac{|v-u|^{2}\nabla_{x}\widetilde{\theta}}{2R\theta^{2}}+\frac{(v-u)\cdot\nabla_{x}\widetilde{u}}{R\theta})M\big\},f)|
\\
&\leq C\eta\frac{1}{\varepsilon}\|f\|_{\sigma}^{2}
+C_{\eta}\varepsilon\|(\nabla_{x}\widetilde{u},\nabla_{x}\widetilde{\theta})\|^{2}.
\end{align*}
In summary, we substitute the above estimates into \eqref{4.23} and choose $\eta>0$, $\eta_{0}>0$ and $\varepsilon>0$ suitably small to get
\begin{align}
\label{4.24}
\frac{1}{2}\frac{d}{dt}\|f\|^2+\frac{c_1}{2}\frac{1}{\varepsilon}\|f\|_\sigma^2
\leq& C\varepsilon(\|(\nabla_{x}\widetilde{u},\nabla_{x}\widetilde{\theta})\|^{2}+\|\nabla_{x}f\|_{\sigma}^{2})+C(\eta_{0}+\varepsilon)\varepsilon^{2}.
\end{align}
Integrating \eqref{4.24} with respect to $t$ and using \eqref{3.5} yields the desired estimate \eqref{4.25}. This then completes the proof of Lemma \ref{lem.zenfp}.
\end{proof}

Combining Lemma \ref{lem.zenfp} together with Lemma \ref{lem.zefp} immediately implies the following result which gives the estimates on the zero-order energy norm for both the fluid and non-fluid parts.

\begin{lemma}\label{lem4.2}
It holds that
\begin{align}
\label{4.2}
&\|(\widetilde{\rho},\widetilde{u},\widetilde{\theta})(t)\|^{2}+\|f(t)\|^{2}
+c\varepsilon\int^{t}_{0}\|\nabla_x
(\widetilde{\rho},\widetilde{u},\widetilde{\theta})(s)\|^{2}\,ds
+c\frac{1}{\varepsilon}\int^{t}_{0}\|f(s)\|_{\sigma}^{2}\,ds
\nonumber\\
&\leq C\varepsilon\int^{t}_{0}\|\nabla_xf(s)\|^{2}_{\sigma}\,ds
+C(1+t)(\eta_{0}+\varepsilon^{\frac{1}{2}})\varepsilon^{2}.
\end{align}
\end{lemma}

\begin{proof}
Multiplying \eqref{4.22} by a large constant $C>0$ and then adding the resultant equation together with \eqref{4.25}, one can obtain \eqref{4.2} by using the smallness of $\varepsilon$. This completes the proof of Lemma \ref{lem4.2}.
\end{proof}

\subsection{Space derivative estimate on fluid part up to $(N-1)$-order}
This subsection is devoted to deriving the 
space derivative  estimate up to the $(N-1)$-order for the fluid part $(\widetilde{\rho},\widetilde{u},\widetilde{\theta})$. As in Section \ref{sec5.1}, the proof is based on the fluid-type systems  \eqref{2.23} and  \eqref{2.22}.

\begin{lemma}\label{lem.sdefpN1}
It holds that
 \begin{align}
\label{4.41}
&\sum_{1\leq |\alpha|\leq N-1}\|\partial^{\alpha}(\widetilde{\rho},\widetilde{u},\widetilde{\theta})(t)\|^2+c\varepsilon\sum_{2\leq |\alpha|\leq N}\int^t_0\|\partial^{\alpha}(\widetilde{\rho},\widetilde{u},\widetilde{\theta})(s)\|^2\,ds
\nonumber\\
\leq& C\varepsilon^2\sum_{|\alpha|=N}
(\|\partial^{\alpha}\widetilde{\rho}(t)\|^2+
\|\partial^{\alpha}f(t)\|^2)+ C\varepsilon\sum_{1\leq |\alpha|\leq N}\int^t_0\| \partial^{\alpha}f(s)\|^2_\sigma\,ds
\nonumber\\
&+C(\eta_{0}+\varepsilon^{\frac{1}{2}})\int^t_0\mathcal{D}_N(s)\,ds+C(1+t)(\eta_{0}+\varepsilon^{\frac{1}{2}})\varepsilon^{2}.
 \end{align}
\end{lemma}

\begin{proof}
It is divided by six steps as follows. In the first three steps we make the direct energy estimates on $\widetilde{\rho},\widetilde{u}$ and $\widetilde{\theta}$ in terms of the Navier-Stokes-type system \eqref{2.23} and then obtain the combined estimate in Step 4. In Step 5 we use the Euler-type system \eqref{2.22} to obtain the energy dissipation of $\widetilde{\rho}$ as in Section \ref{sec5.1} for the zero order estimate. In the last step we combine those results to deduce the desired estimate \eqref{4.41}.

\medskip
\noindent{\it {Step 1.}} Applying $\partial^{\alpha}$
with $1\leq|\alpha|\leq N-1$ to the first equation of  \eqref{2.23} and taking the inner product of the resulting equation with 
$\frac{2\bar{\theta}}{3\bar{\rho}^{2}}\partial^{\alpha}\widetilde{\rho}$, one has
\begin{align}
\label{4.27}
&\frac{1}{2}\frac{d}{dt}(\partial^{\alpha}\widetilde{\rho},\frac{2\bar{\theta}}{3\bar{\rho}^{2}}\partial^{\alpha}\widetilde{\rho})
-\frac{1}{2}(\partial^{\alpha}\widetilde{\rho},\partial_t(\frac{2\bar{\theta}}{3\bar{\rho}^{2}})\partial^{\alpha}\widetilde{\rho})
+(\bar{\rho}\nabla_x\cdot\partial^{\alpha}\widetilde{u},\frac{2\bar{\theta}}{3\bar{\rho}^{2}}\partial^{\alpha}\widetilde{\rho})\notag\\
&+\sum_{1\leq \alpha_{1}\leq \alpha}C^{\alpha_{1}}_\alpha(\partial^{\alpha_1}\bar{\rho}\nabla_x\cdot\partial^{\alpha-\alpha_{1}}\widetilde{u},\frac{2\bar{\theta}}{3\bar{\rho}^{2}}\partial^{\alpha}\widetilde{\rho})
\nonumber\\
&=-(\partial^{\alpha}(u\cdot\nabla_x\widetilde{\rho}),\frac{2\bar{\theta}}{3\bar{\rho}^{2}}\partial^{\alpha}\widetilde{\rho})
-(\partial^{\alpha}(\widetilde{u}\cdot\nabla_x\bar{\rho}),\frac{2\bar{\theta}}{3\bar{\rho}^{2}}\partial^{\alpha}\widetilde{\rho})
-(\partial^{\alpha}(\widetilde{\rho}\nabla_x\cdot u),\frac{2\bar{\theta}}{3\bar{\rho}^{2}}\partial^{\alpha}\widetilde{\rho}).
\end{align}	
Let us now deal with \eqref{4.27} term by term. By the Sobolev inequality, \eqref{4.18} and \eqref{3.1}, one has
\begin{align*}
|(\partial^{\alpha}\widetilde{\rho},\partial_t(\frac{2\bar{\theta}}{3\bar{\rho}^{2}})\partial^{\alpha}\widetilde{\rho})|\leq C(\|\partial_t\bar{\rho}\|_{L^\infty}+
\|\partial_t\bar{\theta}\|_{L^\infty})\|\partial^{\alpha}\widetilde{\rho}\|^2\leq C\eta_0\|\partial^{\alpha}\widetilde{\rho}\|^2
\leq C\eta_0\varepsilon^2.
\end{align*}
Performing the similar calculation as the above estimate implies
\begin{align*}
\sum_{1\leq \alpha_{1}\leq \alpha}C^{\alpha_{1}}_\alpha|(\partial^{\alpha_1}\bar{\rho}\nabla_x\cdot\partial^{\alpha-\alpha_{1}}\widetilde{u},\frac{2\bar{\theta}}{3\bar{\rho}^{2}}\partial^{\alpha}\widetilde{\rho})|
\leq C\eta_0\varepsilon^2.
\end{align*}
The first term on the right hand side of \eqref{4.27} can be written as
\begin{equation*}
-(\partial^{\alpha}(u\cdot\nabla_x\widetilde{\rho}),\frac{2\bar{\theta}}{3\bar{\rho}^{2}}\partial^{\alpha}\widetilde{\rho})=-
(u\cdot\nabla_x\partial^{\alpha}\widetilde{\rho},\frac{2\bar{\theta}}{3\bar{\rho}^{2}}\partial^{\alpha}\widetilde{\rho})-\sum_{1\leq \alpha_{1}\leq \alpha}C^{\alpha_{1}}_\alpha
(\partial^{\alpha_1}u\cdot\nabla_x\partial^{\alpha-\alpha_1}\widetilde{\rho},\frac{2\bar{\theta}}{3\bar{\rho}^{2}}\partial^{\alpha}\widetilde{\rho}).
\end{equation*}
By virtue of integration by parts, the Sobolev inequality, \eqref{1.7} and \eqref{3.1}, we obtain
\begin{align*}
|(u\cdot\nabla_x\partial^{\alpha}\widetilde{\rho},\frac{2\bar{\theta}}{3\bar{\rho}^{2}}\partial^{\alpha}\widetilde{\rho})|&=	|\frac{1}{2}(\partial^{\alpha}\widetilde{\rho},\nabla_x\cdot(u\frac{2\bar{\theta}}{3\bar{\rho}^{2}})\partial^{\alpha}\widetilde{\rho})|
\\	
&\leq C(\|\nabla_{x}(\bar{\rho},\bar{u},\bar{\theta})\|_{L^{\infty}}+\|\nabla_{x}\widetilde{u}\|_{L^{\infty}})
\|\partial^{\alpha}\widetilde{\rho}\|^2\leq C(\eta_0+\varepsilon^{\frac{1}{2}})\varepsilon^2,
\end{align*}
where we have used the fact that
\begin{align}
\label{4.30b}
\|\nabla_{x}\widetilde{u}\|_{L^{\infty}}\leq C\|\nabla^2_{x}\widetilde{u}\|^{\frac{1}{2}}\|\nabla^3_{x}\widetilde{u}\|^{\frac{1}{2}}
\leq C\varepsilon^{\frac{1}{2}}.
\end{align}
For $1\leq |\alpha_{1}|\leq |\alpha|$, it is clear to see that
\begin{align*}
(\partial^{\alpha_1}u\cdot\nabla_x\partial^{\alpha-\alpha_1}\widetilde{\rho},\frac{2\bar{\theta}}{3\bar{\rho}^{2}}\partial^{\alpha}\widetilde{\rho})
=(\partial^{\alpha_1}\widetilde{u}\cdot\nabla_x\partial^{\alpha-\alpha_1}\widetilde{\rho},\frac{2\bar{\theta}}{3\bar{\rho}^{2}}\partial^{\alpha}\widetilde{\rho})+(\partial^{\alpha_1}\bar{u}\cdot\nabla_x\partial^{\alpha-\alpha_1}\widetilde{\rho},\frac{2\bar{\theta}}{3\bar{\rho}^{2}}\partial^{\alpha}\widetilde{\rho}).
\end{align*}
If $1\leq|\alpha_{1}|\leq |\alpha|/2$, we use the similar arguments as \eqref{4.30b} and \eqref{3.1} to get
\begin{align*}
&|(\partial^{\alpha_1}\widetilde{u}\cdot\nabla_x\partial^{\alpha-\alpha_1}\widetilde{\rho},\frac{2\bar{\theta}}{3\bar{\rho}^{2}}\partial^{\alpha}\widetilde{\rho})|
\leq \|\partial^{\alpha_1}\widetilde{u}\|_{L^{\infty}}\|\nabla_x\partial^{\alpha-\alpha_1}\widetilde{\rho}\|\|\partial^{\alpha}\widetilde{\rho}\|
\leq C\varepsilon^{\frac{1}{2}}\varepsilon^2.
\end{align*}
If $|\alpha|/2<|\alpha_{1}|\leq |\alpha|$, we have the same bound as 
\begin{align*}
|(\partial^{\alpha_1}\widetilde{u}\cdot\nabla_x\partial^{\alpha-\alpha_1}\widetilde{\rho},\frac{2\bar{\theta}}{3\bar{\rho}^{2}}\partial^{\alpha}\widetilde{\rho})|
\leq C\|\partial^{\alpha_1}\widetilde{u}\|\|\nabla_x\partial^{\alpha-\alpha_1}\widetilde{\rho}\|_{L^{\infty}}\|\partial^{\alpha}\widetilde{\rho}\|
\leq C\varepsilon^{\frac{1}{2}}\varepsilon^2.
\end{align*}
Thanks to these estimates, it follows that
\begin{align*}
\sum_{1\leq \alpha_{1}\leq \alpha}C^{\alpha_{1}}_\alpha
|(\partial^{\alpha_1}\widetilde{u}\cdot\nabla_x\partial^{\alpha-\alpha_1}\widetilde{\rho},\frac{2\bar{\theta}}{3\bar{\rho}^{2}}\partial^{\alpha}\widetilde{\rho})|
\leq C\varepsilon^{\frac{1}{2}}\varepsilon^2.
\end{align*}
On the other hand, we can obtain
\begin{align*}
\sum_{1\leq \alpha_{1}\leq \alpha}C^{\alpha_{1}}_\alpha
|(\partial^{\alpha_1}\bar{u}\cdot\nabla_x\partial^{\alpha-\alpha_1}\widetilde{\rho},\frac{2\bar{\theta}}{3\bar{\rho}^{2}}\partial^{\alpha}\widetilde{\rho})|\leq C\eta_0\varepsilon^2.
\end{align*}
Collecting the above estimates, we thereby obtain
\begin{align}
\label{4.28}
|(\partial^{\alpha}(u\cdot\nabla_x\widetilde{\rho}),\frac{2\bar{\theta}}{3\bar{\rho}^{2}}\partial^{\alpha}\widetilde{\rho})|
\leq C(\eta_0+\varepsilon^{\frac{1}{2}})\varepsilon^2.
\end{align} 
The second term on the right hand side of \eqref{4.27} is relatively easy and it is bounded by
\begin{align}
\label{4.29}
|(\partial^{\alpha}(\widetilde{u}\cdot\nabla_x\bar{\rho}),\frac{2\bar{\theta}}{3\bar{\rho}^{2}}\partial^{\alpha}\widetilde{\rho})|
\leq C\sum_{\alpha_{1}\leq\alpha}
\|\nabla_x\partial^{\alpha-\alpha_{1}}\bar{\rho}\|_{L^{\infty}}\|\partial^{\alpha_1}\widetilde{u}\|\|\partial^{\alpha}\widetilde{\rho}\|
\leq C\eta_0\varepsilon^2.
\end{align}
For the last term of \eqref{4.27}, we divide it into three parts as
\begin{align*}
-(\partial^{\alpha}(\widetilde{\rho}\nabla_x\cdot u),\frac{2\bar{\theta}}{3\bar{\rho}^{2}}\partial^{\alpha}\widetilde{\rho})=&
-(\widetilde{\rho}\nabla_x\cdot \partial^{\alpha}\widetilde{u},\frac{2\bar{\theta}}{3\bar{\rho}^{2}}\partial^{\alpha}\widetilde{\rho})
-(\partial^{\alpha}(\widetilde{\rho}\nabla_x\cdot \bar{u}),\frac{2\bar{\theta}}{3\bar{\rho}^{2}}\partial^{\alpha}\widetilde{\rho})
\\
&-\sum_{1\leq \alpha_{1}\leq \alpha}C^{\alpha_{1}}_\alpha
(\partial^{\alpha_1}\widetilde{\rho}\nabla_x\cdot\partial^{\alpha-\alpha_{1}}\widetilde{u},\frac{2\bar{\theta}}{3\bar{\rho}^{2}}\partial^{\alpha}\widetilde{\rho}).
\end{align*}
The last two terms of the above equality can be treated in the same way as \eqref{4.28} and \eqref{4.29}, so that we obtain
\begin{align*}
\sum_{1\leq\alpha_{1}\leq\alpha}C^{\alpha_{1}}_\alpha
|(\partial^{\alpha_1}\widetilde{\rho}\nabla_x\cdot\partial^{\alpha-\alpha_{1}}\widetilde{u},\frac{2\bar{\theta}}{3\bar{\rho}^{2}}\partial^{\alpha}\widetilde{\rho})|+
|(\partial^{\alpha}(\widetilde{\rho}\nabla_x\cdot \bar{u}),\frac{2\bar{\theta}}{3\bar{\rho}^{2}}\partial^{\alpha}\widetilde{\rho})|
\leq C(\eta_0+\varepsilon^{\frac{1}{2}})\varepsilon^2.
\end{align*}
Applying the Cauchy-Schwarz and Sobolev inequalities together with \eqref{3.1} gives
\begin{align*}
|(\widetilde{\rho}\nabla_x\cdot \partial^{\alpha}\widetilde{u},\frac{2\bar{\theta}}{3\bar{\rho}^{2}}\partial^{\alpha}\widetilde{\rho})|
\leq \eta\varepsilon\|\nabla_x\cdot\partial^{\alpha}\widetilde{u}\|^2
+C_\eta\frac{1}{\varepsilon}\|\widetilde{\rho}\|_{L^{\infty}}^2
\|\partial^{\alpha}\widetilde{\rho}\|^2\leq \eta\varepsilon\|\nabla_x\cdot\partial^{\alpha}\widetilde{u}\|^2+C_\eta\varepsilon^3.
\end{align*}
It follows from the above estimates that
\begin{align}
\label{4.30}
|(\partial^{\alpha}(\widetilde{\rho}\nabla_x\cdot u),\frac{2\bar{\theta}}{3\bar{\rho}^{2}}\partial^{\alpha}\widetilde{\rho})|
\leq C\eta\varepsilon\|\nabla_x\cdot\partial^{\alpha}\widetilde{u}\|^2+C_\eta(\eta_0+\varepsilon^{\frac{1}{2}})\varepsilon^2.
\end{align}
For $1\leq|\alpha|\leq N-1$ and any small $\eta>0$, we
substitute the above estimates into \eqref{4.27} to get
\begin{align}
\label{4.31}
\frac{1}{2}\frac{d}{dt}\int_{\mathbb{R}^{3}} \frac{2\bar{\theta}}{3\bar{\rho}^{2}}|\partial^{\alpha}\widetilde{\rho}|^{2}\,dx
+(\nabla_x\cdot\partial^{\alpha}\widetilde{u},\frac{2\bar{\theta}}{3\bar{\rho}}\partial^{\alpha}\widetilde{\rho})\leq C\eta\varepsilon\|\nabla_x\cdot\partial^{\alpha}\widetilde{u}\|^2
+C_\eta(\eta_0+\varepsilon^{\frac{1}{2}})\varepsilon^2.
\end{align}

\medskip
\noindent{\it {Step 2.}}
Next we concentrate on the second equation of  \eqref{2.23}. Applying $\partial^{\alpha}$
with $1\leq|\alpha|\leq N-1$ to the second equation of \eqref{2.23} and taking the inner product of the resulting equation with 
$\partial^{\alpha}\widetilde{u}_i$ yields
\begin{align}
\label{4.32}
&\frac{1}{2}\frac{d}{dt}\|\partial^{\alpha}\widetilde{u}_i\|^{2}+
(\frac{2\bar{\theta}}{3\bar{\rho}}\partial^{\alpha}\partial_{x_{i}}\widetilde{\rho},\partial^{\alpha}\widetilde{u}_i)+\sum_{1\leq \alpha_{1}\leq\alpha}C_{\alpha}^{\alpha_{1}}
(\partial^{\alpha_1}(\frac{2\bar{\theta}}{3\bar{\rho}})\partial^{\alpha-\alpha_1}\partial_{x_{i}}\widetilde{\rho},\partial^{\alpha}\widetilde{u}_i)+
(\frac{2}{3}\partial^{\alpha}\partial_{x_{i}}\widetilde{\theta},\partial^{\alpha}\widetilde{u}_i)
\nonumber\\
&\hspace{0.5cm}+(\partial^{\alpha}(u\cdot\nabla_{x}\widetilde{u}_{i})+\partial^{\alpha}(\widetilde{u}\cdot\nabla_{x}\bar{u}_{i}),\partial^{\alpha}\widetilde{u}_i)
+(\partial^{\alpha}[\frac{2}{3}(\frac{\theta}{\rho}-\frac{\bar{\theta}}{\bar{\rho}})\partial_{x_{i}}\rho],\partial^{\alpha}\widetilde{u}_i)
\nonumber\\
&=\varepsilon\sum^{3}_{j=1}(\partial^{\alpha}(\frac{1}{\rho}\partial_{x_{j}}[\mu(\theta)D_{ij}]),\partial^{\alpha}\widetilde{u}_i)
-(\partial^{\alpha}(\frac{1}{\rho}\int_{\mathbb{R}^{3}} v_{i}v\cdot\nabla_{x}L^{-1}_{M}\Theta\,dv),\partial^{\alpha}\widetilde{u}_i).
\end{align}
We will estimate each term for \eqref{4.32}.
Following the same method as in \eqref{4.28} and \eqref{4.29}, we get
\begin{align*}
&\sum_{1\leq\alpha_{1}\leq\alpha}C_{\alpha}^{\alpha_{1}}|
(\partial^{\alpha_1}(\frac{2\bar{\theta}}{3\bar{\rho}})\partial^{\alpha-\alpha_1}\partial_{x_{i}}\widetilde{\rho},\partial^{\alpha}\widetilde{u}_i)|+
|(\partial^{\alpha}(u\cdot\nabla_{x}\widetilde{u}_{i})+\partial^{\alpha}(\widetilde{u}\cdot\nabla_{x}\bar{u}_{i}),\partial^{\alpha}\widetilde{u}_i)|\\
&\leq 
C(\eta_0+\varepsilon^{\frac{1}{2}})\varepsilon^2.
\end{align*}
Carrying out the similar calculations as \eqref{4.30}, one can arrive at
\begin{align*}
|(\partial^{\alpha}[\frac{2}{3}(\frac{\theta}{\rho}-\frac{\bar{\theta}}{\bar{\rho}})\partial_{x_{i}}\rho],\partial^{\alpha}\widetilde{u}_i)|
\leq C\eta\varepsilon\|\partial_{x_i}\partial^{\alpha}\widetilde{u}_i\|^2
+C_\eta(\eta_0+\varepsilon^{\frac{1}{2}})\varepsilon^2.
\end{align*}
We shall carefully deal with the first term on the right hand side of \eqref{4.32}. 
By the definition of $D_{ij}$ in \eqref{def.vst}, we first write 
\begin{align*}
\varepsilon\sum^{3}_{j=1}(\partial^{\alpha}(\frac{1}{\rho}\partial_{x_{j}}[\mu(\theta)D_{ij}]),\partial^{\alpha}\widetilde{u}_i)
&=\varepsilon\sum^{3}_{j=1}(\partial^{\alpha}(\frac{1}{\rho}\partial_{x_{j}}[\mu(\theta)(\partial_{x_{j}}\widetilde{u}_{i}+\partial_{x_{i}}\widetilde{u}_{j}
-\frac{2}{3}\delta_{ij}\nabla_{x}\cdot\widetilde{u})]),\partial^{\alpha}\widetilde{u}_i)
\\
+\varepsilon\sum^{3}_{j=1}&(\partial^{\alpha}(\frac{1}{\rho}\partial_{x_{j}}[\mu(\theta)(\partial_{x_{j}}\bar{u}_{i}+\partial_{x_{i}}\bar{u}_{j}-\frac{2}{3}\delta_{ij}\nabla_{x}\cdot \bar{u})]),\partial^{\alpha}\widetilde{u}_i)
\\
&:=I_2+I_3.
\end{align*}
By the integration by parts, the term $I_2$ can be reduced to
\begin{align*}
I_2=&-\varepsilon\sum^{3}_{j=1}(\partial^{\alpha}[\frac{1}{\rho}\mu(\theta)(\partial_{x_{j}}\widetilde{u}_{i}+\partial_{x_{i}}\widetilde{u}_{j}
-\frac{2}{3}\delta_{ij}\nabla_{x}\cdot \widetilde{u})],\partial^{\alpha}\partial_{x_{j}}\widetilde{u}_{i})
\\
&-\varepsilon\sum^{3}_{j=1}(\partial^{\alpha}[\partial_{x_{j}}(\frac{1}{\rho})\mu(\theta)(\partial_{x_{j}}\widetilde{u}_{i}+\partial_{x_{i}}\widetilde{u}_{j}
-\frac{2}{3}\delta_{ij}\nabla_{x}\cdot\widetilde{u})],\partial^{\alpha}\widetilde{u}_{i})
\\
&:=I^{1}_2+I^{2}_2.
\end{align*}
To compute the term $I_2$, it only suffices to estimate 
$I^{1}_2$ and $I^{2}_2$. Note that
\begin{align*}
I^{1}_2=&-\varepsilon\sum^{3}_{j=1}(\frac{1}{\rho}\mu(\theta)(\partial^{\alpha}\partial_{x_{j}}\widetilde{u}_{i}+\partial^{\alpha}\partial_{x_{i}}\widetilde{u}_{j}
-\frac{2}{3}\delta_{ij}\nabla_{x}\cdot \partial^{\alpha}\widetilde{u}),\partial^{\alpha}\partial_{x_{j}}\widetilde{u}_{i})
\\
&-\varepsilon\sum^{3}_{j=1}\sum_{1\leq \alpha_{1}\leq\alpha}C_{\alpha}^{\alpha_{1}}(\partial^{\alpha_1}[\frac{1}{\rho}\mu(\theta)]
\partial^{\alpha-\alpha_1}(\partial_{x_{j}}\widetilde{u}_{i}+\partial_{x_{i}}\widetilde{u}_{j}-\frac{2}{3}\delta_{ij}\nabla_{x}\cdot\widetilde{u}),\partial^{\alpha}\partial_{x_{j}}\widetilde{u}_{i}).
\end{align*}	
Applying the Sobolev inequality, \eqref{1.7} and \eqref{3.1}, we have,  with $1\leq|\alpha_{1}|\leq |\alpha|/2$, 
\begin{align*}
&\varepsilon|(\partial^{\alpha_1}[\frac{1}{\rho}\mu(\theta)]\partial^{\alpha-\alpha_1}
(\partial_{x_{j}}\widetilde{u}_{i}+\partial_{x_{i}}\widetilde{u}_{j}-\frac{2}{3}\delta_{ij}\nabla_{x}\cdot\widetilde{u}),\partial^{\alpha}\partial_{x_{j}}\widetilde{u}_{i})|
\\
&\leq C\varepsilon \|\partial^{\alpha_1}[\frac{1}{\rho}\mu(\theta)]\|_{L^{\infty}}
\|\nabla_x\partial^{\alpha-\alpha_1}\widetilde{u}\|\|\partial^{\alpha}\partial_{x_{j}}\widetilde{u}_{i}\|
\\
&\leq C\|\{|\partial^{\alpha_1}(\rho,\theta)|+\cdot\cdot\cdot+|\nabla_{x}(\rho,\theta)|^{|\alpha_1|}\}\|_{L^{\infty}}\varepsilon^2
\\
&\leq C(\eta_{0}+\varepsilon^{\frac{1}{2}})\varepsilon^2.
\end{align*}
Likewise, the case $|\alpha|/2<|\alpha_{1}|\leq |\alpha|$ has the same bound
\begin{align*}
&\varepsilon|(\partial^{\alpha_1}[\frac{1}{\rho}\mu(\theta)]\partial^{\alpha-\alpha_1}(\partial_{x_{j}}\widetilde{u}_{i}+\partial_{x_{i}}\widetilde{u}_{j}-\frac{2}{3}\delta_{ij}\nabla_{x}\cdot\widetilde{u}),\partial^{\alpha}\partial_{x_{j}}\widetilde{u}_{i})|
\\
&\leq C\varepsilon\|\partial^{\alpha_1}[\frac{1}{\rho}\mu(\theta)]\|_{L^{3}}
\|\nabla_x\partial^{\alpha-\alpha_1}\widetilde{u}\|_{L^{6}}\|\partial^{\alpha}\partial_{x_{j}}\widetilde{u}_{i}\|_{L^{2}}
\\
&\leq C(\eta_{0}+\varepsilon^{\frac{1}{2}})\varepsilon^2.
\end{align*}
By the three estimates above, we can obtain the estimation for $I^{1}_2$ as
\begin{align*}
I^{1}_2\leq&-\varepsilon\sum^{3}_{j=1}(\frac{1}{\rho}\mu(\theta)(\partial^{\alpha}\partial_{x_{j}}\widetilde{u}_{i}+\partial^{\alpha}\partial_{x_{i}}\widetilde{u}_{j}
-\frac{2}{3}\delta_{ij}\nabla_{x}\cdot \partial^{\alpha}\widetilde{u}),\partial^{\alpha}\partial_{x_{j}}\widetilde{u}_{i})+C(\eta_{0}+\varepsilon^{\frac{1}{2}})\varepsilon^2.
\end{align*}		
The term $I^{2}_2$ can be handled in the similar manner and it can be controlled by $C(\eta_{0}+\varepsilon^{\frac{1}{2}})\varepsilon^2$. We thereby obtain
\begin{align*}
I_2\leq&-\varepsilon\sum^{3}_{j=1}(\frac{1}{\rho}\mu(\theta)(\partial^{\alpha}\partial_{x_{j}}\widetilde{u}_{i}+\partial^{\alpha}\partial_{x_{i}}\widetilde{u}_{j}
-\frac{2}{3}\delta_{ij}\nabla_{x}\cdot \partial^{\alpha}\widetilde{u}),\partial^{\alpha}\partial_{x_{j}}\widetilde{u}_{i})+C(\eta_{0}+\varepsilon^{\frac{1}{2}})\varepsilon^2.
\end{align*}
The estimations for $I_3$ is easier and it is dominated by $C(\eta_{0}+\varepsilon^{\frac{1}{2}})\varepsilon^2$.
Hence, we can conclude from the above estimates
on  $I_2$ and  $I_3$ that
\begin{align}
\label{4.33}
&\varepsilon\sum^{3}_{j=1}(\partial^{\alpha}(\frac{1}{\rho}\partial_{x_{j}}[\mu(\theta)D_{ij}]),\partial^{\alpha}\widetilde{u}_i)
\nonumber\\
\leq&
-\varepsilon\sum^{3}_{j=1}(\frac{1}{\rho}\mu(\theta)(\partial^{\alpha}\partial_{x_{j}}\widetilde{u}_{i}
+\partial^{\alpha}\partial_{x_{i}}\widetilde{u}_{j}
-\frac{2}{3}\delta_{ij}\nabla_{x}\cdot \partial^{\alpha}\widetilde{u}),\partial^{\alpha}\partial_{x_{j}}\widetilde{u}_{i})+C(\eta_{0}+\varepsilon^{\frac{1}{2}})\varepsilon^2.
\end{align}

We still need to estimate the last term of \eqref{4.32}. In light of \eqref{4.14} and the integration by parts, we write
\begin{align}
\label{4.34a}
-&(\partial^{\alpha}[\frac{1}{\rho}\int_{\mathbb{R}^{3}}v_{i}v\cdot\nabla_{x}L^{-1}_{M}\Theta\,dv],\partial^{\alpha}\widetilde{u}_i)
\nonumber\\
=&-\sum^{3}_{j=1}(\partial^{\alpha}[\frac{1}{\rho}\partial_{x_j}(\int_{\mathbb{R}^{3}}R\theta B_{ij}(\frac{v-u}{\sqrt{R\theta}})\frac{\Theta}{M}\, dv)],\partial^{\alpha}\widetilde{u}_i)
\nonumber\\
=&\sum^{3}_{j=1}(\partial^{\alpha}[\partial_{x_j}(\frac{1}{\rho})\int_{\mathbb{R}^{3}}R\theta B_{ij}(\frac{v-u}{\sqrt{R\theta}})\frac{\Theta}{M}\, dv],\partial^{\alpha}\widetilde{u}_i)
\nonumber\\
&+\sum^{3}_{j=1}(\partial^{\alpha}[\frac{1}{\rho}\int_{\mathbb{R}^{3}}R\theta B_{ij}(\frac{v-u}{\sqrt{R\theta}})\frac{\Theta}{M}\, dv],\partial^{\alpha}\partial_{x_j}\widetilde{u}_i).
\end{align}
Consider the first term on the right hand side of \eqref{4.34a} associated with $\Theta:=\varepsilon \partial_tG+\varepsilon P_{1}(v\cdot\nabla_{x}G)-Q(G,G)$.
Thanks to  $1\leq|\alpha|\leq N-1$, applying \eqref{4.16}, the similar arguments as \eqref{5.5}, Lemma \ref{lem5.10}, \eqref{4.18}, \eqref{1.7}, \eqref{3.1},
as well as the Cauchy-Schwarz and Sobolev inequalities, we arrive at
\begin{align*}
\sum^{3}_{j=1}\int_{\mathbb{R}^{3}}\Big\{\partial^{\alpha}[\partial_{x_j}(\frac{1}{\rho})\int_{\mathbb{R}^{3}}R\theta B_{ij}(\frac{v-u}{\sqrt{R\theta}})\frac{\varepsilon\partial_t\overline{G}}{M}\, dv]\,\partial^{\alpha}\widetilde{u}_i\Big\}\,dx
\leq C(\eta_0+\varepsilon^{\frac{1}{2}})
\varepsilon^2.
\end{align*}
We use an integration by parts  about $t$ to get
\begin{align*}
&\sum^{3}_{j=1}\int_{\mathbb{R}^{3}}\Big\{\partial^{\alpha}[\partial_{x_j}(\frac{1}{\rho})\int_{\mathbb{R}^{3}}R\theta B_{ij}(\frac{v-u}{\sqrt{R\theta}})\frac{\varepsilon\sqrt{\mu}\partial_tf}{M}\, dv]\,\partial^{\alpha}\widetilde{u}_i\Big\}\,dx
\nonumber\\
=&\sum^{3}_{j=1}\frac{d}{dt}\int_{\mathbb{R}^{3}}\int_{\mathbb{R}^{3}}\partial^{\alpha}
\Big\{[\partial_{x_j}(\frac{1}{\rho})R\theta B_{ij}(\frac{v-u}{\sqrt{R\theta}})\frac{\varepsilon\sqrt{\mu}}{M}]f\Big\}\partial^{\alpha}\widetilde{u}_i\,dv\,dx
\nonumber\\
&-\sum^{3}_{j=1}\int_{\mathbb{R}^{3}}\int_{\mathbb{R}^{3}}\partial^{\alpha}
\Big\{\partial_t[\partial_{x_j}(\frac{1}{\rho})R\theta B_{ij}(\frac{v-u}{\sqrt{R\theta}})\frac{\varepsilon\sqrt{\mu}}{M}]f\Big\}\partial^{\alpha}\widetilde{u}_i\,dv\,dx
\nonumber\\
&-\sum^{3}_{j=1}\int_{\mathbb{R}^{3}}\int_{\mathbb{R}^{3}}\partial^{\alpha}
\Big\{[\partial_{x_j}(\frac{1}{\rho})R\theta B_{ij}(\frac{v-u}{\sqrt{R\theta}})\frac{\varepsilon\sqrt{\mu}}{M}]f\Big\}\partial^{\alpha}\partial_t\widetilde{u}_{i}\,dv\,dx.
\end{align*}
Denote $\partial^{e_i}=\partial_{x_i}$ with $|e_i|=1$,
we have from this and the integration by parts that
\begin{align*}
&-\sum^{3}_{j=1}\int_{\mathbb{R}^{3}}\int_{\mathbb{R}^{3}}\partial^{\alpha}
\Big\{\partial_t[\partial_{x_j}(\frac{1}{\rho})R\theta B_{ij}(\frac{v-u}{\sqrt{R\theta}})\frac{\varepsilon\sqrt{\mu}}{M}]f\Big\}\partial^{\alpha}\widetilde{u}_i\,dv\,dx
\\
=&\sum^{3}_{j=1}\int_{\mathbb{R}^{3}}\int_{\mathbb{R}^{3}}\partial^{\alpha-e_i}
\Big\{\partial_t[\partial_{x_j}(\frac{1}{\rho})R\theta B_{ij}(\frac{v-u}{\sqrt{R\theta}})\frac{\varepsilon\sqrt{\mu}}{M}]f\Big\}\partial^{\alpha+e_i}\widetilde{u}_i\,dv\,dx
\\
\leq&\eta\varepsilon\|\partial^{\alpha+e_i}\widetilde{u}_i\|^{2}+C_\eta\varepsilon
\sum^{3}_{j=1}\int_{\mathbb{R}^{3}}|\int_{\mathbb{R}^{3}}\partial^{\alpha-e_i}
\Big\{\partial_t[\partial_{x_j}(\frac{1}{\rho})R\theta B_{ij}(\frac{v-u}{\sqrt{R\theta}})\frac{\sqrt{\mu}}{M}]f\Big\}\,dv|^{2}\,dx
\\
\leq& C \eta\varepsilon\|\partial^{\alpha}\nabla_{x}\widetilde{u}_i\|^{2}+C_\eta(\eta_{0}+\varepsilon^{\frac{1}{2}})\varepsilon^{2},
\end{align*}
according to  \eqref{4.16}, Lemma \ref{lem5.10}, \eqref{4.18}, \eqref{1.7} and \eqref{3.1}. Similarly, it also holds that
\begin{align*}
&-\sum^{3}_{j=1}\int_{\mathbb{R}^{3}}\int_{\mathbb{R}^{3}}\partial^{\alpha}
\Big\{[\partial_{x_j}(\frac{1}{\rho})R\theta B_{ij}(\frac{v-u}{\sqrt{R\theta}})\frac{\varepsilon\sqrt{\mu}}{M}]f\Big\}\partial^{\alpha}\partial_t\widetilde{u}_{i}\,dv\,dx
\\
\leq& \eta\varepsilon\|\partial^{\alpha}\partial_t\widetilde{u}_{i}\|^{2}+C_\eta\varepsilon
\sum^{3}_{j=1}\int_{\mathbb{R}^{3}}|\int_{\mathbb{R}^{3}}\partial^{\alpha}
\Big\{[\partial_{x_j}(\frac{1}{\rho})R\theta B_{ij}(\frac{v-u}{\sqrt{R\theta}})\frac{\sqrt{\mu}}{M}]f\Big\}\,dv|^{2}\,dx
\\
\leq& C\eta\varepsilon(\|\partial^{\alpha}
\nabla_x(\widetilde{\rho},\widetilde{u},\widetilde{\theta})\|^{2}+\|\partial^{\alpha}\nabla_x f\|_\sigma^{2})+C_\eta(\eta_{0}+\varepsilon^{\frac{1}{2}})\varepsilon^{2}.
\end{align*}
With these estimates, it is clear to see that
\begin{align*}
&\sum^{3}_{j=1}\int_{\mathbb{R}^{3}}\Big\{\partial^{\alpha}[\partial_{x_j}(\frac{1}{\rho})\int_{\mathbb{R}^{3}}R\theta B_{ij}(\frac{v-u}{\sqrt{R\theta}})\frac{\varepsilon\sqrt{\mu}\partial_tf}{M}\, dv]\,\partial^{\alpha}\widetilde{u}_i\Big\}\,dx
\nonumber\\
\leq&\sum^{3}_{j=1}\frac{d}{dt}\int_{\mathbb{R}^{3}}\int_{\mathbb{R}^{3}}\partial^{\alpha}
\Big\{[\partial_{x_j}(\frac{1}{\rho})R\theta B_{ij}(\frac{v-u}{\sqrt{R\theta}})\frac{\varepsilon\sqrt{\mu}}{M}]f\Big\}\partial^{\alpha}\widetilde{u}_i\,dv\,dx
\\
&+C\eta\varepsilon(\|\partial^{\alpha}\nabla_x(\widetilde{\rho},\widetilde{u},\widetilde{\theta})\|^{2}+\|\partial^{\alpha}\nabla_x f\|_\sigma^{2})+C_\eta(\eta_{0}+\varepsilon^{\frac{1}{2}})\varepsilon^{2}.
\end{align*}
We therefore can conclude from the above estimates and  $G=\overline{G}+\sqrt{\mu}f$ that
\begin{align*}
\sum^{3}_{j=1}&\int_{\mathbb{R}^{3}}\Big\{\partial^{\alpha}[\partial_{x_j}(\frac{1}{\rho})\int_{\mathbb{R}^{3}}R\theta B_{ij}(\frac{v-u}{\sqrt{R\theta}})\frac{\varepsilon \partial_tG}{M}\, dv]\,\partial^{\alpha}\widetilde{u}_i\Big\}\,dx
\\
\leq&\sum^{3}_{j=1}\frac{d}{dt}\int_{\mathbb{R}^{3}}\int_{\mathbb{R}^{3}}\partial^{\alpha}
\Big\{[\partial_{x_j}(\frac{1}{\rho})R\theta B_{ij}(\frac{v-u}{\sqrt{R\theta}})\frac{\varepsilon\sqrt{\mu}}{M}]f\Big\}\partial^{\alpha}\widetilde{u}_i\,dv\,dx
\\
&+C\eta\varepsilon(\|\partial^{\alpha}
\nabla_x(\widetilde{\rho},\widetilde{u},\widetilde{\theta})\|^{2}+\|\partial^{\alpha}\nabla_x f\|_\sigma^{2})+C_\eta(\eta_{0}+\varepsilon^{\frac{1}{2}})\varepsilon^{2}.
\end{align*}
The second term in $\Theta$ can be handled in the similar manner and it can be controlled by
\begin{align*}
&\sum^{3}_{j=1}\int_{\mathbb{R}^{3}}\Big\{\partial^{\alpha}[\partial_{x_j}(\frac{1}{\rho})\int_{\mathbb{R}^{3}}R\theta B_{ij}(\frac{v-u}{\sqrt{R\theta}})\frac{\varepsilon P_{1}(v\cdot\nabla_{x}G)}{M}\, dv]\,\partial^{\alpha}\widetilde{u}_i\Big\}\,dx
\leq
C(\eta_{0}+\varepsilon^{\frac{1}{2}})\varepsilon^{2}.
\end{align*}
For the last term in $\Theta$, it holds by using \eqref{2.19}, \eqref{5.11}, \eqref{4.16}, Lemma \ref{lem5.4}, \eqref{1.7} and \eqref{3.1} that
\begin{align*}
&-\sum^{3}_{j=1}\int_{\mathbb{R}^{3}}\Big\{\partial^{\alpha}[\partial_{x_j}(\frac{1}{\rho})\int_{\mathbb{R}^{3}}R\theta B_{ij}(\frac{v-u}{\sqrt{R\theta}})\frac{Q(G,G)}{M}\, dv]\,\partial^{\alpha}\widetilde{u}_i\Big\}\,dx	
\\
&=-\sum^{3}_{j=1}\int_{\mathbb{R}^{3}}\Big\{\partial^{\alpha}[\partial_{x_j}(\frac{1}{\rho})\int_{\mathbb{R}^{3}}R\theta B_{ij}(\frac{v-u}{\sqrt{R\theta}})\frac{\sqrt{\mu}}{M}\Gamma(\frac{G}{\sqrt{\mu}},\frac{G}{\sqrt{\mu}})\, dv]\,\partial^{\alpha}\widetilde{u}_i\Big\}\,dx
\\	
&\leq C(\eta_{0}+\varepsilon^{\frac{1}{2}})\varepsilon^{2}+C(\eta_{0}+\varepsilon^{\frac{1}{2}})\mathcal{D}_N(t).
\end{align*}
Collecting the above estimates, we can obtain with any small $\eta>0$ that
\begin{align}
\label{4.34}
&\sum^{3}_{j=1}(\partial^{\alpha}[\partial_{x_j}(\frac{1}{\rho})\int_{\mathbb{R}^{3}}R\theta B_{ij}(\frac{v-u}{\sqrt{R\theta}})\frac{\Theta}{M}\, dv],\partial^{\alpha}\widetilde{u}_i)
\nonumber\\
\leq& 
\sum^{3}_{j=1}\frac{d}{dt}\int_{\mathbb{R}^{3}}\int_{\mathbb{R}^{3}}\partial^{\alpha}
\Big\{[\partial_{x_j}(\frac{1}{\rho})R\theta B_{ij}(\frac{v-u}{\sqrt{R\theta}})\frac{\varepsilon\sqrt{\mu}}{M}]f\Big\}\partial^{\alpha}\widetilde{u}_i\,dv\,dx
\nonumber\\
&+C\eta\varepsilon(\|\partial^{\alpha}
\nabla_x(\widetilde{\rho},\widetilde{u},\widetilde{\theta})\|^{2}+\|\partial^{\alpha}\nabla_x f\|_\sigma^{2})+C_\eta(\eta_{0}+\varepsilon^{\frac{1}{2}})\varepsilon^{2}
+C(\eta_{0}+\varepsilon^{\frac{1}{2}})\mathcal{D}_N(t).
\end{align}
The second term of \eqref{4.34a} has the same structure as the first term and it shares the same bound. For brevity, we directly give the following computations:
\begin{align*}
&\sum^{3}_{j=1}(\partial^{\alpha}[\frac{1}{\rho}\int_{\mathbb{R}^{3}}R\theta B_{ij}(\frac{v-u}{\sqrt{R\theta}})\frac{\varepsilon \partial_tG}{M}\, dv],\partial^{\alpha}\partial_{x_j}\widetilde{u}_i)
\\
\leq&
-\sum^{3}_{j=1}\frac{d}{dt}\int_{\mathbb{R}^{3}}\int_{\mathbb{R}^{3}}
\Big\{\partial^{\alpha}\partial_{x_j}[\frac{1}{\rho}R\theta B_{ij}(\frac{v-u}{\sqrt{R\theta}})\frac{\varepsilon \sqrt{\mu}f}{M}]
\partial^{\alpha}\widetilde{u}_i\Big\}\,dv\,dx
\\
&+C\eta\varepsilon\|\partial^{\alpha}
\nabla_x(\widetilde{\rho},\widetilde{u},\widetilde{\theta})\|^{2}+C_\eta\varepsilon\|\partial^{\alpha}\nabla_x f\|_\sigma^{2}+C_\eta(\eta_{0}+\varepsilon^{\frac{1}{2}})\varepsilon^{2},
\end{align*}
and
\begin{align*}
\sum^{3}_{j=1}&|(\partial^{\alpha}[\frac{1}{\rho}\int_{\mathbb{R}^{3}}R\theta B_{ij}(\frac{v-u}{\sqrt{R\theta}})\frac{\varepsilon P_{1}(v\cdot\nabla_{x}G)}{M}\,dv],\partial^{\alpha}\partial_{x_j}\widetilde{u}_i)|
\\
\leq& C\eta\varepsilon\|\partial^{\alpha}\nabla_{x}\widetilde{u}\|^2+ C_\eta\varepsilon\|\partial^{\alpha}\nabla_xf\|^2_\sigma+C_\eta(\eta_{0}+\varepsilon^{\frac{1}{2}})\varepsilon^{2},
\end{align*}
as well as
\begin{align*}
&\sum^{3}_{j=1}|(\partial^{\alpha}[\frac{1}{\rho}\int_{\mathbb{R}^{3}}R\theta B_{ij}(\frac{v-u}{\sqrt{R\theta}})\frac{Q(G,G)}{M}\,dv],\partial^{\alpha}\partial_{x_j}\widetilde{u}_i)|			
\\
\leq& C\eta\varepsilon\|\partial^{\alpha}\nabla_{x}\widetilde{u}\|^2+C_\eta(\eta_{0}+\varepsilon^{\frac{1}{2}})\varepsilon^{2}
+C_\eta(\eta_{0}+\varepsilon^{\frac{1}{2}})\mathcal{D}_N(t).
\end{align*}
Then for any small $\eta>0$, the second term of \eqref{4.34a} is bounded by
\begin{align*}
&\sum^{3}_{j=1}(\partial^{\alpha}[\frac{1}{\rho}\int_{\mathbb{R}^{3}}R\theta B_{ij}(\frac{v-u}{\sqrt{R\theta}})\frac{\Theta}{M}\, dv],\partial^{\alpha}\partial_{x_j}\widetilde{u}_i)
\nonumber\\
\leq&-
\sum^{3}_{j=1}\frac{d}{dt}\int_{\mathbb{R}^{3}}\int_{\mathbb{R}^{3}}
\Big\{\partial^{\alpha}\partial_{x_j}[\frac{1}{\rho}R\theta B_{ij}(\frac{v-u}{\sqrt{R\theta}})\frac{\varepsilon \sqrt{\mu}f}{M}]
\partial^{\alpha}\widetilde{u}_i\Big\}\,dv\,dx
\nonumber\\
&+C\eta\varepsilon\|\partial^{\alpha}
\nabla_x(\widetilde{\rho},\widetilde{u},\widetilde{\theta})\|^{2}+C_\eta\varepsilon\|\partial^{\alpha}\nabla_x f\|_\sigma^{2}
+C_\eta(\eta_{0}+\varepsilon^{\frac{1}{2}})\varepsilon^{2}
+C_\eta(\eta_{0}+\varepsilon^{\frac{1}{2}})\mathcal{D}_N(t).
\end{align*}
This estimate, combined with \eqref{4.34}, as well as \eqref{4.34a}, gives
\begin{align*}
-&(\partial^{\alpha}[\frac{1}{\rho}\int_{\mathbb{R}^{3}}v_{i}v\cdot\nabla_{x}L^{-1}_{M}\Theta\,dv],\partial^{\alpha}\widetilde{u}_i)
\nonumber\\
\leq&-
\sum^{3}_{j=1}\frac{d}{dt}\int_{\mathbb{R}^{3}}\int_{\mathbb{R}^{3}}
\Big\{\partial^{\alpha}[\frac{1}{\rho}\partial_{x_j}(R\theta B_{ij}(\frac{v-u}{\sqrt{R\theta}})\frac{\varepsilon \sqrt{\mu}f}{M})]
\partial^{\alpha}\widetilde{u}_i\Big\}\,dv\,dx
\nonumber\\
&+C\eta\varepsilon\|\partial^{\alpha}
\nabla_x(\widetilde{\rho},\widetilde{u},\widetilde{\theta})\|^{2}+C_\eta\varepsilon\|\partial^{\alpha}\nabla_x f\|_\sigma^{2}
+C_\eta(\eta_{0}+\varepsilon^{\frac{1}{2}})\varepsilon^{2}
+C_\eta(\eta_{0}+\varepsilon^{\frac{1}{2}})\mathcal{D}_N(t).
\end{align*}
For $1\leq|\alpha|\leq N-1$
and any small $\eta>0$,
plugging all the estimates above into \eqref{4.32}  and summing $i$ from 1 to 3, we obtain 
\begin{align}
\label{4.35}
&\frac{1}{2}\frac{d}{dt}\|\partial^{\alpha}\widetilde{u}\|^{2}+
(\frac{2\bar{\theta}}{3\bar{\rho}}\partial^{\alpha}\nabla_{x}\widetilde{\rho},\partial^{\alpha}\widetilde{u})+(\frac{2}{3}\partial^{\alpha}\nabla_{x}\widetilde{\theta},\partial^{\alpha}\widetilde{u})
+c\varepsilon\|\partial^{\alpha}\nabla_{x}\widetilde{u}\|^{2}
\nonumber\\
&+\sum^{3}_{i,j=1}\frac{d}{dt}\int_{\mathbb{R}^{3}}\int_{\mathbb{R}^{3}}
\Big\{\partial^{\alpha}[\frac{1}{\rho}\partial_{x_j}(R\theta B_{ij}(\frac{v-u}{\sqrt{R\theta}})\frac{\varepsilon \sqrt{\mu}f}{M})]
\partial^{\alpha}\widetilde{u}_i\Big\}\,dv\,dx
\nonumber\\
\leq&C\eta\varepsilon\|\partial^{\alpha}
\nabla_x(\widetilde{\rho},\widetilde{u},\widetilde{\theta})\|^{2}+C_\eta\varepsilon\|\partial^{\alpha}\nabla_x f\|_\sigma^{2}
+C_\eta(\eta_{0}+\varepsilon^{\frac{1}{2}})\varepsilon^{2}
+C_\eta(\eta_{0}+\varepsilon^{\frac{1}{2}})\mathcal{D}_N(t).
\end{align}
Here the following crucial inequality has been  used
\begin{align*}
&\varepsilon\sum^{3}_{i,j=1}(\frac{1}{\rho}\mu(\theta)(\partial^{\alpha}\partial_{x_{j}}\widetilde{u}_{i}+\partial^{\alpha}\partial_{x_{i}}\widetilde{u}_{j}
-\frac{2}{3}\delta_{ij}\nabla_{x}\cdot \partial^{\alpha}\widetilde{u}),\partial^{\alpha}\partial_{x_{j}}\widetilde{u}_{i})
\\
=&\varepsilon\sum^{3}_{i,j=1}(\frac{\mu(\theta)}{\rho}\partial^{\alpha}\partial_{x_{j}}\widetilde{u}_{i},\partial^{\alpha}\partial_{x_{j}}\widetilde{u}_{i})
+\frac{1}{3}\varepsilon\sum^{3}_{i,j=1}(\frac{\mu(\theta)}{\rho}\partial^{\alpha}\partial_{x_{j}}\widetilde{u}_{j},\partial^{\alpha}\partial_{x_{i}}\widetilde{u}_{i})
\\
&-\varepsilon\sum^{3}_{i,j=1}(\partial_{x_{j}}[\frac{\mu(\theta)}{\rho}]\partial^{\alpha}\partial_{x_{i}}\widetilde{u}_{j},\partial^{\alpha}\widetilde{u}_{i})
+\varepsilon\sum^{3}_{i,j=1}(\partial_{x_{i}}[\frac{\mu(\theta)}{\rho}]\partial^{\alpha}\partial_{x_{j}}\widetilde{u}_{j},\partial^{\alpha}\widetilde{u}_{i})
\\
&\geq c\varepsilon\|\partial^{\alpha}\nabla_{x}\widetilde{u}\|^{2}-C(\eta_0+\varepsilon^{\frac{1}{2}})\varepsilon^{2}.
\end{align*}

\medskip
\noindent{\it {Step 3.}}
Let us now turn to consider the  third equation of  \eqref{2.23}. Applying $\partial^{\alpha}$
with $1\leq|\alpha|\leq N-1$ to  the third equation of \eqref{2.23} and taking the inner product of the resulting equation with 
$\frac{1}{\bar{\theta}}\partial^{\alpha}\widetilde{\theta}$ yields
\begin{align}
\label{4.36}
&(\partial_{t}\partial^{\alpha}\widetilde{\theta},\frac{1}{\bar{\theta}}\partial^{\alpha}\widetilde{\theta})+(\frac{2}{3}\nabla_{x}\cdot \partial^{\alpha}\widetilde{u},\partial^{\alpha}\widetilde{\theta})+\sum_{1\leq \alpha_{1}\leq \alpha}C_{\alpha}^{\alpha_{1}}(\frac{2}{3}\partial^{\alpha_{1}}\bar{\theta}\nabla_{x}\cdot \partial^{\alpha-\alpha_{1}}\widetilde{u},\frac{1}{\bar{\theta}}\partial^{\alpha}\widetilde{\theta})
\nonumber\\
=&-(\partial^{\alpha}(u\cdot\nabla_{x}\widetilde{\theta}),\frac{1}{\bar{\theta}}\partial^{\alpha}\widetilde{\theta})-(\partial^{\alpha}(\widetilde{u}\cdot\nabla_{x}\bar{\theta}),\frac{1}{\bar{\theta}}\partial^{\alpha}\widetilde{\theta})
-(\frac{2}{3}\partial^{\alpha}(\widetilde{\theta}\nabla_{x}\cdot u)
,\frac{1}{\bar{\theta}}\partial^{\alpha}\widetilde{\theta})\nonumber\\
&
+\varepsilon\sum^{3}_{j=1}(\partial^{\alpha}[\frac{1}{\rho}\partial_{x_{j}}(\kappa(\theta)\partial_{x_{j}}\theta)],\frac{1}{\bar{\theta}}\partial^{\alpha}\widetilde{\theta})
+\varepsilon\sum^{3}_{i,j=1}(\partial^{\alpha}[\frac{1}{\rho}\mu(\theta)\partial_{x_{j}}u_{i}D_{ij}],\frac{1}{\bar{\theta}}\partial^{\alpha}\widetilde{\theta})
\notag\\
&-(\partial^{\alpha}[\frac{1}{\rho}\int_{\mathbb{R}^{3}} \frac{1}{2}|v|^{2} v\cdot\nabla_{x}L^{-1}_{M}\Theta\,dv],\frac{1}{\bar{\theta}}\partial^{\alpha}\widetilde{\theta})
+(\partial^{\alpha}[\frac{1}{\rho}u\cdot\int_{\mathbb{R}^{3}} v\otimes v\cdot\nabla_{x}L^{-1}_{M}\Theta\,dv]
,\frac{1}{\bar{\theta}}\partial^{\alpha}\widetilde{\theta}).
\end{align}
We will estimate each term for \eqref{4.36}. We use the Sobolev inequality, \eqref{4.18} and  \eqref{3.1} to get
\begin{align*}
(\partial_{t}\partial^{\alpha}\widetilde{\theta},\frac{1}{\bar{\theta}}\partial^{\alpha}\widetilde{\theta})&=\frac{1}{2}\frac{d}{dt}(\partial^{\alpha}\widetilde{\theta},\frac{1}{\bar{\theta}}\partial^{\alpha}\widetilde{\theta})-\frac{1}{2}(\partial^{\alpha}\widetilde{\theta},\partial_{t}(\frac{1}{\bar{\theta}})\partial^{\alpha}\widetilde{\theta})
\\
&\geq\frac{1}{2}\frac{d}{dt}(\partial^{\alpha}\widetilde{\theta},\frac{1}{\bar{\theta}}\partial^{\alpha}\widetilde{\theta})
-C\|\partial_t\bar{\theta}\|_{\infty}\|\partial^{\alpha}\widetilde{\theta}\|^2\geq \frac{1}{2}\frac{d}{dt}(\partial^{\alpha}\widetilde{\theta},\frac{1}{\bar{\theta}}\partial^{\alpha}\widetilde{\theta})-C\eta_0\varepsilon^2.
\end{align*}
Performing calculations similar to those for \eqref{4.28}, \eqref{4.29} and \eqref{4.30}, we get
\begin{multline*}
\sum_{1\leq\alpha_{1}\leq \alpha} C_{\alpha}^{\alpha_{1}}|(\frac{2}{3}\partial^{\alpha_{1}}\bar{\theta}\nabla_{x}\cdot \partial^{\alpha-\alpha_{1}}\widetilde{u},\frac{1}{\bar{\theta}}\partial^{\alpha}\widetilde{\theta})|
+|(\partial^{\alpha}(u\cdot\nabla_{x}\widetilde{\theta}),\frac{1}{\bar{\theta}}\partial^{\alpha}\widetilde{\theta})|
\\
+|(\partial^{\alpha}(\widetilde{u}\cdot\nabla_{x}\bar{\theta}),\frac{1}{\bar{\theta}}\partial^{\alpha}\widetilde{\theta})|
+|(\frac{2}{3}\partial^{\alpha}(\widetilde{\theta}\nabla_{x}\cdot u)
,\frac{1}{\bar{\theta}}\partial^{\alpha}\widetilde{\theta})|
\leq C\eta\varepsilon\|\nabla_x\cdot\partial^{\alpha}\widetilde{u}\|^2+C_\eta(\eta_0+\varepsilon^{\frac{1}{2}})\varepsilon^2.
\end{multline*}
The fourth term on the right hand side of \eqref{4.36} can be handled in the same way as in
\eqref{4.33}, it follows that
\begin{align*}
\varepsilon\sum^{3}_{j=1}(\partial^{\alpha}[\frac{1}{\rho}\partial_{x_{j}}(\kappa(\theta)\partial_{x_{j}}\theta)],\frac{1}{\bar{\theta}}\partial^{\alpha}\widetilde{\theta})
&\leq-\varepsilon\sum^{3}_{j=1}(\frac{1}{\rho}\kappa(\theta)\partial^{\alpha}\partial_{x_{j}}\widetilde{\theta},
\frac{1}{\bar{\theta}}\partial^{\alpha}\partial_{x_{j}}\widetilde{\theta})
+C(\eta_0+\varepsilon^{\frac{1}{2}})\varepsilon^2.
\end{align*}
The fifth term on the right hand side of \eqref{4.36} is controlled by
$C(\eta_0+\varepsilon^{\frac{1}{2}})\varepsilon^2$.
Since the structure of the last two terms in \eqref{4.36} is almost the same as \eqref{4.34a}, we thus  arrive at
\begin{align*}
-&(\partial^{\alpha}[\frac{1}{\rho}\int_{\mathbb{R}^{3}} \frac{1}{2}|v|^{2} v\cdot\nabla_{x}L^{-1}_{M}\Theta\,dv],\frac{1}{\bar{\theta}}\partial^{\alpha}\widetilde{\theta})
+(\partial^{\alpha}[\frac{1}{\rho}u\cdot\int_{\mathbb{R}^{3}} v\otimes v\cdot\nabla_{x}L^{-1}_{M}\Theta\,dv]
,\frac{1}{\bar{\theta}}\partial^{\alpha}\widetilde{\theta})
\\
=&-\sum^{3}_{i=1}
(\partial^{\alpha}[\frac{1}{\rho}\partial_{x_i}(\int_{\mathbb{R}^{3}}
(\frac{1}{2}|v|^{2}v_i-u\cdot vv_{i})L^{-1}_{M}\Theta\,dv)],\frac{1}{\bar{\theta}}\partial^{\alpha}\widetilde{\theta})\\
&\quad-\sum^{3}_{i=1}
(\partial^{\alpha}[\frac{1}{\rho}\int_{\mathbb{R}^{3}}
\partial_{x_i}u\cdot vv_{i}L^{-1}_{M}\Theta\,dv],\frac{1}{\bar{\theta}}\partial^{\alpha}\widetilde{\theta})
\\
=&-\sum^{3}_{i=1}
(\partial^{\alpha}[\frac{1}{\rho}\partial_{x_i}
((R\theta)^{\frac{3}{2}}\int_{\mathbb{R}^{3}}A_{i}(\frac{v-u}{\sqrt{R\theta}})\frac{\Theta}{M}\,dv)],\frac{1}{\bar{\theta}}\partial^{\alpha}\widetilde{\theta})\\
&\quad-\sum^{3}_{i,j=1}
(\partial^{\alpha}[\frac{1}{\rho}
\partial_{x_i}u_j R\theta\int_{\mathbb{R}^{3}}B_{ij}(\frac{v-u}{\sqrt{R\theta}})\frac{\Theta}{M}\,dv],\frac{1}{\bar{\theta}}\partial^{\alpha}\widetilde{\theta}),
\nonumber
\end{align*}
which can be further bounded by
\begin{align*}
&
-\sum^{3}_{i=1}\frac{d}{dt}\int_{\mathbb{R}^{3}}\int_{\mathbb{R}^{3}}
\Big\{\partial^{\alpha}[\frac{1}{\rho}\partial_{x_i}
((R\theta)^{\frac{3}{2}}A_{i}(\frac{v-u}{\sqrt{R\theta}})\frac{\varepsilon\sqrt{\mu}f}{M})]\frac{1}{\bar{\theta}}\partial^{\alpha}\widetilde{\theta}\Big\}\,dv\,dx
\nonumber\\	
&-\sum^{3}_{i,j=1}\frac{d}{dt}\int_{\mathbb{R}^{3}}\int_{\mathbb{R}^{3}}
\Big\{\partial^{\alpha}[\frac{1}{\rho}
\partial_{x_i}u_j R\theta B_{ij}(\frac{v-u}{\sqrt{R\theta}})\frac{\varepsilon\sqrt{\mu}f}{M}]\frac{1}{\bar{\theta}}\partial^{\alpha}\widetilde{\theta}\Big\}\,dv\,dx
\nonumber\\	
&+C\eta\varepsilon\|\partial^{\alpha}
\nabla_x(\widetilde{\rho},\widetilde{u},\widetilde{\theta})\|^{2}+C_\eta\varepsilon\|\partial^{\alpha}\nabla_x f\|_\sigma^{2}
+C_\eta(\eta_{0}+\varepsilon^{\frac{1}{2}})\varepsilon^{2}
+C_\eta(\eta_{0}+\varepsilon^{\frac{1}{2}})\mathcal{D}_N(t).
\end{align*}
Hence, substituting the above estimates into \eqref{4.36}, we have established, for $1\leq|\alpha|\leq N-1$ and any small $\eta>0$,
\begin{align}
\label{4.37}
&\frac{1}{2}\frac{d}{dt}(\partial^{\alpha}\widetilde{\theta},\frac{1}{\bar{\theta}}\partial^{\alpha}\widetilde{\theta})+(\frac{2}{3}\nabla_{x}\cdot \partial^{\alpha}\widetilde{u},\partial^{\alpha}\widetilde{\theta})
+\varepsilon(\frac{1}{\rho}\kappa(\theta)\partial^{\alpha}\nabla_{x}\widetilde{\theta},\frac{1}{\bar{\theta}}\partial^{\alpha}\nabla_{x}\widetilde{\theta})
\nonumber\\
&+\sum^{3}_{i=1}\frac{d}{dt}\int_{\mathbb{R}^{3}}\int_{\mathbb{R}^{3}}
\Big\{\partial^{\alpha}[\frac{1}{\rho}\partial_{x_i}
((R\theta)^{\frac{3}{2}}A_{i}(\frac{v-u}{\sqrt{R\theta}})\frac{\varepsilon\sqrt{\mu}f}{M})]\frac{1}{\bar{\theta}}\partial^{\alpha}\widetilde{\theta}\Big\}\,dv\,dx
\nonumber\\	
&+\sum^{3}_{i,j=1}\frac{d}{dt}\int_{\mathbb{R}^{3}}\int_{\mathbb{R}^{3}}
\Big\{\partial^{\alpha}[\frac{1}{\rho}
\partial_{x_i}u_j R\theta B_{ij}(\frac{v-u}{\sqrt{R\theta}})\frac{\varepsilon\sqrt{\mu}f}{M}]\frac{1}{\bar{\theta}}\partial^{\alpha}\widetilde{\theta}\Big\}\,dv\,dx
\nonumber\\	
&\leq C\eta\varepsilon\|\partial^{\alpha}
\nabla_x(\widetilde{\rho},\widetilde{u},\widetilde{\theta})\|^{2}+C_\eta\varepsilon\|\partial^{\alpha}\nabla_x f\|_\sigma^{2}
+C_\eta(\eta_{0}+\varepsilon^{\frac{1}{2}})\varepsilon^{2}
+C_\eta(\eta_{0}+\varepsilon^{\frac{1}{2}})\mathcal{D}_N(t).
\end{align}

\medskip
\noindent{\it {Step 4.}}
In summary, for any $1\leq|\alpha|\leq N-1$ and any small $\eta>0$, adding \eqref{4.31}, \eqref{4.35} and \eqref{4.37} together, 
combined with the following estimates
\begin{align*}
&|(\nabla_x\cdot\partial^{\alpha}\widetilde{u},\frac{2\bar{\theta}}{3\bar{\rho}}\partial^{\alpha}\widetilde{\rho})+
(\frac{2\bar{\theta}}{3\bar{\rho}}\partial^{\alpha}\nabla_{x}\widetilde{\rho},\partial^{\alpha}\widetilde{u})
+(\frac{2}{3}\partial^{\alpha}\nabla_{x}\widetilde{\theta},\partial^{\alpha}\widetilde{u})
+(\frac{2}{3}\nabla_{x}\cdot \partial^{\alpha}\widetilde{u},\partial^{\alpha}\widetilde{\theta})|
\\
&=|(\partial^{\alpha}\widetilde{u},\nabla_x(\frac{2\bar{\theta}}{3\bar{\rho}})\partial^{\alpha}\widetilde{\rho})|
\leq C\|\nabla_x(\bar{\rho},\bar{\theta})\|_{L^{\infty}}
\|\partial^{\alpha}\widetilde{u}\|\|\partial^{\alpha}\widetilde{\rho}\|\leq C\eta_{0}\varepsilon^{2},
\end{align*}
then the summation of the resulting equation over $|\alpha|$ through a suitable linear combination gives
\begin{align}
\label{4.38}
\frac{1}{2}&\sum_{1\leq |\alpha|\leq N-1}\frac{d}{dt}\int_{\mathbb{R}^{3}}( \frac{2\bar{\theta}}{3\bar{\rho}^{2}}|\partial^{\alpha}\widetilde{\rho}|^{2}+|\partial^{\alpha}\widetilde{u}|^{2}
+\frac{1}{\bar{\theta}}|\partial^{\alpha}\widetilde{\theta}|^{2})
\,dx+\frac{d}{dt}E_1(t)+c\varepsilon\sum_{2\leq |\alpha|\leq N}\|\partial^{\alpha}(\widetilde{u},\widetilde{\theta})\|^{2}
\nonumber\\	
\leq&C\eta\varepsilon\sum_{2\leq |\alpha|\leq N}\|\partial^{\alpha}
(\widetilde{\rho},\widetilde{u},\widetilde{\theta})\|^{2}+C_\eta\varepsilon\sum_{2\leq |\alpha|\leq N}
\|\partial^{\alpha}f\|_\sigma^{2}+C_\eta(\eta_{0}+\varepsilon^{\frac{1}{2}})\varepsilon^{2}
+C_\eta(\eta_{0}+\varepsilon^{\frac{1}{2}})\mathcal{D}_N(t).
\end{align}
Here we have denoted that
\begin{align}
\label{4.40a}
E_1(t)
=&\sum_{1\leq |\alpha|\leq N-1}\sum^{3}_{i,j=1}\int_{\mathbb{R}^{3}}\int_{\mathbb{R}^{3}}
\Big\{\partial^{\alpha}[\frac{1}{\rho}\partial_{x_j}(R\theta B_{ij}(\frac{v-u}{\sqrt{R\theta}})\frac{\varepsilon \sqrt{\mu}f}{M})]
\partial^{\alpha}\widetilde{u}_i\Big\}\,dv\,dx
\nonumber\\
&+\sum_{1\leq |\alpha|\leq N-1}\sum^{3}_{i=1}\int_{\mathbb{R}^{3}}\int_{\mathbb{R}^{3}}
\Big\{\partial^{\alpha}[\frac{1}{\rho}\partial_{x_i}
((R\theta)^{\frac{3}{2}}A_{i}(\frac{v-u}{\sqrt{R\theta}})\frac{\varepsilon\sqrt{\mu}f}{M})]\frac{1}{\bar{\theta}}\partial^{\alpha}\widetilde{\theta}\Big\}\,dv\,dx
\nonumber\\	
&+\sum_{1\leq |\alpha|\leq N-1}\sum^{3}_{i,j=1}\int_{\mathbb{R}^{3}}\int_{\mathbb{R}^{3}}
\Big\{\partial^{\alpha}[\frac{1}{\rho}
\partial_{x_i}u_j R\theta B_{ij}(\frac{v-u}{\sqrt{R\theta}})\frac{\varepsilon\sqrt{\mu}f}{M}]\frac{1}{\bar{\theta}}\partial^{\alpha}\widetilde{\theta}\Big\}\,dv\,dx.
\end{align}

\medskip
\noindent{\it {Step 5.}}
As before, to get the dissipation rate for density function, we apply  $\partial^{\alpha}$ to the second equation of \eqref{2.22} with $1\leq|\alpha|\leq N-1$  
and then take the inner product of the resulting equation with  $\nabla_x\partial^{\alpha}\widetilde{\rho}$ to obtain
\begin{align*}	
&(\partial^{\alpha}\partial_t\widetilde{u},\nabla_x\partial^{\alpha}\widetilde{\rho})
+(\frac{2\bar{\theta}}{3\bar{\rho}}\nabla_x\partial^{\alpha}\widetilde{\rho},\nabla_x\partial^{\alpha}\widetilde{\rho})
=-\sum_{1\leq \alpha_{1}\leq \alpha}C_{\alpha}^{\alpha_{1}}(\partial^{\alpha_1}(\frac{2\bar{\theta}}{3\bar{\rho}})\nabla_x\partial^{\alpha-\alpha_{1}}\widetilde{\rho},\nabla_x\partial^{\alpha}\widetilde{\rho})
\nonumber\\
&
-\big(\partial^{\alpha}(u\cdot\nabla_x\widetilde{u})+\frac{2}{3}\nabla_x\partial^{\alpha}\widetilde{\theta}+\partial^{\alpha}(\widetilde{u}\cdot\nabla_x\bar{u})+\partial^{\alpha}[
\frac{2}{3}(\frac{\theta}{\rho}-\frac{\bar{\theta}}{\bar{\rho}})\nabla_x\rho],\nabla_x\partial^{\alpha}\widetilde{\rho}\big)
\nonumber\\
&-(\partial^{\alpha}(\frac{1}{\rho}\int_{\mathbb{R}^{3}} v\otimes v\cdot\nabla_x G\,dv),\nabla_x\partial^{\alpha}\widetilde{\rho}).
\end{align*}
Thanks to  $1\leq|\alpha|\leq N-1$, we follow the similar strategy as \eqref{4.21} to claim that
\begin{equation}	
\label{4.41a}
\varepsilon \frac{d}{dt}(\partial^{\alpha}\widetilde{u},\nabla_x\partial^{\alpha}\widetilde{\rho})
+c\varepsilon\|\nabla_x\partial^{\alpha}\widetilde{\rho}\|^2\leq
C\varepsilon(\|\nabla_x\partial^{\alpha}\widetilde{u}\|^2+\|\nabla_x\partial^{\alpha}\widetilde{\theta}\|^2+\|\nabla_x \partial^{\alpha}f\|^2_\sigma
)+C(\eta_0+\varepsilon^{\frac{1}{2}})\varepsilon^{2}.
\end{equation}
For any $1\leq|\alpha|\leq N-1$, the
summation of \eqref{4.41a} over $|\alpha|$ through a suitable linear combination gives
\begin{align}	
\label{4.40}
&\varepsilon\sum_{1\leq |\alpha|\leq N-1} \frac{d}{dt}(\partial^{\alpha}\widetilde{u},\nabla_x\partial^{\alpha}\widetilde{\rho})
+c\varepsilon\sum_{2\leq |\alpha|\leq N}\|\partial^{\alpha}\widetilde{\rho}\|^2
\nonumber\\
&\leq C\varepsilon\sum_{2\leq |\alpha|\leq N}
(\|\partial^{\alpha}\widetilde{u}\|^2+\|\partial^{\alpha}\widetilde{\theta}\|^2+\| \partial^{\alpha}f\|^2_\sigma)
+C(\eta_0+\varepsilon^{\frac{1}{2}})\varepsilon^2.
\end{align}

\medskip
\noindent{\it {Step 6.}}
Multiplying \eqref{4.38} by a large constant $C_2>1$ and adding the resulted equation with \eqref{4.40}
together, we have by choosing $\eta>0$ small enough that
\begin{align}
\label{4.44a}
\frac{1}{2}&C_2\sum_{1\leq |\alpha|\leq N-1}\frac{d}{dt}\int_{\mathbb{R}^{3}}( \frac{2\bar{\theta}}{3\bar{\rho}^{2}}|\partial^{\alpha}\widetilde{\rho}|^{2}+|\partial^{\alpha}\widetilde{u}|^{2}
+\frac{1}{\bar{\theta}}|\partial^{\alpha}\widetilde{\theta}|^{2})
\,dx+C_2\frac{d}{dt}E_1(t)
\nonumber\\
&+\varepsilon\sum_{1\leq |\alpha|\leq N-1} \frac{d}{dt}(\partial^{\alpha}\widetilde{u},\nabla_x\partial^{\alpha}\widetilde{\rho})
+c\varepsilon\sum_{2\leq |\alpha|\leq N}\|\partial^{\alpha}(\widetilde{\rho},\widetilde{u},\widetilde{\theta})\|^{2}
\nonumber\\	
\leq&C\varepsilon\sum_{1\leq |\alpha|\leq N}\|\partial^{\alpha}f\|_\sigma^{2}+C(\eta_{0}+\varepsilon^{\frac{1}{2}})\varepsilon^{2}
+C(\eta_{0}+\varepsilon^{\frac{1}{2}})\mathcal{D}_N(t).
\end{align}
Recall $E_1(t)$ in \eqref{4.40a}. We then employ \eqref{4.16}, \eqref{1.7} and \eqref{3.1} to get
\begin{align*}
|E_1(t)|\leq C\eta 
\sum_{1\leq |\alpha|\leq N-1}\|\partial^{\alpha}(\widetilde{\rho},\widetilde{u},\widetilde{\theta})\|^2+C_{\eta}\varepsilon^2\sum_{|\alpha|=N}\|\partial^{\alpha}f\|^2+
C_{\eta}(\eta_{0}+\varepsilon^{\frac{1}{2}})\varepsilon^{2}.
\end{align*}
By integrating \eqref{4.44a} with respect to $t$,
we deduce the desired estimate \eqref{4.41} from  the above estimate on $E_1(t)$ with a small $\eta>0$ and
\eqref{3.5}, as well as \eqref{3.1}. This then completes the proof of Lemma \ref{lem.sdefpN1}.
\end{proof}

\subsection{Space derivative estimate on non-fluid part up to $(N-1)$-order}
Next we deduce the space derivative estimate up to $(N-1)$-order for the non-fluid part $f(t,x,v)$. As before, the proof is based on the microscopic equation \eqref{2.21}.

\begin{lemma}\label{lem.sdenfpN1}
It holds that
\begin{align}
\label{4.44}
&\sum_{1\leq|\alpha|\leq N-1}(\|\partial^{\alpha}f(t)\|^{2}+c\frac{1}{\varepsilon}\int^t_0\|\partial^{\alpha}f(s)\|_{\sigma}^{2}\,ds)
\nonumber\\
\leq& C\varepsilon\sum_{2\leq|\alpha|\leq N}\int^t_0\|(\partial^{\alpha}\widetilde{u},\partial^{\alpha}\widetilde{\theta})(s)\|^{2}\,ds
+C\varepsilon\sum_{|\alpha|= N}\int^t_0\|\partial^{\alpha}f(s)\|_{\sigma}^{2}\,ds
\nonumber\\
&+C(\eta_{0}+\varepsilon^{\frac{1}{2}})\int^t_0\mathcal{D}_N(s)\,ds
+C(1+t)(\eta_{0}+\varepsilon^{\frac{1}{2}})\varepsilon^{2}.
\end{align}
\end{lemma}

\begin{proof}
Applying $\partial^{\alpha}$ to \eqref{2.21} with $1\leq|\alpha|\leq N-1$ and taking the inner product of the resulting equation with
$\partial^{\alpha}f$ over $\mathbb{R}^{3}\times\mathbb{R}^{3}$, we obtain
\begin{align}
\label{4.42}
&\frac{1}{2}\frac{d}{dt}\|\partial^{\alpha}f\|^{2}+
c_1 \frac{1}{\varepsilon}\|\partial^{\alpha}f\|_{\sigma}^{2}\notag\\
&\leq\frac{1}{\varepsilon}(\partial^{\alpha}\Gamma(\frac{M-\mu}{\sqrt{\mu}},f)
+\partial^{\alpha}\Gamma(f,\frac{M-\mu}{\sqrt{\mu}}),\partial^{\alpha}f)+\frac{1}{\varepsilon}(\partial^{\alpha}\Gamma(\frac{G}{\sqrt{\mu}},\frac{G}{\sqrt{\mu}}),\partial^{\alpha}f)
\nonumber\\
&\quad+(\frac{\partial^{\alpha}P_{0}(v\sqrt{\mu}\cdot\nabla_{x}f)}{\sqrt{\mu}},\partial^{\alpha}f)
-(\frac{\partial^{\alpha}P_{1}(v\cdot\nabla_{x}\overline{G})}{\sqrt{\mu}},\partial^{\alpha}f)
-(\frac{\partial^{\alpha}\partial_{t}\overline{G}}{\sqrt{\mu}},\partial^{\alpha}f)
\nonumber\\
&\quad-(\frac{1}{\sqrt{\mu}}\partial^{\alpha}P_{1}\{v\cdot(\frac{|v-u|^{2}
	\nabla_{x}\widetilde{\theta}}{2R\theta^{2}}
+\frac{(v-u)\cdot\nabla_{x}\widetilde{u}}{R\theta})M\},\partial^{\alpha}f).
\end{align}
We now compute \eqref{4.42} term by term.
With  Lemma \ref{lem5.7} and Lemma \ref{lem5.8} in hand, it is clear to see
\begin{align*}
\frac{1}{\varepsilon}(\partial^{\alpha}\Gamma(\frac{M-\mu}{\sqrt{\mu}},f)+\partial^{\alpha}\Gamma(f,\frac{M-\mu}{\sqrt{\mu}}),\partial^{\alpha}f)
\leq C\eta\frac{1}{\varepsilon}\|\partial^{\alpha}f\|^{2}_{\sigma}+C_{\eta}(\eta_{0}+\varepsilon^{\frac{1}{2}})\mathcal{D}_N(t),
\end{align*}
and
\begin{equation*}
\frac{1}{\varepsilon}|(\partial^{\alpha}\Gamma(\frac{G}{\sqrt{\mu}},\frac{G}{\sqrt{\mu}}),\partial^{\alpha}f)|\leq  C\eta\frac{1}{\varepsilon}\|\partial^{\alpha}f\|^{2}_{\sigma}+C_{\eta}(\eta_{0}+\varepsilon^{\frac{1}{2}})\mathcal{D}_N(t)+C_{\eta}(\eta_{0}+\varepsilon^{\frac{1}{2}})\varepsilon^2.
\end{equation*}
In view of \eqref{2.5}, \eqref{2.4}, \eqref{2.26}, \eqref{1.7}, \eqref{3.1} and the Cauchy-Schwarz and Sobolev inequalities, we arrive at
\begin{align*}
&|(\frac{1}{\sqrt{\mu}}\partial^{\alpha}P_{0}(v\sqrt{\mu}\cdot\nabla_{x}f),\partial^{\alpha}f)|
\\
&\leq \eta\frac{1}{\varepsilon}\|\langle v\rangle^{-\frac{1}{2}}\partial^{\alpha}f\|^{2}+
C_{\eta}\varepsilon\|\langle v\rangle^{\frac{1}{2}}\frac{1}{\sqrt{\mu}}\partial^{\alpha}P_{0}(v\sqrt{\mu}\cdot\nabla_{x}f)\|^{2}
\\
&\leq C\eta\frac{1}{\varepsilon}\|\partial^{\alpha}f\|_{\sigma}^{2}
+C_{\eta}\varepsilon\|\partial^{\alpha}\nabla_{x}f\|_{\sigma}^{2}+C_{\eta}(\eta_{0}+\varepsilon^{\frac{1}{2}})\varepsilon^2.
\end{align*}
To deal with the term involving $\overline{G}$, we use Lemma \ref{lem5.4},  Lemma \ref{lem5.10}, \eqref{2.5}, \eqref{1.7}, \eqref{2.26} and \eqref{3.1}
to get
\begin{align*}
&|(\frac{\partial^{\alpha}P_{1}(v\cdot\nabla_{x}\overline{G})}{\sqrt{\mu}},\partial^{\alpha}f)|+
|(\frac{\partial^{\alpha}\partial_{t}\overline{G}}{\sqrt{\mu}},\partial^{\alpha}f)|
\\
&\leq \eta\frac{1}{\varepsilon}\|\langle v\rangle^{-\frac{1}{2}}\partial^{\alpha}f\|^{2}
+C_{\eta}\varepsilon(\|\langle v\rangle^{\frac{1}{2}}\frac{\partial^{\alpha}P_{1}(v\cdot\nabla_{x}\overline{G})}{\sqrt{\mu}}\|^{2}
+\|\langle v\rangle^{\frac{1}{2}}\frac{\partial^{\alpha}\partial_{t}\overline{G}}{\sqrt{\mu}}\|^{2})
\\
&\leq C\eta\frac{1}{\varepsilon}\|\partial^{\alpha}f\|_{\sigma}^{2}+C_{\eta}(\eta_{0}+\varepsilon^{\frac{1}{2}})\varepsilon^2.
\end{align*}
Thanks to \eqref{2.5} and \eqref{5.1},  the following identity holds true:
\begin{multline*}
P_{1}v\cdot\{\frac{|v-u|^{2}\nabla_{x}\widetilde{\theta}}{2R\theta^{2}}+\frac{(v-u)\cdot\nabla_{x}\widetilde{u}}{R\theta}\}M\\
=\frac{\sqrt{R}}{\sqrt{\theta}}\sum_{j=1}^{3}\frac{\partial\widetilde{\theta}}{\partial x_{j}}\hat{A}_{j}(\frac{v-u}{\sqrt{R\theta}})M
+\sum_{j=1}^{3}\sum_{i=1}^{3}\frac{\partial\widetilde{u}_{j}}{\partial x_{i}}\hat{B}_{ij}(\frac{v-u}{\sqrt{R\theta}})M.
\end{multline*}
We then see that the last term of \eqref{4.42} can be bounded by
\begin{align*}	&|(\frac{1}{\sqrt{\mu}}\partial^{\alpha}P_{1}\big\{v\cdot(\frac{|v-u|^{2}\nabla_{x}\widetilde{\theta}}{2R\theta^{2}}+\frac{(v-u)\cdot\nabla_{x}\widetilde{u}}{R\theta})M\big\},\partial^{\alpha}f)|
\\
&\leq 
\eta\frac{1}{\varepsilon}\|\langle v\rangle^{-\frac{1}{2}}\partial^{\alpha}f\|^{2}
+C_{\eta}\varepsilon\|\langle v\rangle^{\frac{1}{2}}\frac{1}{\sqrt{\mu}}\partial^{\alpha}P_{1}\big\{v\cdot(\frac{|v-u|^{2}
\nabla_{x}\widetilde{\theta}}{2R\theta^{2}}+\frac{(v-u)\cdot\nabla_{x}\widetilde{u}}{R\theta})M\big\}\|^{2}
\\
&\leq C\eta\frac{1}{\varepsilon}\|\partial^{\alpha}f\|_{\sigma}^{2}+ C_{\eta}\varepsilon\|\partial^{\alpha}(\nabla_{x}\widetilde{u},\nabla_{x}\widetilde{\theta})\|^{2}
+C_{\eta}(\eta_{0}+\varepsilon^{\frac{1}{2}})\varepsilon^2.
\end{align*}
Therefore, substituting all the above estimates into \eqref{4.42} and choosing  a small $\eta>0$, we get
\begin{align}
\label{4.43}
\frac{d}{dt}\|\partial^{\alpha}f\|^{2}+\frac{c_1}{2}\frac{1}{\varepsilon}\|\partial^{\alpha}f\|_{\sigma}^{2}
&\leq C\varepsilon(\|\partial^{\alpha}(\nabla_{x}\widetilde{u},\nabla_{x}\widetilde{\theta})\|^{2}+\|\partial^{\alpha}\nabla_{x}f\|_{\sigma}^{2})\notag\\
&\quad +C(\eta_{0}+\varepsilon^{\frac{1}{2}})\mathcal{D}_N(t)+C(\eta_{0}+\varepsilon^{\frac{1}{2}})\varepsilon^2.
\end{align}
Integrating \eqref{4.43} with respect to $t$ and using  \eqref{3.5}, the summation of the resulting equation over $|\alpha|$ with $1\leq|\alpha|\leq N-1$ with a suitable linear combination gives the desired estimate \eqref{4.44}. This completes the proof of Lemma \ref{lem.sdenfpN1}.
\end{proof}

From Lemma \ref{lem.sdenfpN1} and Lemma \ref{lem.sdefpN1}, we immediately have the full energy estimate on space derivatives up to $(N-1)$-order for both fluid and non-fluid parts.

\begin{lemma}\label{lem.fpnfpN1}
It holds that
\begin{align}
\label{4.45}
&\sum_{1\leq |\alpha|\leq N-1}(\|\partial^{\alpha}(\widetilde{\rho},\widetilde{u},\widetilde{\theta})(t)\|^2+\|\partial^{\alpha}f(t)\|^{2}
+c\frac{1}{\varepsilon}\int^t_0\|\partial^{\alpha}f(s)\|_{\sigma}^{2}\,ds)
\nonumber\\
&\hspace{1cm}+c\varepsilon\sum_{2\leq |\alpha|\leq N}\int^t_0\|\partial^{\alpha}(\widetilde{\rho},\widetilde{u},\widetilde{\theta})(s)\|^2\,ds
\nonumber\\
\leq& 
 C\varepsilon^2\sum_{|\alpha|=N}
(\|\partial^{\alpha}\widetilde{\rho}(t)\|^2+
\|\partial^{\alpha}f(t)\|^2)+ C\varepsilon\sum_{|\alpha|=N}\int^t_0\| \partial^{\alpha}f(s)\|^2_\sigma\,ds
\nonumber\\
&+C(\eta_{0}+\varepsilon^{\frac{1}{2}})\int^t_0\mathcal{D}_N(s)\,ds
+C(1+t)(\eta_{0}+\varepsilon^{\frac{1}{2}})\varepsilon^{2}.
\end{align}
\end{lemma}

\begin{proof}
Taking the summation of  \eqref{4.44} with $\eqref{4.41}$ multiplied by a large constant $C_3>0$ and using  the smallness of  $\varepsilon$ directly gives \eqref{4.45}. 
\end{proof}
 
\subsection{$N$-order space derivative estimate}

To complete the estimate on space derivatives of all orders, we need to treat the highest $N$-order space derivatives more carefully. Note that for $|\alpha|=N$, we cannot directly obtain the dissipation of $\|\partial^{\alpha}f\|^2_{\sigma}$ in terms of \eqref{2.21} since the estimate on the transport term $(\frac{1}{\sqrt{\mu}}\partial^{\alpha}P_{0}[v\cdot\nabla_{x}(\sqrt{\mu}f)],\partial^{\alpha}f)$ induces $N$+$1$-order
derivatives so that the estimates cannot be closed. For this, we must use the original equation \eqref{1.1} to deduce the $N$-order energy estimates. In fact, in view of \eqref{2.19} and \eqref{2.20}, the original equation \eqref{1.1} can be equivalently rewritten as
\begin{align}
\label{4.46}
\partial_{t}(\frac{F}{\sqrt{\mu}})+v\cdot\nabla_{x}(\frac{F}{\sqrt{\mu}})
-\frac{1}{\varepsilon}\mathcal{L}f=&\frac{1}{\varepsilon}\Gamma(\frac{M-\mu}{\sqrt{\mu}},f)+\frac{1}{\varepsilon}\Gamma(f,\frac{M-\mu}{\sqrt{\mu}})+\frac{1}{\varepsilon}\Gamma(\frac{G}{\sqrt{\mu}},\frac{G}{\sqrt{\mu}})
\nonumber\\
&+\frac{1}{\sqrt{\mu}}P_{1}\big\{v\cdot(\frac{|v-u|^{2}\nabla_{x}\bar{\theta}}{2R\theta^{2}}+\frac{(v-u)\cdot\nabla_{x}\bar{u}}{R\theta}) M\big\}.
\end{align}
Using the above formulation, the corresponding linear  transport term induces the inner product $(\frac{1}{\sqrt{\mu}}v\cdot\nabla_x\partial^{\alpha}F,\frac{1}{\sqrt{\mu}}\partial^{\alpha}F)$, which vanishes by the integration by parts. Extra efforts have to be made to estimate all other terms, such as $\frac{1}{\varepsilon}\mathcal{L}f$, $\frac{1}{\varepsilon}\Gamma(\frac{M-\mu}{\sqrt{\mu}},f)$ and $\partial_{t}(\frac{F}{\sqrt{\mu}})$, because of singularity of the highest-order space derivatives. In particular, we need to develop delicate estimates on the inner products $-\frac{1}{\varepsilon}(\mathcal{L}\partial^{\alpha}f,\frac{\partial^{\alpha}F}{\sqrt{\mu}})$ and $\frac{1}{\varepsilon}(\Gamma(\frac{M-\mu}{\sqrt{\mu}},\partial^{\alpha}f),\frac{\partial^{\alpha}F}{\sqrt{\mu}})$ for $|\alpha|=N$. Thus, we first obtain the following lemma and the characterization of low bound of the energy norm $\|\frac{\partial^{\alpha}F(t)}{\sqrt{\mu}}\|^{2}$ in terms of $\|\partial^{\alpha}(\widetilde{\rho},\widetilde{u},\widetilde{\theta})\|^{2}+\|\partial^{\alpha}f\|^{2}$ will be treated in Lemma \ref{lem.NFlbdd} later on.

\begin{lemma}\label{lem.Nsde}
It holds that
\begin{align}
\label{4.65}
\varepsilon^{2}&\sum_{|\alpha|=N}(\|\partial^{\alpha}(\widetilde{\rho},\widetilde{u},\widetilde{\theta})(t)\|^{2}+\|\partial^{\alpha}f(t)\|^{2})
+c\varepsilon\sum_{|\alpha|=N}\int^{t}_{0}\|\partial^{\alpha}f(s)\|_{\sigma}^{2}\,ds
\nonumber\\
\leq&C(\eta_{0}+\varepsilon^{\frac{1}{2}})\int^{t}_{0}\mathcal{D}_N(s)\,ds+C(1+t)(\eta_{0}+\varepsilon^{\frac{1}{2}})\varepsilon^{2}.
\end{align}
\end{lemma}

\begin{proof}
 Applying $\partial^{\alpha}$ to \eqref{4.46} with $|\alpha|=N$ and taking the inner product of the resulting equation with $\frac{\partial^{\alpha}F}{\sqrt{\mu}}$ over $\mathbb{R}^{3}\times\mathbb{R}^{3}$, we obtain
\begin{align}
\label{4.47}
&\frac{1}{2}\frac{d}{dt}\|\frac{\partial^{\alpha}F}{\sqrt{\mu}}\|^{2}-\frac{1}{\varepsilon}(\mathcal{L}\partial^{\alpha}f,\frac{\partial^{\alpha}F}{\sqrt{\mu}})
=\frac{1}{\varepsilon}(\partial^{\alpha}\Gamma(\frac{M-\mu}{\sqrt{\mu}},f)+\partial^{\alpha}\Gamma(f,\frac{M-\mu}{\sqrt{\mu}}),\frac{\partial^{\alpha}F}{\sqrt{\mu}})
\nonumber\\
&+\frac{1}{\varepsilon}(\partial^{\alpha}\Gamma(\frac{G}{\sqrt{\mu}},\frac{G}{\sqrt{\mu}}),\frac{\partial^{\alpha}F}{\sqrt{\mu}})+(\frac{1}{\sqrt{\mu}}\partial^{\alpha}P_{1}\big\{v\cdot(\frac{|v-u|^{2}
\nabla_x\overline{\theta}}{2R\theta^{2}}+\frac{(v-u)\cdot\nabla_x\bar{u}}{R\theta}) M\big\},\frac{\partial^{\alpha}F}{\sqrt{\mu}}).
\end{align}
We now compute \eqref{4.47} term by term.
By the fact that  $F=M+\overline{G}+\sqrt{\mu}f$, 
the second on the left hand side of \eqref{4.47} becomes
\begin{align}
\label{4.52a}
-\frac{1}{\varepsilon}(\mathcal{L}\partial^{\alpha}f,\frac{\partial^{\alpha}F}{\sqrt{\mu}})=-
\frac{1}{\varepsilon}(\mathcal{L}\partial^{\alpha}f,\frac{\partial^{\alpha}M}{\sqrt{\mu}})
-\frac{1}{\varepsilon}(\mathcal{L}\partial^{\alpha}f,\partial^{\alpha}f)
-\frac{1}{\varepsilon}(\mathcal{L}\partial^{\alpha}f,\frac{\partial^{\alpha}\overline{G}}{\sqrt{\mu}}).
\end{align}
We first present the calculations  for the last two terms of \eqref{4.52a}, since the first term of \eqref{4.52a} is more complex and is thus left to the end. 
Thanks to $f\in (\ker\mathcal{L})^{\perp}$, one has from \eqref{2.22a} that
$$
-\frac{1}{\varepsilon}(\mathcal{L}\partial^{\alpha}f,\partial^{\alpha}f)\geq c_1\frac{1}{\varepsilon}\|\partial^{\alpha}f\|_{\sigma}^{2}. $$
By the definition  $\mathcal{L}f=\Gamma(\sqrt{\mu},f)+\Gamma(f,\sqrt{\mu})$ in \eqref{2.19}, one has from \eqref{5.11},
Lemma \ref{lem5.4}, \eqref{2.26}, \eqref{1.7} and \eqref{3.1} that
\begin{align*}
\frac{1}{\varepsilon}|(\mathcal{L}\partial^{\alpha}f,\frac{\partial^{\alpha}\overline{G}}{\sqrt{\mu}})|
&\leq C\frac{1}{\varepsilon}(\|\partial^{\alpha}f\|_{\sigma}+\|\langle v\rangle^{-1}\partial^{\alpha}f\|)
\|\frac{\partial^{\alpha}\overline{G}}{\sqrt{\mu}}\|_{\sigma}
\\
&\leq \eta_0\frac{1}{\varepsilon}\|\partial^{\alpha}f\|^{2}_{\sigma}
+C\eta_{0}\varepsilon.
\end{align*}
To estimate the first term of \eqref{4.52a}. Recall \eqref{4.19B},  one has
$$
\partial_{x_i}M=M\big(\frac{\partial_{x_i}\rho}{\rho}+\frac{(v-u)\cdot\partial_{x_i}u}{R\theta}
+(\frac{|v-u|^{2}}{2R\theta}-\frac{3}{2})\frac{\partial_{x_i}\theta}{\theta} \big).
$$
Let $\partial^{\alpha}=\partial^{\alpha'}\partial_{x_i}$ with $|\alpha'|=N-1$ due to $|\alpha|=N$ , then 
\begin{align}
\label{4.48}
\partial^{\alpha}M=&M\big(\frac{\partial^{\alpha'}\partial_{x_i}\rho}{\rho}+\frac{(v-u)\cdot\partial^{\alpha'}\partial_{x_i}u}{R\theta}+(\frac{|v-u|^{2}}{2R\theta}-\frac{3}{2})\frac{\partial^{\alpha'}\partial_{x_i}\theta}{\theta} \big)
\nonumber\\
&+\sum_{1\leq\alpha_{1}\leq \alpha'}C^{\alpha_1}_{\alpha'}\Big(\partial^{\alpha_{1}}(M\frac{1}{\rho})\partial^{\alpha'-\alpha_{1}}\partial_{x_i}\rho+\partial^{\alpha_{1}}(M\frac{v-u}{R\theta})\cdot\partial^{\alpha'-\alpha_{1}}\partial_{x_i}u
\nonumber\\
&\hspace{2cm}+\partial^{\alpha_{1}}(M\frac{|v-u|^{2}}{2R\theta^{2}}-M\frac{3}{2\theta})\partial^{\alpha'-\alpha_{1}}\partial_{x_i}\theta\Big)
\nonumber\\
&:=I_{4}+I_{5}.
\end{align}
Note that the linear term $\frac{1}{\varepsilon}(\mathcal{L}\partial^{\alpha}f,\frac{I_{4}}{\sqrt{\mu}})$ presents a significant difficulty and  cannot be estimated directly. The key technique to handle this term is to use the properties  of the linearized operator  $\mathcal{L}$ and the smallness of 
$M-\mu$. For this, we denote
\begin{align*}
I_{4}&=\big(\mu+(M-\mu)\big)\big(\frac{\partial^{\alpha}\rho}{\rho}+\frac{(v-u)\cdot\partial^{\alpha}u}{R\theta}+(\frac{|v-u|^{2}}{2R\theta}-\frac{3}{2})\frac{\partial^{\alpha}\theta}{\theta}\big):=I^{1}_{4}+I^{2}_{4}.
\end{align*}
Since $\frac{I^{1}_{4}}{\sqrt{\mu}}\in\ker{\mathcal{L}}$, it follows that
$(\mathcal{L}f,\frac{I^{1}_{4}}{\sqrt{\mu}})=0$. As for $I^{2}_{4}$, we further
decompose it as
\begin{align*}
I^{2}_{4}=&(M-\mu)\big(\frac{\partial^{\alpha}\widetilde{\rho}}{\rho}
+\frac{(v-u)\cdot\partial^{\alpha}\widetilde{u}}{R\theta}+(\frac{|v-u|^{2}}{2R\theta}
-\frac{3}{2})\frac{\partial^{\alpha}\widetilde{\theta}}{\theta}\big)
\\
&+(M-\mu)\big(\frac{\partial^{\alpha}\bar{\rho}}{\rho}
+\frac{(v-u)\cdot\partial^{\alpha}\bar{u}}{R\theta}+(\frac{|v-u|^{2}}{2R\theta}
-\frac{3}{2})\frac{\partial^{\alpha}\bar{\theta}}{\theta}\big):=I^{21}_{4}+I^{22}_{4}.
\end{align*}
Thanks to $\mathcal{L}f=\Gamma(\sqrt{\mu},f)+\Gamma(f,\sqrt{\mu})$, we deduce from \eqref{5.11}, \eqref{2.26} and \eqref{5.16}  that
\begin{align*}
\frac{1}{\varepsilon}|(\mathcal{L}\partial^{\alpha}f,\frac{I^{21}_{4}}{\sqrt{\mu}})|
&\leq C(\eta_{0}+\varepsilon)\frac{1}{\varepsilon}(\|\partial^{\alpha}f\|_{\sigma}+\|\langle v\rangle^{-1}\partial^{\alpha}f\|)\|\partial^{\alpha}(\widetilde{\rho},\widetilde{u},\widetilde{\theta})\|
\\
&\leq C(\eta_{0}+\varepsilon)\frac{1}{\varepsilon}(\|\partial^{\alpha}f\|^{2}_{\sigma}+\|\partial^{\alpha}(\widetilde{\rho},\widetilde{u},\widetilde{\theta})\|^{2}).
\end{align*}
Let $\partial^{\alpha}=\partial^{\alpha'}\partial_{x_i}$ with $|\alpha'|=N-1$ due to $|\alpha|=N$, we have from
the integration by parts,  \eqref{5.11}, \eqref{2.26}, \eqref{1.7} and \eqref{3.1} that
\begin{align}
\label{4.49}
\frac{1}{\varepsilon}|(\mathcal{L}\partial^{\alpha}f,\frac{I^{22}_{4}}{\sqrt{\mu}})|=&\frac{1}{\varepsilon}|\big(\partial^{\alpha'}\partial_{x_i}\mathcal{L}f,\frac{M-\mu}{\sqrt{\mu}}\{\frac{\partial^{\alpha}\bar{\rho}}{\rho}+\frac{(v-u)\cdot\partial^{\alpha}\bar{u}}{R\theta}+(\frac{|v-u|^{2}}{2R\theta}-\frac{3}{2})\frac{\partial^{\alpha}\bar{\theta}}{\theta}\}\big)|
\nonumber\\
=&\frac{1}{\varepsilon}|(\partial^{\alpha'}\mathcal{L}f,\partial_{x_i}\big[\frac{M-\mu}{\sqrt{\mu}}\{\frac{\partial^{\alpha}\bar{\rho}}{\rho}+\frac{(v-u)\cdot\partial^{\alpha}\bar{u}}{R\theta}+(\frac{|v-u|^{2}}{2R\theta}-\frac{3}{2})\frac{\partial^{\alpha}\bar{\theta}}{\theta}\}\big])|
\nonumber\\
\leq& C\frac{1}{\varepsilon}\int_{\mathbb{R}^{3}}
(|\partial^{\alpha'}f|_{\sigma}+|\langle v\rangle^{-1}\partial^{\alpha'}f|_2)
\big\{(|\partial_{x_i}\partial^{\alpha}\bar{\rho}|+|\partial_{x_i}\partial^{\alpha}\bar{u}|+|\partial_{x_i}\partial^{\alpha}\bar{\theta}|)
\nonumber\\
&\hspace{1cm}+(|\partial^{\alpha}\bar{\rho}|+|\partial^{\alpha}\bar{u}|+|\partial^{\alpha}\bar{\theta}|)(|\partial_{x_i}\rho|+|\partial_{x_i}u|+|\partial_{x_i}\theta|)\big\}\,dx
\nonumber\\
\leq& C(\eta_{0}+\varepsilon)\frac{1}{\varepsilon}\|\partial^{\alpha'}f\|_{\sigma}
\leq C(\eta_{0}+\varepsilon)\frac{1}{\varepsilon}(\frac{1}{\varepsilon^{2}}\|\partial^{\alpha'}f\|^{2}_{\sigma}+\varepsilon^2).
\end{align} 
Combining the above estimates on $I^{1}_{4}$ and $I^2_{4}$, we obtain that for $|\alpha'|=N-1$,
\begin{equation}
\label{4.50}
\frac{1}{\varepsilon}|(\mathcal{L}\partial^{\alpha}f,\frac{I_{4}}{\sqrt{\mu}})|=\frac{1}{\varepsilon}|(\mathcal{L}\partial^{\alpha}f,\frac{I^2_{4}}{\sqrt{\mu}})|
\leq C(\eta_{0}+\varepsilon)\frac{1}{\varepsilon}(\|\partial^{\alpha}f\|^{2}_{\sigma}+\|\partial^{\alpha}(\widetilde{\rho},\widetilde{u},\widetilde{\theta})\|^{2}+\frac{1}{\varepsilon^{2}}\|\partial^{\alpha'} f\|^{2}_{\sigma}+\varepsilon^{2}).
\end{equation}
In order to get the estimates on $I_5$, we first write  $I_{5}=I^{1}_{5}+I^{2}_{5}$ with
\begin{align*}
I_{5}^1=\sum_{1\leq\alpha_{1}\leq \alpha'}C^{\alpha_1}_{\alpha'}\Big(\partial^{\alpha_{1}}(M\frac{1}{\rho})\partial^{\alpha'-\alpha_{1}}\partial_{x_i}\widetilde{\rho}&+\partial^{\alpha_{1}}(M\frac{v-u}{R\theta})\cdot\partial^{\alpha'-\alpha_{1}}\partial_{x_i}\widetilde{u}\\
&+\partial^{\alpha_{1}}(M\frac{|v-u|^{2}}{2R\theta^{2}}-M\frac{3}{2\theta})\partial^{\alpha'-\alpha_{1}}\partial_{x_i}\widetilde{\theta}\Big),
\end{align*}
and 
\begin{align*}
I_{5}^2=\sum_{1\leq\alpha_{1}\leq \alpha'}C^{\alpha_1}_{\alpha'}\Big(\partial^{\alpha_{1}}(M\frac{1}{\rho})\partial^{\alpha'-\alpha_{1}}\partial_{x_i}\bar{\rho}&+\partial^{\alpha_{1}}(M\frac{v-u}{R\theta})\cdot\partial^{\alpha'-\alpha_{1}}\partial_{x_i}\bar{u}\\
&+\partial^{\alpha_{1}}(M\frac{|v-u|^{2}}{2R\theta^{2}}-M\frac{3}{2\theta})\partial^{\alpha'-\alpha_{1}}\partial_{x_i}\bar{\theta}\Big).
\end{align*}
For $I^{1}_{5}$, using $\mathcal{L}f=\Gamma(\sqrt{\mu},f)+\Gamma(f,\sqrt{\mu})$, \eqref{5.11}, \eqref{2.26} and  \eqref{3.6} again, we see that $\frac{1}{\varepsilon}|(\mathcal{L}\partial^{\alpha}f,\frac{I^{1}_{5}}{\sqrt{\mu}})|$ is bounded as
\begin{equation*}
C\sum_{1\leq\alpha_{1}\leq \alpha'}\frac{1}{\varepsilon}\int_{\mathbb{R}^{3}}
|\partial^{\alpha}f|_{\sigma}|\partial^{\alpha'-\alpha_{1}}\partial_{x_i}(\widetilde{\rho}, \widetilde{u},\widetilde{\theta})|
(|\partial^{\alpha_{1}}(\rho, u,\theta)|+\cdots+|\nabla_{x}(\rho, u,\theta)|^{|\alpha_{1}|})
\,dx.
\end{equation*}
If $|\alpha_{1}|=|\alpha'|=N-1$, then $|\alpha'-\alpha_{1}|=0$, we take the $L^{6}-L^{3}-L^{2}$ H\"{o}lder inequality and use Lemma \ref{lem5.1}, \eqref{1.7} and \eqref{3.1},
as well as the Cauchy-Schwarz and Sobolev inequalities, to get
\begin{align*}
&\frac{1}{\varepsilon}\int_{\mathbb{R}^{3}}
|\partial^{\alpha}f|_{\sigma}|\partial^{\alpha'-\alpha_{1}}\partial_{x_i}(\widetilde{\rho}, \widetilde{u},\widetilde{\theta})|
(|\partial^{\alpha_{1}}(\rho, u,\theta)|+\cdot\cdot\cdot+|\nabla_{x}(\rho, u,\theta)|^{|\alpha_{1}|})\,dx
\\
&\leq C\frac{1}{\varepsilon}\||\partial^{\alpha}f|_{\sigma}\|_{L^{2}}\|\partial_{x_i}(\widetilde{\rho},\widetilde{u},\widetilde{\theta})\|_{L^{6}}
\|(|\partial^{\alpha_{1}}(\rho, u,\theta)|+\cdot\cdot\cdot+|\nabla_{x}(\rho, u,\theta)|^{|\alpha_{1}|})\|_{L^{3}}  
\\
&\leq C(\eta_{0}+\varepsilon^{\frac{1}{2}})\frac{1}{\varepsilon}\|\partial^{\alpha}f\|_{\sigma}
\|\partial_{x_i}(\widetilde{\rho},\widetilde{u},\widetilde{\theta})\|_{H^1}
\leq C(\eta_{0}+\varepsilon^{\frac{1}{2}})\frac{1}{\varepsilon}(\|\partial^{\alpha}f\|^{2}_{\sigma}
+\varepsilon^{2}).
\end{align*}
If $1\leq|\alpha_{1}|<|\alpha'|=N-1$, we also have
\begin{align*}
&\frac{1}{\varepsilon}\int_{\mathbb{R}^{3}}
|\partial^{\alpha}f|_{\sigma}|\partial^{\alpha'-\alpha_{1}}\partial_{x_i}(\widetilde{\rho}, \widetilde{u},\widetilde{\theta})|
(|\partial^{\alpha_{1}}(\rho, u,\theta)|+\cdot\cdot\cdot+|\nabla_{x}(\rho, u,\theta)|^{|\alpha_{1}|})\,dx
\\
&\leq C\frac{1}{\varepsilon}\||\partial^{\alpha}f|_{\sigma}\|_{L^{2}}\|\partial^{\alpha'-\alpha_{1}}\partial_{x_i}(\widetilde{\rho},\widetilde{u},\widetilde{\theta})\|_{L^{2}}\|(|\partial^{\alpha_{1}}(\rho, u,\theta)|+\cdot\cdot\cdot+|\nabla_{x}(\rho, u,\theta)|^{|\alpha_{1}|})\|_{L^{\infty}}  
\\
&
\leq C(\eta_{0}+\varepsilon^{\frac{1}{2}})\frac{1}{\varepsilon}(\|\partial^{\alpha}f\|^{2}_{\sigma}+\varepsilon^{2}).
\end{align*}
It follows from the above three estimates that
\begin{align}
\label{4.51}
\frac{1}{\varepsilon}|(\mathcal{L}\partial^{\alpha}f,\frac{I^{1}_{5}}{\sqrt{\mu}})|\leq
C(\eta_{0}+\varepsilon^{\frac{1}{2}})\frac{1}{\varepsilon}(\|\partial^{\alpha}f\|^{2}_{\sigma}+\varepsilon^{2}).
\end{align}
The term $I^{2}_{5}$ can be treated in the similar way as \eqref{4.49} and it holds that
\begin{align*}
\frac{1}{\varepsilon}|(\mathcal{L}\partial^{\alpha}f,\frac{I^{2}_{5}}{\sqrt{\mu}})|
\leq C(\eta_{0}+\varepsilon^{\frac{1}{2}})\frac{1}{\varepsilon}(\frac{1}{\varepsilon^{2}}\|\partial^{\alpha'}f\|^{2}_{\sigma}+\varepsilon^{2}).
\end{align*}
This estimate together with \eqref{4.51} yields
\begin{align}
\label{4.52}
\frac{1}{\varepsilon}|(\mathcal{L}\partial^{\alpha}f,\frac{I_{5}}{\sqrt{\mu}})|\leq C(\eta_{0}+\varepsilon^{\frac{1}{2}})\frac{1}{\varepsilon}(\|\partial^{\alpha}f\|^{2}_{\sigma}+\frac{1}{\varepsilon^{2}}\|\partial^{\alpha'} f\|^{2}_{\sigma}+\varepsilon^{2}).
\end{align}
Combining \eqref{4.48}, \eqref{4.50} and \eqref{4.52}, we get
\begin{align}
\label{4.60b}
\frac{1}{\varepsilon}|(\mathcal{L}\partial^{\alpha}f,\frac{\partial^{\alpha}M}{\sqrt{\mu}})|&\leq C(\eta_{0}+\varepsilon^{\frac{1}{2}})\frac{1}{\varepsilon}(\|\partial^{\alpha}f\|^{2}_{\sigma}+\|\partial^{\alpha}(\widetilde{\rho},\widetilde{u},\widetilde{\theta})\|^{2}+\frac{1}{\varepsilon^{2}}\|\partial^{\alpha'} f\|^{2}_{\sigma}+\varepsilon^{2})
\nonumber\\
&\leq C(\eta_{0}+\varepsilon^{\frac{1}{2}})\frac{1}{\varepsilon^{2}}\mathcal{D}_N(t)+C(\eta_{0}+\varepsilon^{\frac{1}{2}})\varepsilon.
\end{align}
Recalling \eqref{4.52a} and collecting all estimates above, we thereby obtain
\begin{align}
\label{4.53}
-\frac{1}{\varepsilon}(\mathcal{L}\partial^{\alpha}f,\frac{\partial^{\alpha}F}{\sqrt{\mu}})
\geq& \frac{c_1}{2}\frac{1}{\varepsilon}\|\partial^{\alpha}f\|_{\sigma}^{2}
-C(\eta_{0}+\varepsilon^{\frac{1}{2}})\frac{1}{\varepsilon^{2}}\mathcal{D}_N(t)-C(\eta_{0}+\varepsilon^{\frac{1}{2}})\varepsilon.
\end{align}

Let us now consider the terms on the right hand of \eqref{4.47}. First note that
\begin{equation*}
\partial^{\alpha}\Gamma(\frac{M-\mu}{\sqrt{\mu}},f)=
\Gamma(\frac{M-\mu}{\sqrt{\mu}},\partial^{\alpha}f)
+\sum_{1\leq\alpha_{1}\leq\alpha}C^{\alpha_{1}}_{\alpha}\Gamma(\frac{\partial^{\alpha_{1}}(M-\mu)}{\sqrt{\mu}},
\partial^{\alpha-\alpha_{1}}f).
\end{equation*}
By the fact that $F=M+\overline{G}+\sqrt{\mu}f$, one writes
\begin{equation}
\label{4.54}
\frac{1}{\varepsilon}(\Gamma(\frac{M-\mu}{\sqrt{\mu}},\partial^{\alpha}f),\frac{\partial^{\alpha}F}{\sqrt{\mu}})
=\frac{1}{\varepsilon}(\Gamma(\frac{M-\mu}{\sqrt{\mu}},\partial^{\alpha}f),\frac{\partial^{\alpha}M}{\sqrt{\mu}}
+\frac{\partial^{\alpha}\overline{G}}{\sqrt{\mu}}+\partial^{\alpha}f).
\end{equation}
Recalling \eqref{4.48}, the first term of \eqref{4.54} can be denoted as
\begin{align*}
\frac{1}{\varepsilon}(\Gamma(\frac{M-\mu}{\sqrt{\mu}},\partial^{\alpha}f),\frac{\partial^{\alpha}M}{\sqrt{\mu}})=\frac{1}{\varepsilon}(\Gamma(\frac{M-\mu}{\sqrt{\mu}},\partial^{\alpha}f),\frac{I_{4}+I_{5}}{\sqrt{\mu}}).
\end{align*}
Using the definition of $I_4$ in \eqref{4.48} again, one has
\begin{align*}
&\frac{1}{\varepsilon}\big(\Gamma(\frac{M-\mu}{\sqrt{\mu}},\partial^{\alpha}f),\frac{I_{4}}{\sqrt{\mu}}\big)\\
&=\frac{1}{\varepsilon}\Big(\Gamma(\frac{M-\mu}{\sqrt{\mu}},\partial^{\alpha}f),\frac{M}{\sqrt{\mu}}
\{\frac{\partial^{\alpha}\widetilde{\rho}}{\rho}+\frac{(v-u)\cdot\partial^{\alpha}\widetilde{u}}{R\theta}
+(\frac{|v-u|^{2}}{2R\theta}-\frac{3}{2})\frac{\partial^{\alpha}\widetilde{\theta}}{\theta}\}\Big)
\\
&\qquad+\frac{1}{\varepsilon}\Big(\Gamma(\frac{M-\mu}{\sqrt{\mu}},\partial^{\alpha}f),\frac{M}{\sqrt{\mu}}\{\frac{\partial^{\alpha}\bar{\rho}}{\rho}+\frac{(v-u)\cdot\partial^{\alpha}\bar{u}}{R\theta}
+(\frac{|v-u|^{2}}{2R\theta}-\frac{3}{2})\frac{\partial^{\alpha}\bar{\theta}}{\theta}\}\Big)
\\
&\leq  C(\eta_{0}+\varepsilon)\frac{1}{\varepsilon}(\|\partial^{\alpha}f\|^{2}_{\sigma}+\|\partial^{\alpha}(\widetilde{\rho},\widetilde{u},\widetilde{\theta})\|^{2}+\frac{1}{\varepsilon^{2}}\|\partial^{\alpha'} f\|^{2}_{\sigma}+\varepsilon^{2}),
\end{align*}
where we have used the similar arguments as \eqref{4.50} and $\partial^{\alpha}=\partial^{\alpha'}\partial_{x_i}$ with $|\alpha'|=N-1$.
On the other hand, performing the similar calculations as \eqref{4.52} gives
\begin{align*}
\frac{1}{\varepsilon}(\Gamma(\frac{M-\mu}{\sqrt{\mu}},\partial^{\alpha}f),\frac{I_{5}}{\sqrt{\mu}})
\leq C(\eta_{0}+\varepsilon^{\frac{1}{2}})\frac{1}{\varepsilon}(\|\partial^{\alpha}f\|^{2}_{\sigma}+\frac{1}{\varepsilon^{2}}\|\partial^{\alpha'} f\|^{2}_{\sigma}+\varepsilon^{2}).
\end{align*}
With the help of the above two estimates,
the first term of \eqref{4.54} is bounded by
\begin{align*}
\frac{1}{\varepsilon}|(\Gamma(\frac{M-\mu}{\sqrt{\mu}},\partial^{\alpha}f),\frac{\partial^{\alpha}M}{\sqrt{\mu}})| 
\leq C(\eta_{0}+\varepsilon^{\frac{1}{2}})\frac{1}{\varepsilon}(\|\partial^{\alpha}f\|^{2}_{\sigma}+\|\partial^{\alpha}(\widetilde{\rho},\widetilde{u},\widetilde{\theta})\|^{2}+\frac{1}{\varepsilon^{2}}\|\partial^{\alpha'} f\|^{2}_{\sigma}+\varepsilon^{2}),
\end{align*}
for $|\alpha'|=N-1$.
The estimates on the last two terms of \eqref{4.54} will be much easier and they are controlled
by
\begin{align*}
\frac{1}{\varepsilon}(\Gamma(\frac{M-\mu}{\sqrt{\mu}},\partial^{\alpha}f),\frac{\partial^{\alpha}\overline{G}}{\sqrt{\mu}}+\partial^{\alpha}f)
&\leq C\frac{1}{\varepsilon}\int_{\mathbb{R}^{3}}|\frac{M-\mu}{\sqrt{\mu}}|_{2}|\partial^{\alpha}f|_{\sigma}(|\frac{\partial^{\alpha}\overline{G}}{\sqrt{\mu}}|_{\sigma}+|\partial^{\alpha}f|_{\sigma})\,dx
\\
&\leq C(\eta_{0}+\varepsilon)\frac{1}{\varepsilon}(\|\partial^{\alpha}f\|^{2}_{\sigma}+\|\partial^{\alpha}(\widetilde{\rho},\widetilde{u},\widetilde{\theta})\|^{2}+\varepsilon^{2}),
\end{align*}
according to \eqref{5.11}, \eqref{5.16}, Lemma \ref{lem5.4}, \eqref{1.7} and \eqref{3.1}.
Plugging the above estimates into \eqref{4.54} gives
\begin{align}
\label{4.55}
\frac{1}{\varepsilon}|(\Gamma(\frac{M-\mu}{\sqrt{\mu}},\partial^{\alpha}f),\frac{\partial^{\alpha}F}{\sqrt{\mu}})|
&\leq C(\eta_{0}+\varepsilon^{\frac{1}{2}})\frac{1}{\varepsilon}(\|\partial^{\alpha}f\|^{2}_{\sigma}+\|\partial^{\alpha}(\widetilde{\rho},\widetilde{u},\widetilde{\theta})\|^{2}+\frac{1}{\varepsilon^{2}}\|\partial^{\alpha'} f\|^{2}_{\sigma}+\varepsilon^{2})
\nonumber\\
&\leq C(\eta_{0}+\varepsilon^{\frac{1}{2}})\frac{1}{\varepsilon^{2}}\mathcal{D}_N(t)+C(\eta_{0}+\varepsilon^{\frac{1}{2}})\varepsilon.
\end{align}
Using \eqref{5.11} again, it is straightforward to see that
\begin{align*}
\frac{1}{\varepsilon}\sum_{1\leq\alpha_{1}\leq\alpha}&C^{\alpha_{1}}_{\alpha}(\Gamma(\partial^{\alpha_{1}}(\frac{M-\mu}{\sqrt{\mu}}),
\partial^{\alpha-\alpha_{1}}f),\frac{\partial^{\alpha}F}{\sqrt{\mu}})	
\\
\leq	
&C\sum_{1\leq\alpha_{1}\leq\alpha}\underbrace{\frac{1}{\varepsilon}\int_{\mathbb{R}^{3}}|\langle v\rangle^{-1}\partial^{\alpha_{1}}(\frac{M-\mu}{\sqrt{\mu}})|_{2}
|\partial^{\alpha-\alpha_{1}}f|_{\sigma}|\frac{\partial^{\alpha}F}{\sqrt{\mu}}|_{\sigma}\,dx}_{J_3}.
\end{align*}
Let us carefully deal with the term $J_3$.
If $1\leq|\alpha_{1}|\leq|\alpha|/2$, we use the Cauchy-Schwarz and Sobolev inequalities, \eqref{4.19B}, \eqref{4.20B}, as well as
\eqref{1.7} and \eqref{3.1},  to get
\begin{align*}	
J_3\leq& C\frac{1}{\varepsilon}\||\partial^{\alpha_{1}}(\frac{M-\mu}{\sqrt{\mu}})|_{2}\|_{L^{\infty}}\|\partial^{\alpha-\alpha_{1}}f\|_{\sigma}\|\frac{\partial^{\alpha}F}{\sqrt{\mu}}\|_{\sigma}
\\
\leq& C\frac{1}{\varepsilon}\||(|\partial^{\alpha_1}(\rho,u,\theta)|+\cdot\cdot\cdot+|\nabla_x(\rho,u,\theta)|^{|\alpha_{1}|})|\|_{L^{\infty}}\|\partial^{\alpha-\alpha_{1}}f\|_{\sigma}\|\frac{\partial^{\alpha}F}{\sqrt{\mu}}\|_{\sigma}
\\
\leq& C(\eta_{0}+\varepsilon^{\frac{1}{2}})\frac{1}{\varepsilon}\|\partial^{\alpha-\alpha_{1}}f\|_{\sigma}(\|\partial^{\alpha}f\|_{\sigma}+\|\partial^{\alpha}(\widetilde{\rho},\widetilde{u},\widetilde{\theta})\|+\eta_{0}+\varepsilon)
\\
\leq& C(\eta_{0}+\varepsilon^{\frac{1}{2}})
\frac{1}{\varepsilon}\big\{\|\partial^{\alpha}f\|^{2}_{\sigma}+\|\partial^{\alpha}(\widetilde{\rho},\widetilde{u},\widetilde{\theta})\|^{2}+\frac{1}{\varepsilon^{2}}\|\partial^{\alpha-\alpha_{1}}f\|^2_{\sigma}+\varepsilon^{2}(\eta_{0}+\varepsilon)\big\}
\\
\leq& 
C(\eta_{0}+\varepsilon^{\frac{1}{2}})\frac{1}{\varepsilon^{2}}\mathcal{D}_N(t)+C(\eta_{0}+\varepsilon^{\frac{1}{2}})\varepsilon,
\end{align*}
where we have used the smallness of $\varepsilon$ and
$\eta_{0}$ as well as the following estimate for $|\alpha|=N$
\begin{align}
\label{4.56}
\|\frac{\partial^{\alpha}F}{\sqrt{\mu}}\|_{\sigma}
&\leq \|\frac{\partial^{\alpha}\sqrt{\mu}f}{\sqrt{\mu}}\|_{\sigma}+\|\frac{\partial^{\alpha}\overline{G}}{\sqrt{\mu}}\|_{\sigma}+\|\frac{\partial^{\alpha}M}{\sqrt{\mu}}\|_{\sigma}
\nonumber\\
&\leq C(\|\partial^{\alpha}f\|_{\sigma}
+\|\partial^{\alpha}(\widetilde{\rho},\widetilde{u},\widetilde{\theta})\|+\eta_{0}+\varepsilon).
\end{align}
If $|\alpha|/2<|\alpha_{1}|\leq|\alpha|-1$, one takes the $L^{6}-L^{3}-L^{2}$ H\"{o}lder inequality to claim
\begin{align*}	
J_3\leq& C\frac{1}{\varepsilon}\||\partial^{\alpha_{1}}(\frac{M-\mu}{\sqrt{\mu}})|_2\|_{L^{3}}\||\partial^{\alpha-\alpha_{1}}f|_{\sigma}\|_{L^{6}}\|\frac{\partial^{\alpha}F}{\sqrt{\mu}}\|_{\sigma}
\\
\leq& C\frac{1}{\varepsilon}(\eta_{0}+\varepsilon^{\frac{1}{2}})\|\nabla_{x}\partial^{\alpha-\alpha_{1}}f\|_{\sigma}(\|\partial^{\alpha}f\|_{\sigma}
+\|\partial^{\alpha}(\widetilde{\rho},\widetilde{u},\widetilde{\theta})\|+\eta_{0}+\varepsilon)
\\
\leq& C(\eta_{0}+\varepsilon^{\frac{1}{2}})\frac{1}{\varepsilon^{2}}\mathcal{D}_N(t)+C(\eta_{0}+\varepsilon^{\frac{1}{2}})\varepsilon.
\end{align*}
If $|\alpha_{1}|=|\alpha|$, then  $|\alpha-\alpha_{1}|=0$ and it is easy to check that
\begin{align*}
J_3 \leq& C\frac{1}{\varepsilon}\|\partial^{\alpha_{1}}(\frac{M-\mu}{\sqrt{\mu}})\|\||f|_{\sigma}\|_{L^{\infty}}\|\frac{\partial^{\alpha}F}{\sqrt{\mu}}\|_{\sigma}
\\
\leq&
C\frac{1}{\varepsilon}(\varepsilon^{-\frac{3}{2}}\|\nabla_xf\|_{\sigma}\|\nabla_x^{2}f\|_{\sigma}+\varepsilon^{\frac{3}{2}}\|\frac{\partial^{\alpha}F}{\sqrt{\mu}}\|^2_{\sigma})
\\
\leq& C\varepsilon^{\frac{1}{2}}\frac{1}{\varepsilon^{2}}\mathcal{D}_N(t)+C(\eta_{0}+\varepsilon)\varepsilon^{\frac{1}{2}}.
\end{align*}
We thus conclude from the above estimates on $J_3$ that
\begin{align}
\label{4.57}
\frac{1}{\varepsilon}\sum_{1\leq\alpha_{1}\leq\alpha}C^{\alpha_{1}}_{\alpha}(\Gamma(\frac{\partial^{\alpha_{1}}(M-\mu)}{\sqrt{\mu}},
\partial^{\alpha-\alpha_{1}}f),\frac{\partial^{\alpha}F}{\sqrt{\mu}})
\leq C(\eta_{0}+\varepsilon^{\frac{1}{2}})\frac{1}{\varepsilon^{2}}\mathcal{D}_N(t)+C(\eta_{0}+\varepsilon^{\frac{1}{2}})\varepsilon^{\frac{1}{2}}.
\end{align}
With \eqref{4.55} and \eqref{4.57} in hand, we get
\begin{align}
\label{4.58}
\frac{1}{\varepsilon}|(\partial^{\alpha}\Gamma(\frac{M-\mu}{\sqrt{\mu}},f),\frac{\partial^{\alpha}F}{\sqrt{\mu}})|\leq C(\eta_{0}+\varepsilon^{\frac{1}{2}})\frac{1}{\varepsilon^{2}}\mathcal{D}_N(t)+C(\eta_{0}+\varepsilon^{\frac{1}{2}})\varepsilon^{\frac{1}{2}}.
\end{align}
The second term on the right hand side of \eqref{4.47} can be treated in the same way as \eqref{4.58}. It follows that 
\begin{align}
\label{4.59}	
\frac{1}{\varepsilon}|(\partial^{\alpha}\Gamma(f,\frac{M-\mu}{\sqrt{\mu}}),\frac{\partial^{\alpha}F}{\sqrt{\mu}})|\leq C(\eta_{0}+\varepsilon^{\frac{1}{2}})\frac{1}{\varepsilon^{2}}\mathcal{D}_N(t)+C(\eta_{0}+\varepsilon^{\frac{1}{2}})\varepsilon^{\frac{1}{2}}.
\end{align}
Next we shall concentrate on the third term on the right hand side of \eqref{4.47}. By $G=\overline{G}+\sqrt{\mu}f$, we see 
\begin{equation}
\label{4.60}
\frac{1}{\varepsilon}(\partial^{\alpha}\Gamma(\frac{G}{\sqrt{\mu}},\frac{G}{\sqrt{\mu}}),\frac{\partial^{\alpha}F}{\sqrt{\mu}})=\frac{1}{\varepsilon}(\partial^{\alpha}\Gamma(\frac{\overline{G}}{\sqrt{\mu}},\frac{\overline{G}}{\sqrt{\mu}})+\partial^{\alpha}\Gamma(\frac{\overline{G}}{\sqrt{\mu}},f)+\partial^{\alpha}\Gamma(f,\frac{\overline{G}}{\sqrt{\mu}})+\partial^{\alpha}\Gamma(f,f),\frac{\partial^{\alpha}F}{\sqrt{\mu}}).
\end{equation}
We compute \eqref{4.60} term by term. By \eqref{5.11}, Lemma \ref{lem5.4}, the Cauchy-Schwarz and Sobolev inequalities, \eqref{4.56}, \eqref{1.7} and \eqref{3.1}, one has
\begin{align*}
\frac{1}{\varepsilon}|(\partial^{\alpha}\Gamma(\frac{\overline{G}}{\sqrt{\mu}},
\frac{\overline{G}}{\sqrt{\mu}}),\frac{\partial^{\alpha}F}{\sqrt{\mu}})|\leq &C\frac{1}{\varepsilon}\sum_{\alpha_{1}\leq\alpha}\int_{\mathbb{R}^{3}}
|\langle v\rangle^{-1}\partial^{\alpha_{1}}(\frac{\overline{G}}{\sqrt{\mu}})|_{2}
|\partial^{\alpha-\alpha_{1}}(\frac{\overline{G}}{\sqrt{\mu}})|_{\sigma}|\frac{\partial^{\alpha}F}{\sqrt{\mu}}|_{\sigma}\,dx
\\
&\leq C(\eta_{0}+\varepsilon^{\frac{1}{2}})\frac{1}{\varepsilon^{2}}\mathcal{D}_N(t)+C(\eta_{0}+\varepsilon^{\frac{1}{2}})\varepsilon^{\frac{1}{2}}.
\end{align*}
Using \eqref{5.11} again, the second term on the right hand side of \eqref{4.60} is bounded by
\begin{align*}
\frac{1}{\varepsilon}|(\partial^{\alpha}\Gamma(\frac{\overline{G}}{\sqrt{\mu}},f),\frac{\partial^{\alpha}F}{\sqrt{\mu}})|
\leq C\sum_{\alpha_{1}\leq\alpha}\underbrace{\frac{1}{\varepsilon}\int_{\mathbb{R}^{3}}|\langle v\rangle^{-1}\partial^{\alpha_{1}}(\frac{\overline{G}}{\sqrt{\mu}})|_{2}
|\partial^{\alpha-\alpha_{1}}f|_{\sigma}
|\frac{\partial^{\alpha}F}{\sqrt{\mu}}|_{\sigma}\,dx}_{J_4}.
\end{align*}
If $|\alpha_{1}|\leq|\alpha|/2$, we can deduce from  Lemma \ref{lem5.4},
the Cauchy-Schwarz and Sobolev inequalities, \eqref{1.7}, \eqref{3.1} and \eqref{4.56} that
\begin{align*}
J_4&\leq C\frac{1}{\varepsilon}\||\partial^{\alpha_{1}}(\frac{\overline{G}}{\sqrt{\mu}})|_2\|_{L^{\infty}}\|\partial^{\alpha-\alpha_{1}}f\|_{\sigma}\|\frac{\partial^{\alpha}F}{\sqrt{\mu}}\|_{\sigma}
\\
&\leq C\eta_{0}\|\partial^{\alpha-\alpha_{1}}f\|_{\sigma}(\|\partial^{\alpha}f\|_{\sigma}+\|\partial^{\alpha}(\widetilde{\rho},\widetilde{u},\widetilde{\theta})\|+\eta_{0}+\varepsilon)
\\
&\leq C(\eta_{0}+\varepsilon^{\frac{1}{2}})\frac{1}{\varepsilon^{2}}\mathcal{D}_N(t)+C(\eta_{0}+\varepsilon^{\frac{1}{2}})\varepsilon^{\frac{1}{2}}.
\end{align*}
If $|\alpha|/2<|\alpha_{1}|\leq|\alpha|$ and then it holds that
\begin{align*}
J_4&\leq C\frac{1}{\varepsilon}\|\partial^{\alpha_{1}}(\frac{\overline{G}}{\sqrt{\mu}})\|\||\partial^{\alpha-\alpha_{1}}f|_{\sigma}\|_{L^{\infty}}
\|\frac{\partial^{\alpha}F}{\sqrt{\mu}}\|_{\sigma}
\\
&\leq C(\eta_{0}+\varepsilon^{\frac{1}{2}})\frac{1}{\varepsilon^{2}}\mathcal{D}_N(t)+C(\eta_{0}+\varepsilon^{\frac{1}{2}})\varepsilon^{\frac{1}{2}}.
\end{align*}
It follows from the above estimate on $J_4$ that
\begin{align*}
\frac{1}{\varepsilon}|(\partial^{\alpha}\Gamma(\frac{\overline{G}}{\sqrt{\mu}},f),\frac{\partial^{\alpha}F}{\sqrt{\mu}})|
\leq C(\eta_{0}+\varepsilon^{\frac{1}{2}})\frac{1}{\varepsilon^{2}}\mathcal{D}_N(t)+C(\eta_{0}+\varepsilon^{\frac{1}{2}})\varepsilon^{\frac{1}{2}}.
\end{align*}
The third term on the right hand side of \eqref{4.60} can be handled in the same manner and it shares the same bound.
We still compute the last term of \eqref{4.60}. It is clear to see by \eqref{5.11} that
\begin{align*}
\frac{1}{\varepsilon}|(\partial^{\alpha}\Gamma(f,f),\frac{\partial^{\alpha}F}{\sqrt{\mu}})|\leq C\sum_{\alpha_{1}\leq\alpha}\underbrace{\frac{1}{\varepsilon}\int_{\mathbb{R}^{3}}|\langle v\rangle^{-1}\partial^{\alpha_{1}}f|_{2}|\partial^{\alpha-\alpha_{1}}f|_{\sigma}|\frac{\partial^{\alpha}F}{\sqrt{\mu}}|_{\sigma}\,dx}_{J_5}.
\end{align*}
The term $J_5$ is treated in the same way as the term $J_4$. For $|\alpha_{1}|\leq|\alpha|/2$, we have
\begin{equation*}
J_5\leq C\frac{1}{\varepsilon}\|\partial^{\alpha_{1}}f\|_{L^{\infty}}\|\partial^{\alpha-\alpha_{1}}f\|_{\sigma}\|\frac{\partial^{\alpha}F}{\sqrt{\mu}}\|_{\sigma}
\leq C(\eta_{0}+\varepsilon^{\frac{1}{2}})\frac{1}{\varepsilon^{2}}\mathcal{D}_N(t)+C(\eta_{0}+\varepsilon^{\frac{1}{2}})\varepsilon^{\frac{1}{2}}.
\end{equation*}
In case of $|\alpha|/2<|\alpha_{1}|\leq|\alpha|$, it follows that
\begin{equation*}
J_5\leq C\frac{1}{\varepsilon}\|\partial^{\alpha_{1}}f\|\||\partial^{\alpha-\alpha_{1}}f|_{\sigma}\|_{L^{\infty}}\|\frac{\partial^{\alpha}F}{\sqrt{\mu}}\|_{\sigma}\leq C(\eta_{0}+\varepsilon^{\frac{1}{2}})\frac{1}{\varepsilon^{2}}\mathcal{D}_N(t)+C(\eta_{0}+\varepsilon^{\frac{1}{2}})\varepsilon^{\frac{1}{2}}.
\end{equation*}
With these estimates, we can obtain	
\begin{align*}
\frac{1}{\varepsilon}|(\partial^{\alpha}\Gamma(f,f),\frac{\partial^{\alpha}F}{\sqrt{\mu}})|
\leq C(\eta_{0}+\varepsilon^{\frac{1}{2}})\frac{1}{\varepsilon^{2}}\mathcal{D}_N(t)+C(\eta_{0}+\varepsilon^{\frac{1}{2}})\varepsilon^{\frac{1}{2}}.
\end{align*}
Consequently, plugging the above estimates into \eqref{4.60} leads us to
\begin{align}
\label{4.61}	
\frac{1}{\varepsilon}|\partial^{\alpha}(\Gamma(\frac{G}{\sqrt{\mu}},\frac{G}{\sqrt{\mu}}),\frac{\partial^{\alpha}F}{\sqrt{\mu}})|\leq C(\eta_{0}+\varepsilon^{\frac{1}{2}})\frac{1}{\varepsilon^{2}}\mathcal{D}_N(t)+C(\eta_{0}+\varepsilon^{\frac{1}{2}})\varepsilon^{\frac{1}{2}}.
\end{align}
The last term of \eqref{4.47} can be controlled by
\begin{align}
\label{4.62}
&C\|\langle v\rangle^{\frac{1}{2}}\frac{1}{\sqrt{\mu}}\partial^{\alpha}P_{1}\big\{v\cdot(\frac{|v-u|^{2}\nabla_x\overline{\theta}}{2R\theta^{2}}+\frac{(v-u)\cdot\nabla_x\bar{u}}{R\theta}) M\big\}\|^{2}+C\|\langle v\rangle^{-\frac{1}{2}}\frac{\partial^{\alpha}F}{\sqrt{\mu}}\|^{2}
\nonumber\\
&\leq C(\|\partial^{\alpha}f\|^{2}_{\sigma}+\|\partial^{\alpha}(\widetilde{\rho},\widetilde{u},\widetilde{\theta})\|^{2}+\eta_{0}+\varepsilon)
\leq C(\eta_{0}+\varepsilon^{\frac{1}{2}})\frac{1}{\varepsilon^{2}}\mathcal{D}_N(t)+C(\eta_{0}+\varepsilon^{\frac{1}{2}}),
\end{align}
where the Cauchy-Schwarz and Sobolev inequalities, \eqref{2.26}, \eqref{1.7}, \eqref{3.1} and \eqref{4.56} have been used.

As a consequence, substituting \eqref{4.53}, \eqref{4.58}, \eqref{4.59}, \eqref{4.61} and \eqref{4.62} 
into \eqref{4.47} and using the smallness of $\varepsilon$ and $\eta_{0}$, we get
\begin{align}
\label{4.63}
\frac{1}{2}\frac{d}{dt}\sum_{|\alpha|=N}\|\frac{\partial^{\alpha}F}{\sqrt{\mu}}\|^{2}+c\frac{1}{\varepsilon}\sum_{|\alpha|=N}\|\partial^{\alpha}f\|_{\sigma}^{2}\leq C
(\eta_{0}+\varepsilon^{\frac{1}{2}})\frac{1}{\varepsilon^{2}}\mathcal{D}_N(t)+C(\eta_{0}+\varepsilon^{\frac{1}{2}}).
\end{align}
To further estimate \eqref{4.63}, using \eqref{lem.NFlbdd.est} whose proof will be postponed to Lemma \ref{lem.NFlbdd} later on, one has
\begin{align}
\label{4.64}
\varepsilon^{2}\sum_{|\alpha|=N}\|\frac{\partial^{\alpha}F(t)}{\sqrt{\mu}}\|^{2}\geq c_3\varepsilon^{2}\sum_{|\alpha|=N}(\|\partial^{\alpha}(\widetilde{\rho},\widetilde{u},\widetilde{\theta})\|^{2}+\|\partial^{\alpha}f\|^{2})-C(\eta_{0}+\varepsilon^{\frac{1}{2}})\varepsilon^{2}.
\end{align}
Therefore, by integrating \eqref{4.63} with respect to $t$ and multiplying the resulting equation by $\varepsilon^{2}$, and then using \eqref{4.64},
 \eqref{3.5} and the following estimate
\begin{align*}
\varepsilon^{2}\sum_{|\alpha|=N}\|\frac{\partial^{\alpha}F(0)}{\sqrt{\mu}}\|^{2}
&\leq C\varepsilon^{2}\sum_{|\alpha|=N}(\|\frac{\sqrt{\mu}\partial^{\alpha}f(0)}{\sqrt{\mu}}\|^{2}+\|\frac{\partial^{\alpha}\overline{G}(0)}{\sqrt{\mu}}\|^{2}+\|\frac{\partial^{\alpha}M(0)}{\sqrt{\mu}}\|^{2})
\\
&\leq C(\eta_{0}+\varepsilon^{\frac{1}{2}})\varepsilon^{2},
\end{align*}
we thus arrive at the desired estimate \eqref{4.65}. This completes the proof of Lemma \ref{lem.Nsde}.
\end{proof}

For applying \eqref{4.64} in the proof of  Lemma \ref{lem.Nsde} above, we need the following result.

\begin{lemma}\label{lem.NFlbdd}
There is a constant $c_2>0$ such that
\begin{equation}\label{lem.NFlbdd.est}
\|\frac{\partial^{\alpha}F(t)}{\sqrt{\mu}}\|^{2}
\geq c_2(\|\partial^{\alpha}(\widetilde{\rho},\widetilde{u},\widetilde{\theta})\|^{2}+\|\partial^{\alpha}f\|^{2})-C(\eta_{0}+\varepsilon^{\frac{1}{2}}),
\end{equation}
for any $\alpha$ with $|\alpha|=N$.
\end{lemma}

\begin{proof}
Let $|\alpha|=N$. In view of the macro-micro decomposition  $F=M+G$, it holds that
\begin{align}
\label{4.70aa}
\|\frac{\partial^{\alpha}F(t)}{\sqrt{\mu}}\|^{2}
=\int_{\mathbb{R}^{3}}\int_{\mathbb{R}^{3}}\frac{(\partial^{\alpha}M)^{2}+(\partial^{\alpha}G)^{2}+2\partial^{\alpha}G\partial^{\alpha}M}{\mu}\,dv\,dx.
\end{align}
First of all, we write
\begin{align}
\label{4.70a}
\int_{\mathbb{R}^{3}}\int_{\mathbb{R}^{3}}\frac{(\partial^{\alpha}M)^{2}}{\mu}\,dv\,dx
=\int_{\mathbb{R}^{3}}\int_{\mathbb{R}^{3}}\frac{(\partial^{\alpha}M)^{2}}{M}+(\frac{1}{\mu}-\frac{1}{M})(\partial^{\alpha}M)^{2}\,dv\,dx.
\end{align}
Since $\partial^{\alpha}M=I_4+I_5$ holds in terms of \eqref{4.48}, one has
\begin{align*}
\int_{\mathbb{R}^{3}}\int_{\mathbb{R}^{3}}\frac{(\partial^{\alpha}M)^{2}}{M}\,dv\,dx
=\int_{\mathbb{R}^{3}}\int_{\mathbb{R}^{3}}\frac{(I_{4})^{2}+(I_{5})^{2}+2I_{4}I_{5}}{M}\,dv\,dx.
\end{align*}
Recalling the definition of $I_4$ in \eqref{4.48},
we can deduce from \eqref{2.4} and \eqref{1.7} that
\begin{align*}
\int_{\mathbb{R}^{3}}\int_{\mathbb{R}^{3}}\frac{(I_{4})^{2}}{M}\,dv\,dx&=\int_{\mathbb{R}^{3}}\int_{\mathbb{R}^{3}}M\Big\{\frac{\partial^{\alpha}\rho}{\rho}+\frac{(v-u)\cdot\partial^{\alpha}u}{R\theta}+(\frac{|v-u|^{2}}{2R\theta}-\frac{3}{2})\frac{\partial^{\alpha}\theta}{\theta} \Big\}^{2}\,dv\,dx
\\
&=\int_{\mathbb{R}^{3}}\int_{\mathbb{R}^{3}}M
\Big\{(\frac{\partial^{\alpha}\rho}{\rho})^{2}+(\frac{(v-u)\cdot\partial^{\alpha}u}{R\theta})^{2}+((\frac{|v-u|^{2}}{2R\theta}-\frac{3}{2})\frac{\partial^{\alpha}\theta}{\theta})^{2} \Big\}\,dv\,dx,
\end{align*}
which is further bounded from below as
$$
c\|\partial^{\alpha}(\rho,u,\theta)\|^{2}
\geq c\|\partial^{\alpha}(\widetilde{\rho},\widetilde{u},\widetilde{\theta})\|^{2}-C\|\partial^{\alpha}(\bar{\rho},\bar{u},\bar{\theta})\|^{2}
\geq c\|\partial^{\alpha}(\widetilde{\rho},\widetilde{u},\widetilde{\theta})\|^{2}-C\eta_{0}.
$$
By the definition of $I_5$ in \eqref{4.48}, we see that $\int_{\mathbb{R}^{3}}\int_{\mathbb{R}^{3}}\frac{(I_{5})^{2}}{M}\,dvdx$ is given by
\begin{align*}
\int_{\mathbb{R}^{3}}\int_{\mathbb{R}^{3}}\frac{1}{M}
\Big\{\sum_{1\leq\alpha_{1}\leq \alpha'}C^{\alpha_1}_{\alpha'}&\big[\partial^{\alpha_{1}}(M\frac{1}{\rho})\partial^{\alpha'-\alpha_{1}}\partial_{x_i}\rho
+\partial^{\alpha_{1}}(M\frac{v-u}{R\theta})\cdot\partial^{\alpha'-\alpha_{1}}\partial_{x_i}u\\
&\quad+\partial^{\alpha_{1}}(M\frac{|v-u|^{2}}{2R\theta^{2}}-M\frac{3}{2\theta})\partial^{\alpha'-\alpha_{1}}\partial_{x_i}\theta\big]\Big\}^2
\,dv\,dx,	
\end{align*}
which can be further bounded by $C(\eta_{0}+\varepsilon^{\frac{1}{2}})$. Similarly, using the definition of $I_4$ and $I_5$ in \eqref{4.48} again, we claim that
\begin{align*}
\int_{\mathbb{R}^{3}}\int_{\mathbb{R}^{3}}\frac{2I_{4}I_{5}}{M}\,dv\,dx\leq C(\eta_{0}+\varepsilon^{\frac{1}{2}}).
\end{align*}
It follows from the above estimates that
\begin{align*}
\int_{\mathbb{R}^{3}}\int_{\mathbb{R}^{3}}\frac{(\partial^{\alpha}M)^{2}}{M}\,dv\,dx
\geq c\|\partial^{\alpha}(\widetilde{\rho},\widetilde{u},\widetilde{\theta})\|^{2}-C(\eta_{0}+\varepsilon^{\frac{1}{2}}).
\end{align*}
On the other hand, it holds that
\begin{align*}
\int_{\mathbb{R}^{3}}\int_{\mathbb{R}^{3}}(\frac{1}{\mu}-\frac{1}{M})(\partial^{\alpha}M)^{2}\,dv\,dx\leq
C(\eta_{0}+\varepsilon^{\frac{1}{2}}).
\end{align*}
Plugging the above two estimates into \eqref{4.70a}, we obtain
\begin{align*}
\int_{\mathbb{R}^{3}}\int_{\mathbb{R}^{3}}\frac{(\partial^{\alpha}M)^{2}}{\mu}\,dv\,dx
\geq c\|\partial^{\alpha}(\widetilde{\rho},\widetilde{u},\widetilde{\theta})\|^{2}-C(\eta_{0}+\varepsilon^{\frac{1}{2}}).
\end{align*}
Thanks to $G=\overline{G}+\sqrt{\mu}f$, one has
\begin{align*}
\int_{\mathbb{R}^{3}}\int_{\mathbb{R}^{3}}\frac{(\partial^{\alpha}G)^{2}}{\mu}\,dv\,dx
=&\int_{\mathbb{R}^{3}}\int_{\mathbb{R}^{3}}\frac{(\sqrt{\mu}\partial^{\alpha}f+\partial^{\alpha}\overline{G})^{2}}{\mu}\,dv\,dx
\\
=&\int_{\mathbb{R}^{3}}\int_{\mathbb{R}^{3}}(\partial^{\alpha}f)^{2}\,dv\,dx
+\int_{\mathbb{R}^{3}}\int_{\mathbb{R}^{3}}\frac{(\partial^{\alpha}\overline{G})^{2}
+2\sqrt{\mu}\partial^{\alpha}f\partial^{\alpha}\overline{G}}{\mu}\,dv\,dx
\\
\geq& \frac{1}{2}\|\partial^{\alpha}f\|^{2}-C(\eta_{0}+\varepsilon),
\end{align*}
where in the last inequality we have used Lemma \ref{lem5.4}, \eqref{1.7}, \eqref{3.1} and the smallness of $\varepsilon$ and $\eta_{0}$.
For the last term of \eqref{4.70aa}, we see
\begin{align*}
\int_{\mathbb{R}^{3}}\int_{\mathbb{R}^{3}}\frac{2\partial^{\alpha}G\partial^{\alpha}M}{\mu}\,dv\,dx
&=	\int_{\mathbb{R}^{3}}\int_{\mathbb{R}^{3}}\frac{2\partial^{\alpha}G\partial^{\alpha}M}{M}
+(\frac{1}{\mu}-\frac{1}{M})2\partial^{\alpha}G\partial^{\alpha}M\,dv\,dx.
\end{align*}
First note that
\begin{align*}
\int_{\mathbb{R}^{3}}\int_{\mathbb{R}^{3}}\frac{2\partial^{\alpha}GI_{4}}{M}\,dv\,dx
=\int_{\mathbb{R}^{3}}\int_{\mathbb{R}^{3}}2\partial^{\alpha}G\Big(\frac{\partial^{\alpha}\rho}{\rho}+
\frac{(v-u)\cdot\partial^{\alpha}u}{R\theta}+(\frac{|v-u|^{2}}{2R\theta}-\frac{3}{2})\frac{\partial^{\alpha}\theta}{\theta} \Big)\,dv\,dx=0,
\end{align*}
due to \eqref{4.48} and \eqref{2.6}.
Then it holds by this and the definition of $I_5$ in \eqref{4.48} that
\begin{align*}
\int_{\mathbb{R}^{3}}\int_{\mathbb{R}^{3}}\frac{2\partial^{\alpha}G\partial^{\alpha}M}{M}\,dv\,dx
&=\int_{\mathbb{R}^{3}}\int_{\mathbb{R}^{3}}\frac{2\partial^{\alpha}G(I_{4}+I_{5})}{M}\,dv\,dx
=\int_{\mathbb{R}^{3}}\int_{\mathbb{R}^{3}}\frac{2\partial^{\alpha}GI_{5}}{M}\,dv\,dx,
\end{align*}
which can be bounded by $C(\eta_{0}+\varepsilon^{\frac{1}{2}})$.
On the other hand, it is easy to check that
\begin{align*}
\int_{\mathbb{R}^{3}}\int_{\mathbb{R}^{3}}(\frac{1}{\mu}-\frac{1}{M})2\partial^{\alpha}G\partial^{\alpha}M\,dv\,dx
\leq C(\eta_{0}+\varepsilon^{\frac{1}{2}}).
\end{align*}
It follows that
\begin{align*}
\int_{\mathbb{R}^{3}}\int_{\mathbb{R}^{3}}\frac{2\partial^{\alpha}G\partial^{\alpha}M}{\mu}\,dvdx
\leq C(\eta_{0}+\varepsilon^{\frac{1}{2}}).	
\end{align*}
Consequently, substituting all the above estimates into \eqref{4.70aa}, the desired estimate \eqref{lem.NFlbdd.est} follows. This then completes the proof of Lemma \ref{lem.NFlbdd}.
\end{proof}

Combining those estimates in Lemma \ref{lem.Nsde} and Lemma \ref{lem.fpnfpN1}, we are able to obtain all the  space derivative estimates for both the fluid and non-fluid parts.

\begin{lemma} \label{lem4.3}
It holds that
\begin{align}
\label{4.26}
&\sum_{1\leq|\alpha|\leq N-1}\{\|\partial^{\alpha}(\widetilde{\rho},\widetilde{u},\widetilde{\theta})(t)\|^{2}
+\|\partial^{\alpha}f(t)\|^{2}+c\frac{1}{\varepsilon}\int^{t}_{0}\|\partial^{\alpha}f(s)\|_{\sigma}^{2}\,ds\}\notag\\
&+c\varepsilon\sum_{2\leq|\alpha|\leq N}\int^{t}_{0}\|\partial^{\alpha}(\widetilde{\rho},\widetilde{u},
\widetilde{\theta})(s)\|^{2}\,ds
\nonumber\\
&+C\varepsilon^{2}\sum_{|\alpha|=N}\{\|\partial^{\alpha}(\widetilde{\rho},\widetilde{u},\widetilde{\theta})(t)\|^{2}
+\|\partial^{\alpha}f(t)\|^{2}+c\frac{1}{\varepsilon}\int^{t}_{0}\|\partial^{\alpha}f(s)\|_{\sigma}^{2}\,ds\}
\nonumber\\
&\leq C(\eta_{0}+\varepsilon^{\frac{1}{2}})\int^{t}_{0}\mathcal{D}_N(s)\,ds+C(1+t)(\eta_{0}+\varepsilon^{\frac{1}{2}})\varepsilon^{2}.
\end{align}
\end{lemma}

Notice that Lemma \ref{lem4.3} does not include the zero-order energy estimate. Therefore, adding \eqref{4.26} together with \eqref{4.2}, we conclude the energy estimate on solutions without velocity derivatives.
 
\begin{lemma}\label{lem.noweight}
It holds that
\begin{align}
\label{4.1}
&\sum_{|\alpha|\leq N-1}(\|\partial^{\alpha}(\widetilde{\rho},\widetilde{u},\widetilde{\theta})(t)\|^{2}
+\|\partial^{\alpha}f(t)\|^{2})
+\varepsilon^{2}\sum_{|\alpha|=N}(\|\partial^{\alpha}(\widetilde{\rho},\widetilde{u},\widetilde{\theta})(t)\|^{2}
+\|\partial^{\alpha}f(t)\|^{2})
\nonumber\\
&\quad+c\varepsilon\sum_{1\leq|\alpha|\leq N}\int^{t}_{0}\|\partial^{\alpha}(\widetilde{\rho},\widetilde{u},
\widetilde{\theta})(s)\|^{2}\,ds\notag\\
&\quad+c\frac{1}{\varepsilon}\sum_{|\alpha|\leq N-1}\int^{t}_{0}\|\partial^{\alpha}f(s)\|_{\sigma}^{2}\,ds
+c\varepsilon\sum_{|\alpha|=N}\int^{t}_{0}\|\partial^{\alpha}f(s)\|_{\sigma}^{2}\,ds
\nonumber\\
&\leq
C(\eta_{0}+\varepsilon^{\frac{1}{2}})\int^{t}_{0}\mathcal{D}_N(s)\,ds+
C(1+t)(\eta_{0}+\varepsilon^{\frac{1}{2}})\varepsilon^{2}.
\end{align}
\end{lemma}

\subsection{Mixed derivative estimate}
This subsection is devoted to deriving the mixed derivative estimate on the microscopic component $f$. We follows the iteration technique for velocity derivatives as in \cite{Guo-2002}.

\begin{lemma}\label{lem.mde}
It holds that
\begin{align}
\label{4.66}
&\sum_{|\alpha|+|\beta|\leq N,|\beta|\geq1}\big\{\|\partial_{\beta}^{\alpha}f(t)\|_{2,|\beta|}^{2}
+c\frac{1}{\varepsilon}\int^{t}_{0}\|\partial_{\beta}^{\alpha}f(s)\|_{\sigma,|\beta|}^{2}\,ds\big\}
\nonumber\\
\leq& C\frac{1}{\varepsilon}\sum_{|\alpha|\leq N-1}\int^{t}_{0}\|\partial^{\alpha}f(s)\|_{\sigma}^{2}\,ds
+C\varepsilon\sum_{1\leq|\alpha|\leq N}\int^{t}_{0}
\big\{\|\partial^{\alpha}f(s)\|_{\sigma}^{2}
+\|\partial^{\alpha}(\widetilde{\rho},\widetilde{u},\widetilde{\theta})(s)\|^{2}\big\}\,ds
\nonumber\\
&+C(\eta_{0}+\varepsilon^{\frac{1}{2}})\int^{t}_{0}\mathcal{D}_N(s)\,ds
+C(1+t)(\eta_{0}+\varepsilon^{\frac{1}{2}})\varepsilon^{2}.
\end{align}
\end{lemma}

\begin{proof}
Let $|\alpha|+|\beta|\leq N$ with $|\beta|\geq1$
and $w$ be defined in \eqref{2.24}, then we apply $\partial_{\beta}^{\alpha}$ to \eqref{2.21} and take the inner product of the resulting equation with $w^{2|\beta|}\partial^{\alpha}_{\beta}f$ over $\mathbb{R}^{3}\times\mathbb{R}^{3}$ to get
\begin{align}
\label{4.67}
&\frac{1}{2}\frac{d}{dt}\|\partial_{\beta}^{\alpha}f\|^{2}_{2,|\beta|}+(v\cdot\nabla_{x}\partial^{\alpha}_{\beta}f,w^{2|\beta|}\partial^{\alpha}_{\beta}f)
+(C^{\beta-e_{i}}_{\beta}\delta^{e_{i}}_{\beta}\partial^{\alpha+e_{i}}_{\beta-e_{i}}f,w^{2|\beta|}\partial^{\alpha}_{\beta}f)
-\frac{1}{\varepsilon}(\partial^{\alpha}_{\beta}\mathcal{L}f,w^{2|\beta|}\partial^{\alpha}_{\beta}f)
\nonumber\\
=&\frac{1}{\varepsilon}(\partial^{\alpha}_{\beta}\Gamma(\frac{M-\mu}{\sqrt{\mu}},f)
+\partial^{\alpha}_{\beta}\Gamma(f,\frac{M-\mu}{\sqrt{\mu}}),w^{2|\beta|}\partial^{\alpha}_{\beta}f)
+\frac{1}{\varepsilon}(\partial^{\alpha}_{\beta}\Gamma(\frac{G}{\sqrt{\mu}},\frac{G}{\sqrt{\mu}}),
w^{2|\beta|}\partial^{\alpha}_{\beta}f)
\nonumber\\
&+(\partial^{\alpha}_{\beta}[\frac{P_{0}(v\sqrt{\mu}\cdot\nabla_{x}f)}{\sqrt{\mu}}],w^{2|\beta|}\partial^{\alpha}_{\beta}f)
-(\partial^{\alpha}_{\beta}[\frac{P_{1}(v\cdot\nabla_{x}\overline{G})}{\sqrt{\mu}}],w^{2|\beta|}\partial^{\alpha}_{\beta}f)
-(\partial^{\alpha}_{\beta}(\frac{\partial_{t}\overline{G}}{\sqrt{\mu}}),w^{2|\beta|}\partial^{\alpha}_{\beta}f)
\nonumber\\
&-(\partial^{\alpha}_{\beta}\{\frac{1}{\sqrt{\mu}}P_{1}[v\cdot(\frac{|v-u|^{2}\nabla_{x}\widetilde{\theta}}{2R\theta^{2}}+\frac{(v-u)\cdot\nabla_{x}\widetilde{u}}{R\theta})M]\},w^{2|\beta|}\partial^{\alpha}_{\beta}f).
\end{align}
Here $\delta^{e_{i}}_{\beta}=1$ if $e_{i}\leq\beta$ or $\delta^{e_{i}}_{\beta}=0$ otherwise.

We now compute \eqref{4.67} term by term. The second term on the left hand side of \eqref{4.67}
vanishes by integration by parts. Thanks to $|\beta-e_{i}|=|\beta|-1$ and
$\|w^{\frac{1}{2}}w^{|\beta|}\partial^{\alpha}_{\beta}f\|\leq C\|\partial^{\alpha}_{\beta}f\|_{\sigma,|\beta|}$
for $w=\langle v\rangle^{-1}$, the third term on the left hand side of \eqref{4.67} can be estimated as
\begin{align}
\label{4.78b}
|(C^{\beta-e_{i}}_{\beta}\delta^{e_{i}}_{\beta}\partial^{\alpha+e_{i}}_{\beta-e_{i}}f,w^{2|\beta|}\partial_{\beta}^{\alpha}f)|
&\leq C\|w^{\frac{1}{2}+(|\beta|-1)}\partial^{\alpha+e_{i}}_{\beta-e_{i}}f\|\|w^{|\beta|+\frac{1}{2}}\partial^{\alpha}_{\beta}f\|
\nonumber\\
&= C\|w^{\frac{1}{2}}w^{|\beta-e_{i}|}\partial^{\alpha+e_{i}}_{\beta-e_{i}}f\|\|w^{\frac{1}{2}}w^{|\beta|}\partial^{\alpha}_{\beta}f\|
\nonumber\\
&\leq \eta\frac{1}{\varepsilon}\|\partial^{\alpha}_{\beta}f\|_{\sigma,|\beta|}^{2}+C_{\eta}\varepsilon\|\partial^{\alpha+e_{i}}_{\beta-e_{i}}f\|_{\sigma,|\beta-e_{i}|}^{2}.
\end{align}
In view of \eqref{5.10}, it is easy to see that
\begin{align*}
-\frac{1}{\varepsilon}(\partial_{\beta}^{\alpha}\mathcal{L}f,w^{2|\beta|}\partial_{\beta}^{\alpha}f)
\geq c\frac{1}{\varepsilon}\|\partial^{\alpha}_{\beta}f\|^{2}_{\sigma,|\beta|}-\eta\frac{1}{\varepsilon}\sum_{|\beta_{1}|\leq|\beta|}\|\partial^{\alpha}_{\beta_{1}}f\|_{\sigma,|\beta_{1}|}^{2}-C_{\eta}\frac{1}{\varepsilon}\|\partial^{\alpha}f\|_{\sigma}^{2}.
\end{align*}
For the first and two terms on the right hand side of \eqref{4.67}, we use
 \eqref{5.13} and \eqref{5.20} respectively, to get
\begin{align*}
\frac{1}{\varepsilon}|(\partial^{\alpha}_{\beta}\Gamma(\frac{M-\mu}{\sqrt{\mu}},f)
+\partial^{\alpha}_{\beta}\Gamma(f,\frac{M-\mu}{\sqrt{\mu}}),w^{2|\beta|}\partial^{\alpha}_{\beta}f)|
\leq C\eta\frac{1}{\varepsilon}\|\partial^{\alpha}_{\beta}f\|^{2}_{\sigma,|\beta|}+C_{\eta}(\eta_{0}+\varepsilon^{\frac{1}{2}})\mathcal{D}_N(t),
\end{align*}
and
\begin{align*}
&\frac{1}{\varepsilon}|(\partial^{\alpha}_{\beta}\Gamma(\frac{G}{\sqrt{\mu}},\frac{G}{\sqrt{\mu}}),
w^{2|\beta|}\partial^{\alpha}_{\beta}f)|
\nonumber\\
&\leq  C\eta\frac{1}{\varepsilon}\|\partial^{\alpha}_{\beta}f\|^{2}_{\sigma,|\beta|}
+C_{\eta}(\eta_{0}+\varepsilon^{\frac{1}{2}})\mathcal{D}_N(t)+C_{\eta}(\eta_{0}+\varepsilon^{\frac{1}{2}})\varepsilon^2.
\end{align*}
The term involving $\overline{G}$ in \eqref{4.67} can be estimated, by
Lemma \ref{lem5.4}, \eqref{2.26}, Lemma \ref{lem5.10}, the Cauchy-Schwarz and Sobolev inequalities, \eqref{1.7} and \eqref{3.1}, as 
\begin{align*}
&|(\partial^{\alpha}_{\beta}[\frac{P_{1}(v\cdot\nabla_{x}\overline{G})}{\sqrt{\mu}}],w^{2|\beta|}\partial^{\alpha}_{\beta}f)|
+|(\partial^{\alpha}_{\beta}(\frac{\partial_{t}\overline{G}}{\sqrt{\mu}}),w^{2|\beta|}\partial^{\alpha}_{\beta}f)|
\\
\leq&C(\|\langle v\rangle^{\frac{1}{2}}w^{|\beta|}\partial^{\alpha}_{\beta}[\frac{P_{1}(v\cdot\nabla_{x}\overline{G})}{\sqrt{\mu}}]\|+\|\langle v\rangle^{\frac{1}{2}}w^{|\beta|}\partial^{\alpha}_{\beta}(\frac{\partial_{t}\overline{G}}{\sqrt{\mu}})\|)\|\langle v\rangle^{-\frac{1}{2}}w^{|\beta|}\partial^{\alpha}_{\beta}f\|
\\
\leq& 
C\eta\frac{1}{\varepsilon}\|\partial^{\alpha}_{\beta}f\|^{2}_{\sigma,|\beta|}+C_{\eta}(\eta_{0}+\varepsilon^{\frac{1}{2}})\varepsilon^2.
\end{align*}
For the third term on the right hand side of \eqref{4.67}, we can deduce from 
\eqref{2.5}, \eqref{2.26}, the Cauchy-Schwarz and Sobolev inequalities, \eqref{1.7} and \eqref{3.1} that
\begin{align*}
&|(\partial^{\alpha}_{\beta}[\frac{1}{\sqrt{\mu}}P_{0}(v\sqrt{\mu}\cdot\nabla_{x}f)]
,w^{2|\beta|}\partial^{\alpha}_{\beta}f)|
\\
&=\big|\sum_{j=0}^{4}(\langle v\rangle^{\frac{1}{2}}w^{|\beta|}\partial_{\beta}^{\alpha}
[\langle v\sqrt{\mu}\cdot\nabla_{x}f,\frac{\chi_{j}}{M}\rangle\frac{\chi_{j}}{\sqrt{\mu}}],\langle v\rangle^{-\frac{1}{2}}w^{|\beta|}\partial^{\alpha}_{\beta}f)\big|
\\
&\leq 
C\eta\frac{1}{\varepsilon}\|\partial^{\alpha}_{\beta}f\|^{2}_{\sigma,|\beta|}+C_{\eta}\varepsilon
\|\nabla_x\partial^{\alpha}f\|^{2}_{\sigma}
+C_{\eta}(\eta_{0}+\varepsilon^{\frac{1}{2}})\varepsilon^2,
\end{align*}
where we have used the fact that $|\langle v\rangle^{l}\mu^{-\frac{1}{2}}\partial_{\beta}M|_{2}\leq C$ for any $l\geq0$ and $|\beta|\geq0$
by \eqref{3.6}. Likewise, the last term of \eqref{4.67} is dominated by
\begin{align*}
C\eta\frac{1}{\varepsilon}\|\partial^{\alpha}_{\beta}f\|^{2}_{\sigma,|\beta|}+C_{\eta}\varepsilon
\|(\nabla_{x}\partial^{\alpha}\widetilde{u},\nabla_{x}\partial^{\alpha}\widetilde{\theta})\|^{2}+C_{\eta}(\eta_{0}+\varepsilon^{\frac{1}{2}})\varepsilon^2.
\end{align*}
Hence, for $|\alpha|+|\beta|\leq N$ with $|\beta|\geq1$ and any small $\eta>0$, we can deduce from the above estimates that
\begin{align}
\label{4.68}
\frac{1}{2}\frac{d}{dt}\|\partial_{\beta}^{\alpha}f\|^{2}_{2,|\beta|}+c\frac{1}{\varepsilon}\|\partial^{\alpha}_{\beta}f\|^{2}_{\sigma,|\beta|}
\leq& 
C\eta\frac{1}{\varepsilon}\sum_{|\beta_{1}|\leq|\beta|}\|\partial^{\alpha}_{\beta_{1}}f\|_{\sigma,|\beta_{1}|}^{2}+C_{\eta}\varepsilon\|\partial^{\alpha+e_{i}}_{\beta-e_{i}}f\|_{\sigma,|\beta-e_{i}|}^{2}
\nonumber\\
&
+C_{\eta}\frac{1}{\varepsilon}\|\partial^{\alpha}f\|_{\sigma}^{2}+C_\eta\varepsilon
(\|(\nabla_{x}\partial^{\alpha}\widetilde{u},\nabla_{x}\partial^{\alpha}\widetilde{\theta})\|^{2}+\|\nabla_{x}\partial^{\alpha}f\|_{\sigma}^{2})
\nonumber\\
&+C_{\eta}(\eta_{0}+\varepsilon^{\frac{1}{2}})\mathcal{D}_N(t)+C_{\eta}(\eta_{0}+\varepsilon^{\frac{1}{2}})\varepsilon^2.
\end{align}

We shall use the induction on $|\beta|$  to cancel the first and second terms on the right hand side of \eqref{4.68}. By the suitable linear combination of \eqref{4.68}
for all the cases that  $|\alpha|+|\beta|\leq N$ with $|\beta|\geq1$ and then taking $\eta>0$ and $\varepsilon>0$ small enough, we see that there exist constants $C_{\alpha,\beta}>0$ such that
\begin{align}
\label{4.69}
&\sum_{|\alpha|+|\beta|\leq N,|\beta|\geq1}\{\frac{d}{dt}C_{\alpha,\beta}\|\partial_{\beta}^{\alpha}f\|_{2,|\beta|}^{2}
+c\frac{1}{\varepsilon}\|\partial_{\beta}^{\alpha}f\|_{\sigma,|\beta|}^{2}\}
\nonumber\\
\leq&C\varepsilon\sum_{1\leq|\alpha|\leq N}(\|\partial^{\alpha}f\|_{\sigma}^{2}+\|\partial^{\alpha}(\widetilde{\rho},\widetilde{u},\widetilde{\theta})\|^{2})+C\frac{1}{\varepsilon}\sum_{|\alpha|\leq N-1}\|\partial^{\alpha}f\|_{\sigma}^{2}
\nonumber\\
&+C(\eta_{0}+\varepsilon^{\frac{1}{2}})\mathcal{D}_N(t)+C(\eta_{0}+\varepsilon^{\frac{1}{2}})\varepsilon^2.
\end{align}
Integrating \eqref{4.69} with respect to $t$ and using \eqref{3.5}, we can obtain the desired estimate \eqref{4.66}. This then completes the proof of Lemma \ref{lem.mde}.
\end{proof}

\section{Proof of the main results}\label{sec.3}
In this section, we will prove our main results Theorem \ref{thm1.1} and Theorem \ref{thm1.2}
by the energy estimates derived in the previous section.
First of all, we prove Theorem \ref{thm1.1} on the
compressible Euler limit for the Landau equation.

\subsection{Proof of Theorem \ref{thm1.1}}
By multiplying \eqref{4.1} with a large positive constant $C$ and then adding the resultant inequality to \eqref{4.66}, 
we get by letting $\eta_{0}>0$ and $\varepsilon>0$ be small enough that
\begin{align}
\label{3.7}
\mathcal{E}_N(t)+\frac{1}{2}\int^{t}_{0}\mathcal{D}_N(s)\,ds\leq \frac{1}{2}\varepsilon^{2},
\end{align}
for  $t\in(0,T]$ with $T\in(0,\tau]$.
Here $\mathcal{E}_N(t)$ and  $\mathcal{D}_N(t)$ are defined in \eqref{3.2} and \eqref{3.8}, respectively.

Note that the a priori assumption \eqref{3.1} can be closed since the estimate  \eqref{3.7} is strictly stronger than \eqref{3.1}. Therefore, by the uniform a priori estimates  and the local existence of the solution,
the standard continuity argument, we can immediately derive the existence and uniqueness of smooth solutions
to the Landau equation \eqref{1.1} with initial data \eqref{1.10} as stated in Theorem \ref{thm1.1}.
Moreover, the desired estimate \eqref{3.21A}  holds true in terms of \eqref{3.7}.

To finish the proof of Theorem \ref{thm1.1}, we still need to prove the uniform convergence rate in
$\varepsilon$ as in \eqref{1.11}. Note from \eqref{3.6} that $(\rho,u,\theta)$ and $(\bar{\rho},\bar{u},\bar{\theta})$
are close enough to the state $(1,0,3/2)$, we can deduce from  this, \eqref{3.7} and
\eqref{3.2} that
\begin{align*}
&\|\frac{(M_{[\rho,u,\theta]}-M_{[\bar{\rho},\bar{u},\bar{\theta}]})(t)}{\sqrt{\mu}}\|_{L_{x}^{2}L_{v}^{2}}
+\|\frac{(M_{[\rho,u,\theta]}-M_{[\bar{\rho},\bar{u},\bar{\theta}]})(t)}{\sqrt{\mu}}\|_{L_{x}^{\infty}L_{v}^{2}}
\nonumber\\
&\leq C_\tau(\|(\widetilde{\rho},\widetilde{u},\widetilde{\theta})(t)\|+\|(\widetilde{\rho},\widetilde{u},\widetilde{\theta})(t)\|_{L_{x}^{\infty}})\leq C_\tau\varepsilon,
\end{align*}
for any $t\in[0,\tau]$.
Similarly, it holds that
\begin{equation*}
\sup_{t\in[0,\tau]}(\|f(t)\|_{L^{2}_{x}L^{2}_{v}}+\|f(t)\|_{L^{\infty}_{x}L^{2}_{v}})\leq C_\tau\varepsilon.
\end{equation*}
With Lemma \ref{lem5.4}, it is easy to check that
\begin{align*}
\sup_{t\in[0,\tau]}(
\|\frac{\overline{G}(t)}{\sqrt{\mu}}\|_{L_{x}^{2}L_{v}^{2}}
+\|\frac{\overline{G}(t)}{\sqrt{\mu}}\|_{L_{x}^{\infty}L_{v}^{2}})
\leq C_\tau\eta_{0}\varepsilon.
\end{align*}
Therefore, by these facts and
$F=M+\overline{G}+\sqrt{\mu}f$, we immediately get
\begin{align}\notag
	&\|\frac{(F-M_{[\bar{\rho},\bar{u},\bar{\theta}]})(t)}{\sqrt{\mu}}\|_{L_{x}^{2}L_{v}^{2}}
	+\|\frac{(F-M_{[\bar{\rho},\bar{u},\bar{\theta}]})(t)}{\sqrt{\mu}}\|_{L_{x}^{\infty}L_{v}^{2}}
	\leq C_\tau\varepsilon,
\end{align}
for any $t\in[0,\tau]$. This, combined with the fact $F=F^{\varepsilon}(t,x,v)$, gives \eqref{1.11} and hence ends the proof of Theorem \ref{thm1.1}. \qed

\subsection{Proof of Theorem \ref{thm1.2}}
We are now in a position to prove Theorem \ref{thm1.2}. We first give the following lemma on the existence result of the compressible Euler system \eqref{1.4} and the initial data associated with \eqref{1.13} and $\delta$ in \eqref{1.16}. It can be found in \cite[Lemma 3.1]{Guo-Jang-Jiang-2010}. Readers also refer to \cite{Kato,Majda} and the references cited therein.
\begin{lemma}\label{lem5.9}
	Consider the compressible Euler system \eqref{1.4}
	with initial data
	\begin{equation}
		\label{5.28}
		(\bar{\rho},\bar{u},\bar{\theta})(0,x)=(1+\delta\varrho_0,\delta\varphi_0,\frac{3}{2}+\frac{3}{2}\delta\vartheta_0)(x)
	\end{equation}
	for any given $(\varrho_0,\varphi_0,\vartheta_0)(x)\in H^{k}(\mathbb{R}^{3})$
	with integer $k\geq 3$,
	where $\delta>0$ is a parameter.
	Choose a suitable constant $\delta_{1}>0$ so that 
	for any $\delta\in(0,\delta_{1}]$, the positivity of $1+\delta\varrho_0$ and $\frac{3}{2}+\frac{3}{2}\delta\vartheta_0$ is guaranteed. Then for each $\delta\in(0,\delta_{1}]$, there exists a family of classical solutions $(\bar{\rho}^{\delta},\bar{u}^{\delta},\bar{\theta}^{\delta})(t,x)\in C([0,\tau^{\delta}];H^{k})\cap
	C^{1}([0,\tau^{\delta}];H^{k-1})$
	of the compressible Euler system \eqref{1.4}
	and \eqref{5.28} such that the following things hold true:  $\bar{\rho}^{\delta}(t,x)>0$, $\bar{\theta}^{\delta}(t,x)>0$ and
	\begin{equation}\notag
		\|(\bar{\rho}^{\delta}(t,x)-1,\bar{u}^{\delta}(t,x),\bar{\theta}^{\delta}(t,x)-\frac{3}{2})\|_{C([0,\tau^{\delta}];H^{k})\cap C^{1}([0,\tau^{\delta}];H^{k-1})}\leq C_0.
	\end{equation}
	Moreover, the life-span $\tau^{\delta}$ has the following lower bound
	\begin{equation}\notag
		\tau^{\delta}>\frac{C_1}{\delta}.
	\end{equation}
	Here the positive constants $C_0$ and $C_1$ are independent of $\delta$, depending only on the $H^{k}$-norm of $(\varrho_0,\varphi_0,\vartheta_0)(x)$.
\end{lemma}
In what follows we give a refined estimate of two solutions to compressible Euler and acoustic systems. Let  $(\bar{\rho}^{\delta},\bar{u}^{\delta},\bar{\theta}^{\delta})(t,x)$ be the compressible Euler solutions
as obtained in Lemma \ref{lem5.9} and
$(\varrho,\varphi,\vartheta)(t,x)$ be the solutions of acoustic system \eqref{1.12}-\eqref{1.13}.
Then we define 
\begin{align}\notag
	\varrho^{\delta}_{d}=\frac{1}{\delta^2}(\bar{\rho}^{\delta}-1-\delta\varrho),\quad
	\varphi^{\delta}_{d}=\frac{1}{\delta^2}(\bar{u}^{\delta}-\delta\varphi),\quad
	\vartheta^{\delta}_{d}
	=\frac{2}{3}\frac{1}{\delta^2}(\bar{\theta}^{\delta}-\frac{3}{2}-\frac{3}{2}\delta\vartheta).
\end{align}
Following the same strategy in \cite[Lemma 3.2]{Guo-Jang-Jiang-2010}, 
we know that $(\varrho^{\delta}_{d},\varphi^{\delta}_{d},\vartheta^{\delta}_{d})(t)$ satisfies
\begin{align}
	\label{4.10A}
	\sup_{t\in[0,\tau]}\|(\varrho^{\delta}_{d},\varphi^{\delta}_{d},\vartheta^{\delta}_{d})(t)\|^2_{H^{k}}\leq C,
\end{align}
for $k\geq 3$, where the constant $C>0$
depends only $\tau$ and the $H^{k+1}$-norm of
$(\varrho_{0},\varphi_{0},\vartheta_{0})(x)$.

In terms of  $(\bar{\rho}^{\delta},\bar{u}^{\delta},\bar{\theta}^{\delta})(t,x)$ as obtained in Lemma \ref{lem5.9}, we
denote the local Maxwellian
$$
\overline{M}^{\delta}\equiv M_{[\bar{\rho}^{\delta},\bar{u}^{\delta},\bar{\theta}^{\delta}]}(t,x,v):=\frac{\bar{\rho}^{\delta}(t,x)}{\sqrt{[2\pi R\bar{\theta}^{\delta}(t,x)]^{3}}}\exp\big\{-\frac{|v-\bar{u}^{\delta}(t,x)|^2}{2R\bar{\theta}^{\delta}(t,x)}\big\}.
$$
In view of \eqref{4.10A}, we can choose a sufficiently small constant $\delta_{0}>0$ such that 
$(\bar{\rho}^{\delta},\bar{u}^{\delta},\bar{\theta}^{\delta})(t,x)$ with any $0<\delta\leq \delta_0$ satisfies \eqref{1.7}. 
With these facts, by using the same arguments as in Theorem \ref{thm1.1} we can prove the existence and uniqueness of smooth solutions to the Landau equation \eqref{1.1} under the assumptions in Theorem \ref{thm1.2}. The details are omitted for brevity of presentation. Therefore, similar to Theorem \ref{thm1.1}, there exists a small constant $\varepsilon_{0}>0$
such that for each $\varepsilon\in(0,\varepsilon_{0})$, the following holds:
\begin{align}
	\label{3.12}
	\sup_{t\in[0,\tau]}\|\frac{F^{\varepsilon}(t)-\overline{M}^{\delta}(t)}{\sqrt{\mu}}
	\|_{L_{x}^{2}L_{v}^{2}}+
	\sup_{t\in[0,\tau]}\|\frac{F^{\varepsilon}(t)-\overline{M}^{\delta}(t)}{\sqrt{\mu}}
	\|_{L_{x}^{\infty}L_{v}^{2}}\leq C_{\tau}\varepsilon.
\end{align}
Following the same method used in \cite[Lemma 3.3]{Guo-Jang-Jiang-2010}, it holds that
\begin{align}
	\label{3.13}
	\sup_{t\in[0,\tau]}\|\frac{\overline{M}^{\delta}(t)-\mu-\delta\sqrt{\mu}\mathbf f(t)}{\sqrt{\mu}}
	\|_{L_{x}^{2}L_{v}^{2}}+
	\sup_{t\in[0,\tau]}\|\frac{\overline{M}^{\delta}(t)-\mu-\delta\sqrt{\mu}\mathbf f(t)}{\sqrt{\mu}}\|_{L_{x}^{\infty}L_{v}^{2}}\leq C_{\tau}\delta^2,
\end{align}
where $\mathbf f(t)$ is given in \eqref{1.17}.
Hence, we can deduce from \eqref{1.15},
\eqref{3.12} and \eqref{3.13} that
\begin{align*}
	\sup_{t\in[0,\tau]}\|\mathbf{f}^{\varepsilon}(t)-\mathbf{f}(t)\|_{L_{x}^{\infty}L_{v}^{2}}
	&=\sup_{t\in[0,\tau]}\|\frac{F^{\varepsilon}(t)-\mu-\delta\sqrt{\mu}\mathbf{f}(t)}{\delta\sqrt{\mu}}\|_{L_{x}^{\infty}L_{v}^{2}}
	\nonumber\\
	&\leq \sup_{t\in[0,\tau]}\|\frac{F^{\varepsilon}(t)-\overline{M}^{\delta}(t)}{\delta\sqrt{\mu}}
	\|_{L_{x}^{\infty}L_{v}^{2}}+
	\sup_{t\in[0,\tau]}\|\frac{\overline{M}^{\delta}(t)-\mu-\delta\sqrt{\mu}\mathbf f(t)}{\delta\sqrt{\mu}}\|_{L_{x}^{\infty}L_{v}^{2}}
	\nonumber\\
	&\leq C_{\tau}(\frac{\varepsilon}{\delta}+\delta).
\end{align*}
Similar estimates also hold for $\sup_{t\in[0,\tau]}\|\mathbf{f}^{\varepsilon}(t)-\mathbf{f}(t)\|_{L_{x}^{2}L_{v}^{2}}$, and then the desired estimate \eqref{1.21} holds. We  consequently finish the proof of Theorem \ref{thm1.2}.
\qed

\section{Appendix}\label{seca.5}
In this appendix, we give the details of deriving the estimate \eqref{4.10} for completeness.

\medskip
\noindent{\it Proof of \eqref{4.10}:}
Note that $\theta=\frac{3}{2}\frac{1}{2\pi e}\rho^{\frac{2}{3}}\exp(S)$ due to \eqref{4.3}, then we have
$$
\theta_{\rho}:=\partial_{\rho}\theta=\frac{1}{2\pi e}\rho^{-\frac{1}{3}}\exp(S)=\frac{2}{3}\frac{\theta}{\rho},\quad \theta_{S}=\frac{3}{2}\frac{1}{2\pi e}\rho^{\frac{2}{3}}\exp(S)=\theta.
$$
Using this together with \eqref{4.5} and \eqref{4.6}, direct computations give
\begin{align*}
\eta_{\bar{\rho}}=-\frac{\rho}{\bar{\rho}}\bar{\theta}(S-\bar{S})-\frac{5}{3\bar{\rho}}\bar{\theta}(\rho-\bar{\rho}),\ 
\quad \eta_{\bar{u}}=-\frac{3}{2}\rho(u-\bar{u}),\ 
\eta_{\bar{S}}=-\frac{3}{2}\rho\bar{\theta}(S-\bar{S})-\bar{\theta}(\rho-\bar{\rho}).
\end{align*}
Similarly, it also holds that
\begin{align*}
q_{\bar{\rho}}&=-u\frac{\rho}{\bar{\rho}}\bar{\theta}(S-\bar{S})-\frac{5u}{3\bar{\rho}}\bar{\theta}(\rho-\bar{\rho})-\frac{5}{3}\bar{\theta}(u-\bar{u}),
\quad
q_{j\bar{u}_{j}}=-\frac{3}{2}u_{j}\rho(u_{j}-\bar{u}_{j})-\rho\theta+\bar{\rho}\bar{\theta},
\\
q_{\bar{S}}&=-\frac{3}{2}u\rho\bar{\theta}(S-\bar{S})+\bar{u}\bar{\theta}\bar{\rho}-u\rho\bar{\theta}.
\end{align*}
On the other hand, one gets from \eqref{1.4} that
\begin{equation*}
\begin{cases}
\partial_t\bar{\rho}=-\bar{u}\cdot\nabla_x\bar{\rho}-\bar{\rho}\nabla_x\cdot\bar{u},
\\
\partial_t\bar{u}=-\bar{u}\cdot\nabla_x\bar{u}-\frac{2}{3}\frac{\bar{\theta}}{\bar{\rho}}\nabla_x\bar{\rho}-\frac{2}{3}\nabla_x\bar{\theta},
\\
\partial_t\bar{\theta}=-\bar{u}\cdot\nabla_x\bar{\theta}-\frac{2}{3}\bar{\theta}\nabla_x\cdot\bar{u}.
\end{cases}
\end{equation*}
In view of these facts, we have from straightforward computations that
\begin{align}
\label{5.31}
&\nabla_{[\bar{\rho},\bar{u},\bar{S}]}\eta(t,x)\cdot \partial_t(\bar{\rho},\bar{u},\bar{S})=\eta_{\bar{\rho}}\partial_t\bar{\rho}+\eta_{\bar{u}}\partial_t\bar{u}+\eta_{\bar{S}}\partial_t\bar{S}
\nonumber\\
=&\frac{\bar{\theta}}{\bar{\rho}}(\bar{\rho}-\rho)\partial_t\bar{\rho}-\frac{3}{2}\rho(u-\bar{u})\cdot\partial_t\bar{u}-\{\frac{3}{2}\rho(S-\bar{S})+(\rho-\bar{\rho})\}\partial_t\bar{\theta}
\nonumber\\
=&\frac{5}{3}(\rho-\bar{\rho})\bar{\theta}\nabla_x\cdot\bar{u}	+\frac{\bar{\theta}}{\bar{\rho}}\rho u\cdot\nabla_x\bar{\rho}+\rho u\cdot\nabla_x\bar{\theta}-\bar{u}\bar{\theta}\cdot\nabla_x\bar{\rho}
\nonumber\\
&+\frac{3}{2}\rho(u-\bar{u})\cdot(\bar{u}\cdot\nabla_x\bar{u})-\bar{\rho}\bar{u}\cdot\nabla_x\bar{\theta}
-\frac{3}{2}\rho(\bar{S}-S)(\bar{u}\cdot\nabla_x\bar{\theta}+\frac{2}{3}\bar{\theta}\nabla_x\cdot\bar{u}).
\end{align}
Likewise it also holds that
\begin{align}
\label{5.32}
\sum^{3}_{j=1}&\nabla_{[\bar{\rho},\bar{u},\bar{S}]}q_{j}(t,x)\cdot \partial_{x_j}(\bar{\rho},\bar{u},\bar{S})
=q_{\bar{\rho}}\cdot\nabla_x\bar{\rho}+\sum^{3}_{j=1}q_{j\bar{u}_{j}}\partial_{x_j}\bar{u}_{j}+q_{\bar{S}}\cdot\nabla_x\bar{S}
\nonumber\\
=&\bar{\rho}\bar{\theta}\nabla_x\cdot\bar{u}-\rho\theta\nabla_x\cdot\bar{u}-\frac{\bar{\theta}}{\bar{\rho}}\rho u\cdot\nabla_x\bar{\rho}
-\rho u\cdot\nabla_x\bar{\theta}+\bar{\theta}\bar{u}\cdot\nabla_x\bar{\rho}
\nonumber\\
&-\frac{3}{2}\rho(u-\bar{u})\cdot(u\cdot\nabla_x\bar{u})+\bar{\rho}\bar{u}\cdot\nabla_x\bar{\theta}
+\frac{3}{2}\rho u(\bar{S}-S)\nabla_x\bar{\theta}.
\end{align}
Combining \eqref{5.31} and \eqref{5.32} yields
\begin{align}
\label{5.33}
&\nabla_{[\bar{\rho},\bar{u},\bar{S}]}\eta(t,x)\cdot \partial_{t}(\bar{\rho},\bar{u},\bar{S})
+\sum^{3}_{j=1}\nabla_{[\bar{\rho},\bar{u},\bar{S}]}q_{j}(t,x)\cdot \partial_{x_j}(\bar{\rho},\bar{u},\bar{S})
\nonumber\\
=&\frac{5}{3}\rho\bar{\theta}\nabla_x\cdot\bar{u}-\frac{2}{3}\bar{\rho}\bar{\theta}\nabla_x\cdot\bar{u}-\frac{3}{2}\rho(u-\bar{u})\widetilde{u}\cdot\nabla_x\bar{u}
\nonumber\\
&-\frac{3}{2}\rho(\bar{S}-S)(\bar{u}\cdot\nabla_x\bar{\theta}+\frac{2}{3}\bar{\theta}\nabla_x\cdot\bar{u})+\frac{3}{2}\rho u(\bar{S}-S)\nabla_x\bar{\theta}-\rho\theta\nabla_x\cdot\bar{u}
\nonumber\\
=&-\frac{3}{2}\rho\widetilde{u}\cdot(\widetilde{u}\cdot\nabla_x\bar{u})
-\frac{2}{3}\rho\bar{\theta}\nabla_x\cdot\bar{u}\Psi(\frac{\bar{\rho}}{\rho})
-\rho\bar{\theta}\nabla_x\cdot\bar{u}\Psi(\frac{\theta}{\bar{\theta}})-\frac{3}{2}\rho\nabla_x\bar{\theta}\cdot\widetilde{u}(\frac{2}{3}\ln\frac{\bar{\rho}}{\rho}+\ln\frac{\theta}{\bar{\theta}}),
\end{align}
where we have used $\Psi(s)=s-\ln s-1$ and 
$
\bar{S}-S=-\frac{2}{3}\ln\frac{\bar{\rho}}{\rho}
-\ln\frac{\theta}{\bar{\theta}}.
$
Hence the desired estimate \eqref{4.10}
follows from \eqref{5.33} and then the proof  is completed.
\qed

\medskip

\noindent {\bf Acknowledgment:}\,
The authors would like to thank the anonymous referees for all valuable and helpful comments on the manuscript. The research of Renjun Duan was partially supported by the General Research Fund (Project No.~14301719) from RGC of Hong Kong and a Direct Grant from CUHK. The research of Hongjun Yu was supported by the GDUPS 2017 and the NNSFC Grant 11371151. Dongcheng Yang would like to thank Department of Mathematics, CUHK  for hosting his visit in the period 2020-2022.

\medskip

\noindent{\bf Conflict of Interest:} The authors declare that they have no conflict of interest.


\end{document}